\documentclass[11pt]{amsart}

\usepackage[a4paper,hmargin=3cm,vmargin=3cm]{geometry}
\usepackage{amsfonts,amssymb,amscd,amstext}
\usepackage{graphicx}
\usepackage[dvips]{epsfig}
\usepackage[ansinew]{inputenc}

\usepackage{fancyhdr}
\usepackage{times}

\usepackage{enumerate}
\usepackage{titlesec}
\usepackage{mathrsfs}
\usepackage{stmaryrd}

\def\R{\mathbb{R}}

\def\C{\mathbb{C}}

\def\M{\mathbb{M}}
\def\D{\mathbb{D}}

\newcommand{\ben}{\begin{enumerate}}
\newcommand{\bit}{\begin{itemize}}
\newcommand{\een}{\end{enumerate}}
\newcommand{\eit}{\end{itemize}}

\newcommand{\ed}{\end{document}}

\def\cU{\mathcal{U}}

\def\cS{\mathcal{S}}

\def\cR{\mathcal{R}}

\def\cW{\mathcal{W}}
\def\cS{\mathcal{S}}
\def\cV{\mathcal{V}}

\def\cL{\mathcal{L}}

\def\cG{\mathcal{G}}
\def\cF{\mathcal{F}}

\let\8=\infty \let\0=\emptyset  
\let\hat=\widehat

\let\landa=\lambda
\let\alfa=\alpha

\let\parc=\partial

\def\ep{\varepsilon}

\def\landa{\lambda}

\def\flecha{\rightarrow}
\def\esiz{\langle}
\def\esde{\rangle}

\def\cte.{\mathop{\rm cte.}\nolimits}

\def\cosh{\mathop{\rm cosh }\nolimits}
\def\tanh{\mathop{\rm tanh }\nolimits}

\def\R{\mathbb{R}}

\def\C{\mathbb{C}}

\def\D{\mathbb{D}}

\def\H{\mathbb{H}}
\def\S{\mathbb{S}}

\def\Ek{\mathbb{E}^3 (\kappa,\tau)}
\def\Rk{\mathcal{R}^3(\kappa,\tau)}

\def\Ek{\mathbb{E}^3 (\kappa,\tau)}
\def\Rk{\mathcal{R}^3(\kappa,\tau)}

\titleformat{\section}
{\filcenter\bfseries\large} {\thesection{.}}{0.2cm}{}
\titleformat{\subsection}[runin]
{\bfseries} {\thesubsection{.}}{0.15cm}{}[.]
\titleformat{\subsubsection}[runin]
{\em}{\thesubsubsection{.}}{0.15cm}{}[.]

\usepackage[up,bf]{caption}

\newtheorem{theorem}{Theorem}[section]
\newtheorem{lemma}[theorem]{Lemma}
\newtheorem{proposition}[theorem]{Proposition}

\newtheorem{remark}[theorem]{Remark}
\newtheorem{corollary}[theorem]{Corollary}
\newtheorem{definition}[theorem]{Definition}

\newtheorem{example}[theorem]{Example}

\theoremstyle{definition}

\numberwithin{equation}{section}
\numberwithin{figure}{section}

\usepackage{color}

\begin{document}
\fancyhead[LO]{Rotational symmetry of Weingarten spheres}
\fancyhead[RE]{José A. Gálvez, Pablo Mira}
\fancyhead[RO,LE]{\thepage}

\thispagestyle{empty}

\begin{center}
{\bf \LARGE Rotational symmetry of Weingarten spheres\\[0.2cm] in homogeneous three-manifolds}
\vspace*{5mm}

\hspace{0.2cm} {\Large José A. Gálvez, Pablo Mira}
\end{center}

\footnote[0]{\vspace*{-0.4cm} \emph{Mathematics Subject Classification}: Primary 53A10; Secondary 49Q05, 53C42}

\vspace*{7mm}

\begin{quote}
{\small
\noindent {\bf Abstract}\hspace*{0.1cm}
Let $M$ be a simply connected homogeneous three-manifold with isometry group of dimension $4$, and let $\Sigma$ be any compact surface of genus zero immersed in $M$ whose mean, extrinsic and Gauss curvatures satisfy a smooth elliptic relation $\Phi(H,K_e,K)=0$. In this paper we prove that $\Sigma$ is a sphere of revolution, provided that the unique inextendible rotational surface $S$ in $M$ that satisfies this equation and touches its rotation axis orthogonally has bounded second fundamental form. 
In particular, we prove that: (i) any elliptic Weingarten sphere immersed in $\H^2\times \R$ is a rotational sphere. (ii) Any sphere of constant positive extrinsic curvature immersed in $M$ is a rotational sphere, and (iii) Any immersed sphere in $M$ that satisfies an elliptic Weingarten equation $H=\phi(H^2-K_e)\geq a>0$ with $\phi$ bounded, is a rotational sphere. As a very particular case of this last result, we recover the Abresch-Rosenberg classification of constant mean curvature spheres in $M$. 

\vspace*{0.1cm}

}
\end{quote}


\section{Introduction}

An immersed surface $\Sigma$ in a Riemannian three-manifold $M$ is an \emph{elliptic Weingarten surface} if its mean curvature $H=(\kappa_1+\kappa_2)/2$ and its extrinsic curvature $K_e=\kappa_1 \kappa_2$ satisfy a smooth elliptic relation $W(H,K_e)=0$. This equation can be rewritten as
\begin{equation}\label{speci1}
H=\phi(H^2-K_e), \hspace{1cm} \phi(t)\in C^{\8}([0,\8)), \hspace{0.5cm} 4t (\phi'(t))^2<1,
\end{equation} 
where the inequality describes the ellipticity of the equation. Elliptic Weingarten surfaces are also called \emph{special Weingarten surfaces} in many works; see e.g. \cite{Ch,HW1,RS,ST}. It follows from classical works by Hopf \cite{Ho0}, Chern \cite{Ch} and Hartman and Wintner \cite{HW1} that \emph{any compact elliptic Weingarten surface of genus zero immersed in $\R^3,\S^3$ or $\H^3$ is a round sphere. }

A simply connected Riemannian homogeneous three-manifold $M$ has isometry group of dimension $n=6$, $4$ or $3$. If $n=6$, $M$ is a space form $\R^3$, $\S^3(c)$ or $\H^3(c)$. If $n=3$, $M$ is a Lie group with a left invariant metric and a discrete isotropy group (see \cite{mpe11}). If $n=4$, $M$ admits a nice geometric structure: these spaces are rotationally invariant, and can be seen as Riemannian fibrations over two-dimensional spaces of constant curvature (see Section \ref{sec2} for details). They include, in particular, the product spaces $\H^2\times \R$, $\S^2\times \R$, and can be viewed as the most symmetric Riemannian three-manifolds other than space forms. They are usually called $\Ek$ spaces.

A deeply influential result in the theory of constant mean curvature (CMC) surfaces in homogeneous manifolds is the solution by Abresch and Rosenberg of the so-called Hopf uniqueness problem for \emph{immersed spheres} (i.e. compact surfaces of genus zero) of constant mean curvature in these rotationally symmetric homogeneous three-manifolds:

\begin{theorem}[\cite{AR1,AR2}]\label{arth}
Any CMC sphere immersed in $M=\Ek$ is a rotational sphere.
\end{theorem}

Motivated by these results, the following is an important, well-known, open problem in the theory of surfaces in homogeneous manifolds: 

\vspace{0.2cm}

{\bf Problem (P):} \emph{Let $\Sigma$ be an elliptic Weingarten sphere immersed in $M=\Ek$. Is then $\Sigma$ a rotational sphere?}

\vspace{0.2cm}

In this paper we develop a systematic approach to study Problem (P). Our main general result, Theorem \ref{mainth}, provides a simple sufficient condition for the rotational symmetry of immersed spheres in $\Ek$ that satisfy very general curvature equations, and in particular, the elliptic Weingarten equation \eqref{speci1}. We next give some consequences of this Theorem \ref{mainth} regarding Problem (P). The first one, which follows directly from Theorem \ref{mainth}, is:

\begin{theorem}\label{tint0}
Assume that there exists a rotational sphere $S$ in $M=\Ek$ that satisfies the elliptic Weingarten equation \eqref{speci1}. Then, any immersed sphere in $M$ that satisfies \eqref{speci1} is equal to $S$, up to ambient isometry.

More generally: let $S$ denote the unique (up to ambient isometry) inextendible rotational surface in $M$ that touches its rotation axis orthogonally, and satisfies \eqref{speci1}. If $S$ has bounded second fundamental form, then any immersed sphere in $M$ that satisfies \eqref{speci1} is a rotational sphere.
\end{theorem}

The inextendible rotational surface $S$ will be proved to exist in Section \ref{sec:app}. It may or may not be a sphere. When the ambient space $M$ is a space form, $S$ is a totally umbilical surface, and thus has bounded second fundamental form. However, when $M=\Ek$, the surface $S$ is not totally umbilical in general, and for some Weingarten functionals \eqref{speci1} it is non-complete and its second fundamental form blows up. Theorem \ref{tint0} solves Problem (P) except for those cases. Some specific consequences of Theorem \ref{tint0} are:

\begin{theorem}\label{tint1}
Any elliptic Weingarten sphere immersed in $\H^2\times\R$ is an embedded rotational sphere.
\end{theorem}

\begin{theorem}\label{tint2}
For every $c>0$ there exists a rotational sphere $S_c$ immersed in $M=\Ek$ with constant extrinsic curvature  $K_e=c$, and any other immersed sphere in $M$ with $K_e=c$ is equal to $S_c$, up to ambient isometry. Moreover, $S_c$ is embedded if $M$ is not compact. If $M$ is compact, $S_c$ maybe be non-embedded, see Figure \ref{figesf}.
\end{theorem}

\begin{theorem}\label{tint3}
Let $\Sigma$ be an immersed sphere in $M=\Ek$ that satisfies an elliptic Weingarten equation \eqref{speci1}, where $\phi\in C^{\8}([0,\8))$ verifies $a<\phi<b$ for positive constants $a,b$.

Then, $\Sigma$ is a rotational sphere in $M$. $\Sigma$ is embedded if $M$ is not compact.

\end{theorem}

For the proof of these theorems, see Theorems \ref{clash2}, \ref{posex} and \ref{bomi}, respectively.

Theorem \ref{tint1} provides a solution to the classification problem of elliptic Weingarten spheres, as well as a positive answer to Problem (P), when the ambient space is $\H^2\times \R$. Theorem \ref{tint2} classifies all immersed spheres in $M=\Ek$ with constant positive extrinsic curvature, thus solving a well-known open problem of the theory. For the case where $M$ is $\H^2\times \R$ or $\S^2\times \R$, Theorem \ref{tint2} was proved in \cite{EGR}. Theorem \ref{tint3} contains, as a very particular situation, the Abresch-Rosenberg classification of CMC spheres.

\begin{figure}
\begin{center}
\includegraphics[height=5.5cm]{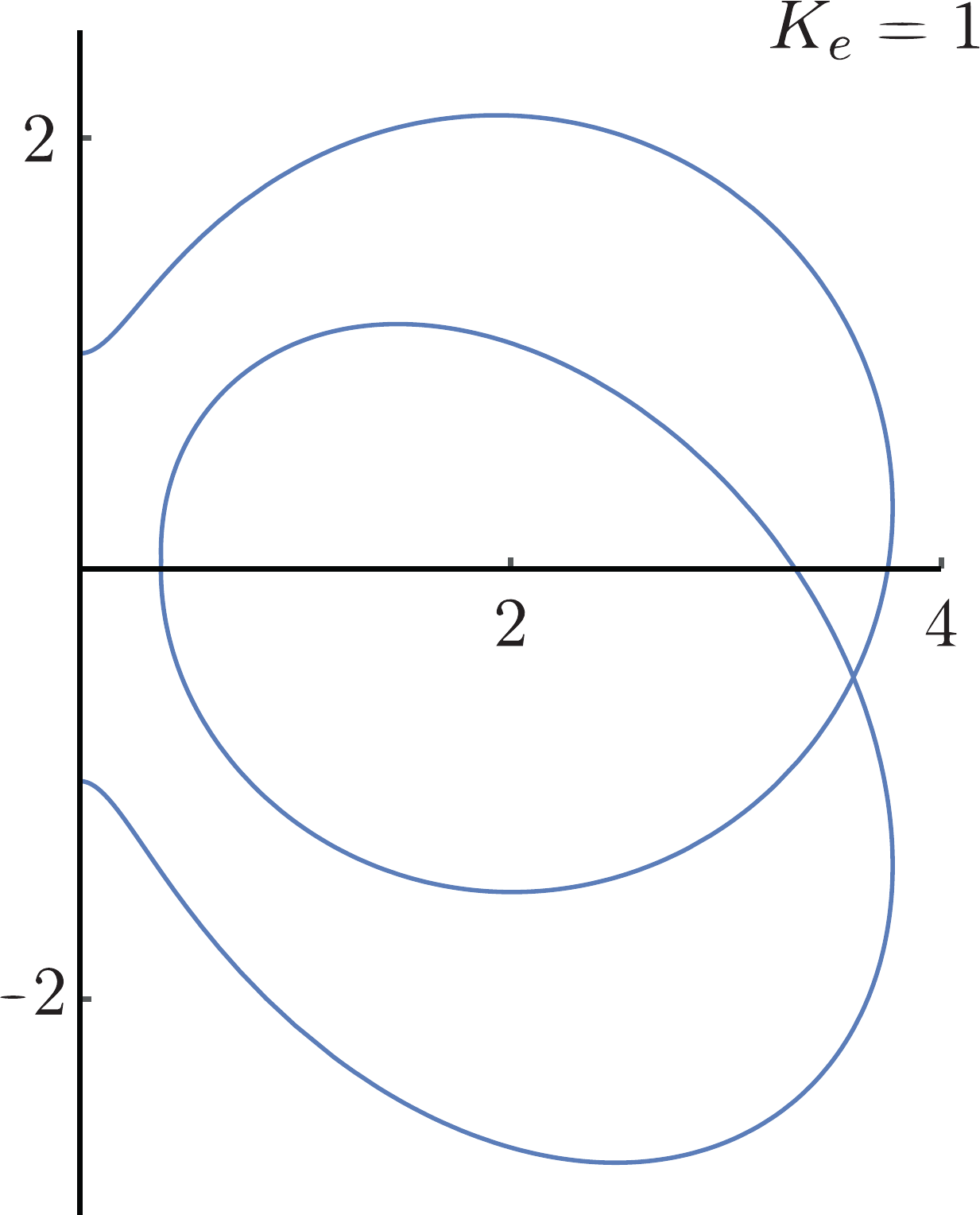}  \hspace{1cm}
\includegraphics[height=5.5cm]{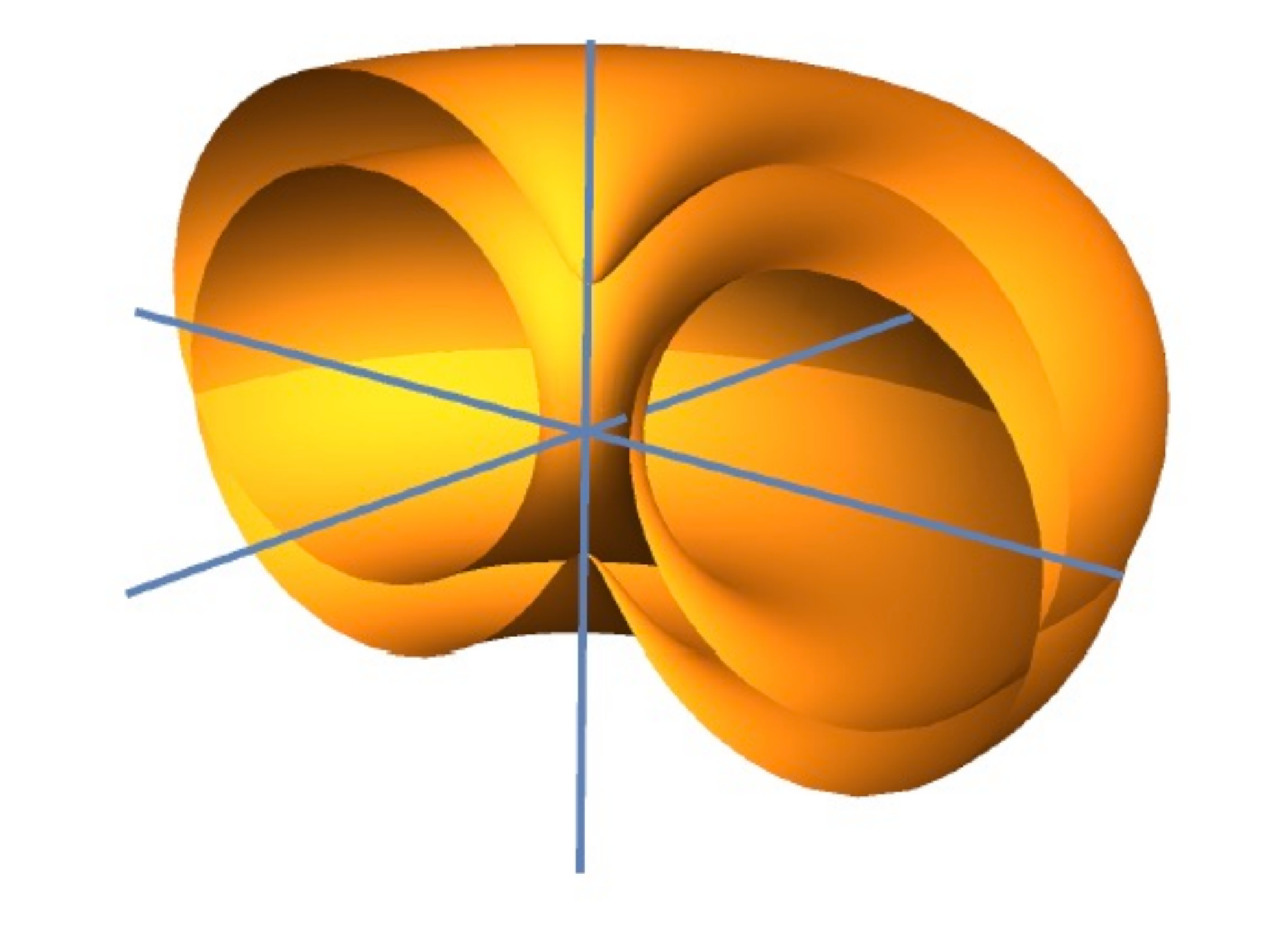}
\caption{Left: profile curve of a rotational, non-embedded, sphere with $K_e=1$ when $M=\Ek$ is a certain Berger sphere, after stereographic projection $M\setminus\{p\}\flecha \R^3$. Right: Picture of one half of this rotational sphere.}\label{figesf}
\end{center}
\end{figure}

We will next explain in detail our main general uniqueness result, valid for elliptic curvature equations more general than \eqref{speci1}. Let $M=\Ek$ be a simply connected homogeneous three-manifold with a $4$-dimensional isometry group, and let $\xi$ denote the vertical unit Killing vector field on $M$ associated to the canonical fibration $\pi:M\flecha \M^2(\kappa)$ of $M$ onto the two-dimensional space form of constant curvature $\kappa$. We will say that an immersed oriented surface $\Sigma$ in $M$ is a \emph{general Weingarten surface} if its mean, extrinsic and Gaussian curvatures $H,K_e,K$ satisfy a smooth elliptic relation $$\cF(H,K_e,K)=0, \hspace{1cm} \cF\in C^{\8} (\R^3).$$ This equation can be rewritten in a more adequate way (see Section \ref{sec:wde}) as 
 \begin{equation}\label{gwi}
 H=\Phi(H^2-K_e, \nu^2),
 \end{equation}
where $\nu:=\esiz \eta,\xi\esde$ is the \emph{angle function} of $\Sigma$ (here, $\eta$ is the unit normal of $\Sigma$ in $M$), and $\Phi=\Phi(t,v) \in C^{\8}([0,\8)\times [0,1])$ satisfies the ellipticity condition 
 \begin{equation}\label{gwi2}
 4 t \left( \frac{\parc \Phi}{\parc t} (t,v)\right)^2 <1 \hspace{1cm} \forall (t,v).
 \end{equation}
When $\Phi(t,v)$ only depends on $t$, we recover the equation \eqref{speci1} of elliptic Weingarten surfaces. The class of all immersed surfaces in $M$ that satisfy a specific general Weingarten equation \eqref{gwi}-\eqref{gwi2} is called a \emph{general Weingarten class}, and denoted by $\cW$.

%
%
%
%
%

As already mentioned for the case of elliptic Weingarten surfaces, if $\cW$ is a general Weingarten class in $M$, one can show (see Section \ref{sec:app})  that there exists an inextendible rotational surface $S$ of $\cW$, unique up to ambient isometry, that touches its rotation axis orthogonally at some point. We call $S$ the \emph{canonical rotational example} of the class $\cW$.



In these conditions, we will prove:

\begin{theorem}\label{mainth}
Let $\cW$ be a class of general Weingarten surfaces in $M=\Ek$, and let $S$ be the canonical rotational example of $\cW$. Assume that $S$ has bounded second fundamental form.

Let $\Sigma$ be an immersed sphere of the class $\cW$. Then:

 \begin{enumerate}
 \item
$\Sigma$ is congruent to $S$; in particular, $\Sigma$ is a sphere of revolution in $M$.
 \item
If $S$ is not compact, $\Sigma$ does not exist, i.e. there are no immersed spheres in the class $\cW$.
 \item
If $M$ is not compact, then $\Sigma$ is embedded. 
 \item
$\Sigma$ is a bigraph, unless it is a totally geodesic slice $\S^2(\kappa)\times \{t_0\}$ in $M=\S^2(\kappa)\times \R$.
 \item
$\Sigma$ is invariant with respect to the $180º$-rotation of $M$ around some geodesic orthogonal to its rotation axis.
 \item
The angle function of $\Sigma$ is a strictly monotonic function of any regular parameter of its profile curve, unless $\Sigma$ is a totally geodesic slice $\S^2(\kappa)\times \{t_0\}$ in $M=\S^2(\kappa)\times \R$.
 \end{enumerate}

\end{theorem}

Let us make some brief comments regarding the statement of Theorem \ref{mainth}. 

The first assertion shows the rotational symmetry of \emph{any} general Weingarten sphere in $M$, under the assumption that $S$ has bounded second fundamental form. Theorem \ref{tint0} is an immediate consequence of this item. In particular, given a general Weingarten class $\cW$ for which $S$ has bounded second fundamental form, there exists at most one immersed sphere in $\cW$ up to ambient isometry. 

The second assertion covers, for instance, the cases $H=0$ in $\R^3$ or $H=1/2$ in $\H^2\times \R$; in these cases the canonical rotational examples are, respectively, planes and certain entire graphs, and no compact examples exist (by the maximum principle). It should be emphasized however that, in our general situation, the fact that the canonical example $S$ is not compact does not contradict the existence of compact immersed surfaces in the class $\cW$, since $S$ might not be an entire graph. In this sense, assertion (2) is non-trivial, and the fact that $\Sigma$ has genus zero is fundamental to it.

The third assertion proves that rotational spheres of a general Weingarten class in $M=\Ek$ are embedded whenever $M$ is not diffeomorphic to $\S^3$. This is not true without the topological hypothesis; for example, there exist non-embedded rotational CMC spheres for some $M=\Ek$ diffeomorphic to $\S^3$ (see \cite{To}). 

The \emph{bigraph} property in the fourth assertion means: the sphere $\Sigma$ is decomposed as $\Sigma=\Sigma_1\cup \Sigma_2$, where $\Sigma_1,\Sigma_2$ are compact, embedded rotational disks in $M$ with $\parc \Sigma_1=\parc \Sigma_2$, and so that, if $\pi:\Ek\flecha \mathbb{M}^2(\kappa)$ is the canonical fibration of $M=\Ek$, then there exists a domain $\mathcal{D}\subset \mathbb{M}^2(\kappa)$ such that for each $i=1,2$, $\pi({\rm int}(\Sigma_i))=\mathcal{D}$  and $\pi |_{{\rm int}(\Sigma_i)}$ defines a diffeomorphism between the interior of $\Sigma_i$ and $\mathcal{D}$. In the case that $M$ is diffeomorphic to $\S^3$, the interiors of $\Sigma_1$ and $\Sigma_2$ can intersect.

As regards the fifth assertion, let us observe that all these $180º$-rotations with respect to geodesics orthogonal to the rotation axis are isometries for every ambient manifold $M=\Ek$; the theorem shows that $\Sigma$ inherits such symmetries. Equivalently, in the decomposition $\Sigma=\Sigma_1\cup \Sigma_2$ of $\Sigma$ as a bigraph, one of these $180º$-rotations leaves $\parc \Sigma_1=\parc \Sigma_2$ invariant and sends $\Sigma_1$ to $\Sigma_2$ and vice versa.

The sixth assertion proves that $\Sigma$ is convex in a certain sense (it is not true in general that $\Sigma$ is convex in the sense that its extrinsic curvature is everywhere positive). Monotonicity of the angle function here means that two points of $\Sigma$ with the same angle function must lie in the same meridian curve of $\Sigma$. When $M$ is not diffeomorphic to $\S^2\times \R$ (and thus it can be given the structure of a Lie group with a left invariant metric, see \cite{mpe11}), this property is equivalent for rotational surfaces to the fact that the left invariant Gauss map of $\Sigma$, which takes values in $\S^2$, is injective. For a detailed discussion about surfaces in Lie groups and the regularity of their left invariant Gauss maps, see \cite{MMP}.

If we considered as ambient space $M$ the Euclidean flat $3$-space $\R^3$ or the $3$-sphere $\S^3(c)$ of constant curvature $c>0$, an analogous result to Theorem \ref{mainth} can be proved, with basically the same arguments. See Appendix 1. 

It is worth noting that the Abresch-Rosenberg classification of CMC spheres in $\Ek$-spaces (Theorem \ref{arth}) has been recently extended to any Riemannian homogeneous three-manifold. Specifically, in \cite{mmpr0,mmpr1} it is proved that two immersed spheres with the same constant mean curvature in \emph{any} homogeneous three-manifold $M$ always coincide up to an ambient isometry, and they inherit all the symmetries of the ambient space $M$. For the special case where $M$ is the three-dimensional Thurston geometry given by the Lie group ${\rm Sol}_3$, see \cite{DM,M}.

The proof of Theorem \ref{mainth} is based on the Poincaré-Hopf theorem. More specifically, if $\Sigma$ is an immersed sphere of a general Weingarten class $\cW$ in $\Ek$, and $\Sigma$ is not a rotational sphere, the idea is to construct a line field $\cL$ on $\Sigma$ with isolated zeros of negative index. As there are no such line fields on a topological sphere (by Poincaré-Hopf), $\Sigma$ cannot exist.

Thus, the main difficulty is to construct a line field with these properties. For that, we will use the previous work by the authors \cite{GM3}, which gives a sufficient condition (the existence of a \emph{transitive family of solutions}, see Definition \ref{defitransi}) for the existence of such line field, in the general context of classes of immersed surfaces modelled by elliptic PDEs in Riemannian three-manifolds. We next give a brief sketch of the organization of the paper.

Sections \ref{sec2} and  \ref{subsec:geo3} present preliminary material needed for our study. In Section \ref{sec2} we review basic properties of $\Ek$ spaces and of rotational surfaces in $\Ek$. In Section \ref{subsec:geo3} we analyze in detail the Weingarten equations \eqref{speci1} and \eqref{gwi}, and we study their associated fully nonlinear elliptic PDEs with respect to local coordinates in the ambient space.

In Section \ref{sec:canonical} we show that on any general Weingarten class of surfaces $\cW$ in $\Ek$ there is a unique (up to ambient isometry) inextendible rotational surface $S$ of $\cW$ that touches its rotation axis orthogonally at some point. We call such surface $S$ the \emph{canonical rotational example} of the class $\cW$ and describe its geometry. Of particular importance is the fact that the angle function of $S$ is a strictly monotonic function in terms of any regular parametrization of the profile curve of $S$.

In Section \ref{secepru} we prove Theorem \ref{mainth}, as we explain next. In Section \ref{sub:ext} we deal with extension properties of the canonical example $S$, through ODE analysis. 
In Section \ref{sec:r3} we prove Theorem \ref{mainth} for the case that $M=\Ek$ is diffeomorphic to $\R^3$. In this situation we show that either: (i) $S$ is an entire graph, and so no compact examples exist in $\cW$, or (ii) the family composed by all the oriented surfaces in $M$ that are congruent to $S$, or that can be expressed as a limit example of surfaces in $M$ congruent to $S$, is a \emph{transitive family of surfaces}, i.e. a family of oriented surfaces whose Legendrian lifts to the unit tangent bundle $TU(M)$ actually foliate $TU(M)$. In these conditions we can apply our results in \cite{GM3} to deduce that any immersed sphere $\Sigma$ in the class $\cW$ is congruent to $S$. The rest of the properties stated in Theorem \ref{mainth} follow from our analysis in Sections \ref{sec:canonical} and \ref{sub:ext}.

In Sections \ref{sec:s2xr} and \ref{sec:s3} we prove Theorem \ref{mainth} when $M=\Ek$ is not diffeomorphic to $\R^3$ (thus, it is diffeomorphic to $\S^2\times \R$ or to $\S^3$). This time, due to the change of topology, the main task is to describe in detail the behavior of the canonical example $S$ when it approaches the \emph{antipodal fiber} of its rotation axis in $M$.

In Sections \ref{sec6}, \ref{sec7} and \ref{sec8} we give natural, general choices of classes $\cW$ for which the canonical example $S$ has bounded second fundamental form, and hence Theorem \ref{mainth} applies. In particular, we prove that this is the case for the following classes in $M=\Ek$:

\begin{enumerate}
\item
$H=\Phi(H^2-K,\nu^2)$ with $a<\Phi<b$, being $a,b>0$. (Section \ref{sec6}).
 \item
$K_e=c>0$, or more generally, $K_e=\Phi(\nu^2)>0$. (Section \ref{sec7}).
 \item
\emph{Any} elliptic Weingarten equation $H=\phi(H^2-K)$ in $M=\H^2\times \R$. (Section \ref{sec8}).
\end{enumerate}
In Section \ref{sec8} we will also show examples of general Weingarten classes in $\H^2\times \R$, and of elliptic Weingarten classes in $\S^2\times \R$, where the canonical example $S$ is non-complete and has unbounded second fundamental form.

The paper closes with two appendices. In the first one we indicate that some of our main theorems also hold (and were not previously known) when the ambient space $M$ is the Euclidean space $\R^3$, or the constant curvature sphere $\S^3(c)$, $c>0$. In the second appendix we review in more detail the geometry of Berger spheres, i.e. of the $\Ek$-spaces diffeomorphic to $\S^3$.

The authors are grateful to Joaquín Pérez and Francisco Torralbo for useful comments and discussions.

\section{Homogeneous manifolds and transitive families of surfaces}\label{sec2}

\subsection{Description of $\Ek$ spaces}\label{descr}

Let $M$ be a simply connected homogeneous three-manifold with ${\rm dim} ({\rm Iso}(M))>3$, and assume that $M$ is not a space form. Then ${\rm dim} ({\rm Iso}(M))=4$ and $M$ is an $\Ek$ space for some $(\kappa,\tau)\in \R^2$ with $\kappa \neq 4\tau^2$ (see \cite{D}); we next make a quick review of some standard aspects of the theory.

Every $\Ek$ space admits a canonical Riemannian fibration $\pi:\Ek\flecha \mathbb{M}^2(\kappa)$ over the simply connected $2$-dimensional space $\mathbb{M}^2(\kappa)$ of constant curvature $\kappa$. After choosing orientations, we can define the vertical unit field $\xi$ associated to this fibration; it is a Killing field on $\Ek$. If $\tau=0$ we get the Riemannian product spaces $\H^2(\kappa)\times \R$ and $\S^2(\kappa)\times \R$. When $\tau\neq 0$ we get the Riemannian Heisenberg space $\rm{Nil}_3$ for $\kappa=0$, Berger spheres for $\kappa>0$, and rotationally symmetric left invariant metrics on the universal cover of ${\rm PSL}(2,\R)$ if $\kappa <0$.

If $\kappa>0$, the space $\Ek$ is homeomorphic to $\S^3$ when $\tau\neq 0$, and to $\S^2\times \R$ when $\tau=0$. If $\kappa\leq 0$, the space $\Ek$ is homeomorphic to $\R^3$. 

All spaces $\Ek$ are rotationally invariant, i.e. for each $p\in \Ek$ there exists a continuous $1$-parameter family of orientation preserving isometries of $\Ek$ that leave pointwise fixed the fiber $\pi^{-1}(\pi(p))$.

It is convenient to use what we will call the \emph{coordinate model} $\Rk$ for $\Ek$. Specifically, we define $\Rk$ to be
$\R^3$ if $\kappa\geq 0$, or
$\D\left(2/\sqrt{-\kappa}\right)\times\R$, if $\kappa<0$,
where $\D(\rho)=\{(x_1,x_2)\in\R^2\,;\,x_1^2+x_2^2<\rho^2\}$,
endowed in any case with the Riemannian metric
\begin{equation}\label{eq:metric}
ds^2 = \lambda^2 (dx_1^2+ dx_2^2)+\big(\tau\lambda (x_2dx_1 -
x_1dx_2) + dx_3\big)^2, \hspace{0.5cm}
\lambda=\frac{1}{1+\frac{\kappa}{4}(x_1^2+x_2^2)}.
\end{equation} 
If $\kappa\leq 0$, the space $(\Rk,ds^2)$ is globally isometric to $\Ek$. When $\kappa>0$, $(\Rk,ds^2)$ is isometric to $(\S^2(\kappa)\setminus\{p\}) \times \R$ if $\tau=0$, and to the Riemannian universal cover of the Berger sphere $\Ek$ minus one fiber if $\tau \neq 0$; see Appendix 2 for a more definite description of these coordinates when $\Ek$ is a Berger sphere. When $\kappa>0$, the space $\Ek$ can be covered by two charts of coordinate models.

In view of \eqref{eq:metric}, the following maps are orientation preserving isometries of $\Rk$ which are independent of the value of $(\kappa,\tau)$, and that give rise to global isometries of $\Ek$:

\begin{enumerate}
\item
Vertical translations $(x_1,x_2,x_3)\mapsto (x_1,x_2,x_3+c)$, for every $c\in \R$.
 \item
Rotations around the $x_3$-axis: $(x_1,x_2,x_3)\mapsto (x_1 \cos \theta + x_2\sin \theta, -x_1\sin \theta + x_2 \cos \theta, x_3)$, for every $\theta\in [0,2\pi)$.
 \item
Rotations of angle $\pi$ around horizontal lines passing through the origin; for instance, $(x_1,x_2,x_3)\mapsto (x_1,-x_2,-x_3)$.
\end{enumerate}

The vertical unit Killing field $\xi$ of $\Ek$ in this model is $\xi=\frac{\parc}{\parc x_3}$; it is the Killing field associated with the isometries of $\Ek$ given by vertical translations. We observe that, in this model, the projection $\pi:\Ek\flecha \mathbb{M}^2(\kappa)$ is represented by the projection $(x_1,x_2,x_3)\mapsto (x_1,x_2)$.

\begin{remark}\label{ekte}
\emph{If we allow the possibility $\kappa=4\tau^2$ for the values of $(\kappa,\tau)$, the model above also recovers the Euclidean three-space $(\kappa=\tau=0)$ and the three-sphere $\S^3(c)$ for $\kappa=4\tau^2=4c>0$, with one caveat: in this $\Ek$ model of $\R^3$ or $\S^3(c)$ we are prescribing a particular unit Killing field of the space as the vertical one. }


\emph{It should be noted that the hyperbolic space $\H^3$ cannot be recovered by any of these models.}
\end{remark}

The unit tangent bundle of $\Ek$ is naturally identified with the space of all oriented planes in the tangent bundle of $\Ek$. In this way, for any oriented plane $\Pi\subset T_p \Ek$ in the unit tangent bundle of $\Ek$, we can consider its \emph{inclination} or \emph{angle} as the product $\esiz N,\xi(p)\esde$ between its unit normal $N$ and the vertical field $\xi$ at $p\in \Ek$. Note that, while three-dimensional space forms are \emph{isotropic}, i.e. any two elements of its unit tangent bundle can be connected by an ambient isometry, this property is not true for general homogeneous three-manifolds. For $\Ek$ spaces an intermediate situation happens:

\begin{lemma}\label{incli}
Let $\Pi_i\subset T_{p_i} \Ek$, $i=1,2$, be two oriented planes in the unit tangent bundle of $\Ek$, and assume that they have the same inclination. Then there exists an orientation preserving isometry $\Psi$ of $\Ek$ with $\Psi(p_2)=p_1$, $d\Psi_{p_2} (\Pi_2)= \Pi_1$ and $d\Psi_{p_2}(\xi(p_2))=\xi(p_1)$. If the inclination is not $\pm 1$, the isometry $\Psi$ is unique.
\end{lemma}
\begin{proof}
Observe first that for every $p_1,p_2\in \Ek$ there exists an orientation preserving isometry $\Psi'$ of $\Ek$ with $\Psi'(p_2)=p_1$ and $d\Psi'_{p_2}(\xi(p_2))=\xi(p_1)$. To see this, it suffices to consider an adequate left translation in the space when $\Ek$ is not $\S^2(\kappa)\times \R$  (in which case $\Ek$ is a Lie group endowed with a left invariant metric, and $\xi$ is left invariant; see e.g. \cite{mpe11}), or the composition of an adequate orientation preserving isometry of $\S^2(\kappa)$ with a vertical translation when the ambient space is $\S^2(\kappa)\times \R$. The inclination of the planes $d\Psi'_{p_2} (\Pi_2)$ and $\Pi_1$ are clearly the same, which implies that there exists a rotation $\mathcal{R}$  with vertical axis passing through $p_1$ in $\Ek$ that sends $d\Psi'_{p_2} (\Pi_2)$ to $\Pi_1$. The composition $\Psi:= \cR\circ \Psi'$ proves the desired existence.

As regards uniqueness, assume that the inclination of $\Pi_2$ is not $\pm 1$, and let $\Psi^1, \Psi^2$ denote two isometries in the conditions of Lemma \ref{incli}. If $v$ denotes the oriented unit normal of the plane $\Pi_2$, then $d\Psi^1_{p_2} (w)=d\Psi^2_{p_2}(w)$ for each $w\in \{\xi(p_2),v,\xi(p_2)\times v\}$. This proves that $\Psi^1=\Psi^2$, as wished.

\end{proof}

\subsection{Transitive families of surfaces in $\Ek$} We next recall a concept from our previous work \cite{GM3} that will be of special importance for our purposes here: the notion of \emph{transitive family of surfaces} in a Riemannian three-manifold $M$.

Let $M$ be an oriented Riemannian three-manifold, and consider for every immersed oriented surface $\Sigma$ in $M$ its associated Legendrian lift $\cL_{\Sigma}: \Sigma \flecha TU(M)$ into the unit tangent bundle $$TU(M)=\{(p,w): p\in M, w\in T_p M, |w|=1\},$$ which assigns to each $q\in \Sigma$ the value $\cL_{\Sigma} (q)= (q, N(q))\in TU(M)$, where $N$ denotes the unit normal of $\Sigma$.

\begin{definition}\label{defitransi}
Let $\cS$ be a family of immersed oriented surfaces in $M$. We say that $\cS$ is a \emph{transitive family} if the family of Legendrian lifts $\{\cL_S : S \in \cS\}$ satisfies:
 \begin{enumerate}
 \item
Each $\cL_S$ is an embedding into $TU(M)$. 
\item
For every $(p,w)\in TU(M)$ there exists a unique $S=S(p,w)\in \cS$ with $(p,w)\in \cL_S$.
\item
 The family $\cS=\{S(p,w):(p,w)\in TU(M)\}$ is $C^3$ with respect to $(p,w)$.
 \end{enumerate}
\end{definition}

Assume now that $M=\Ek$. We will be interested in finding transitive families of \emph{rotational} surfaces $\Sigma$ in $\Ek$, which will be parametrized as 
 \begin{equation}\label{eq4ro}
\psi(s,t)=\Psi_t (\gamma(s)),\end{equation} where $\{\Psi_t : t\in \S^1\}$ is a one-parameter group of rotations in $M$ that leave pointwise fixed some vertical geodesic of $M$ (the \emph{axis} of $\Sigma$), and $\gamma(s)$ is the profile curve.

Let ${\rm Iso}^0(M)$ denote the group of orientation preserving isometries of $M=\Ek$ that preserve the vertical Killing field $\xi$. Given $\Sigma$ a rotational surface in $\Ek$, let us denote $$\cS_{\Sigma}:=\{\Psi (\Sigma) : \Psi \in {\rm Iso}^0(M)\}.$$ In this family, we will identify two elements $\Psi_1(\Sigma)$, $\Psi_2(\Sigma)$ if there exists some rotation $R$ in $\Ek$ such that $\Psi_2=\Psi_1\circ R$. Note that in this case $\Psi_i(\Sigma)$, $i=1,2$, define the same point set in $\Ek$.

Given an immersed oriented surface $\Sigma$ in $\Ek$, the \emph{angle function} $\nu:\Sigma\flecha [-1,1]$ of $\Sigma$ is defined as the product $\nu=\esiz \eta,\xi\esde$, where $\eta$ is the unit normal of $\Sigma$ and $\xi$ the unit vertical Killing field of $\Ek$. In other words, the angle function assigns to each $q\in \Sigma$ the inclination of $T_q \Sigma$ in $\Ek$.

If $\Sigma$ is a rotational surface in $\Ek$, given by \eqref{eq4ro}, then the angle function of $\Sigma$ only depends on $s$, i.e. it does not depend on the rotation parameter $t$.

\begin{lemma}\label{gaussmap}
Let $\Sigma$ be a surface of revolution immersed in $M=\Ek$, and let $\nu(s):I\subset \R\flecha [-1,1]$ denote its angle function, defined in terms of a regular parameter $s$ of the profile curve of $\Sigma$. Assume that $\nu(s)$ is injective. Then:
 \begin{enumerate}
 \item
The Legendrian lift $\cL_{\Sigma}:\Sigma\flecha TU(M)$ is an embedding.
\item
If $\Sigma$ is a sphere and $\nu(s)$ is surjective onto $[-1,1]$, then $\cS_{\Sigma}$ is a transitive family of surfaces in $M$.
\end{enumerate}
\end{lemma}
\begin{proof}
Let $q_1,q_2\in \Sigma$ have $\cL_{\Sigma}(q_1)=\cL_{\Sigma} (q_2)$. Then the angle functions of $\Sigma$ at both points agree, which implies (by injectivity of $\nu$) that $q_1,q_2$ lie in the same meridian curve of $\Sigma$. As the unit normal of $\Sigma$ at both points also agrees, we easily conclude that $q_1=q_2$; thus $\cL_{\Sigma}$ is an embedding, what proves the first item.

For the second item, note that the family $\cS_{\Sigma}$ satisfies conditions (1) and (3) of Definition \ref{defitransi}. We now show that $\cS_{\Sigma}$ also satisfies the second one. Let $(p,w)\in TU(M)$, let $\nu_0$ denote the inclination of $w$ (i.e. $\nu_0:=\esiz w, \xi(p)\esde$), and let $\beta$ denote the meridian curve of $\Sigma$ along which the angle $\nu(s)$ of the surface equals $\nu_0$ (if $\nu_0=\pm1$, $\beta$ is just a point at which $\Sigma$ touches its rotation axis orthogonally). By Lemma \ref{incli} we see that for any $q\in \beta$ there exists an isometry $\Psi\in {\rm Iso}^0(M)$ with $\Psi(q)=p$, $d\Psi_q (N(q))=w$ and $d\Psi_q (\xi(q))=\xi(p)$, where $N$ denotes the unit normal of $\Sigma$. Moreover, by the rotational symmetry of $\Sigma$, two such isometries $\Psi, \Psi'$ corresponding to different points $q,q'\in \beta$ are related by $\Psi=\Psi'\circ R$ where $R$ is a rotation in $M$ around the axis of the surface $\Sigma$, with $R(q)=q'$. Thus, $\Psi(\Sigma)$ and $\Psi'(\Sigma)$ define the same element in the family $\cS_{\Sigma}$. The existence and uniqueness of this isometry proves condition (2) of Definition \ref{defitransi} and completes the proof of Lemma \ref{gaussmap}.
\end{proof}

\subsection{Rotational surfaces in $\Ek$: basic formulas}

Let $\Sigma$ be a rotational surface in $\Ek$, which we will assume is contained in a coordinate model $\cR^3(\kappa,\tau)$, and so that its rotation axis is the $x_3$-axis in the $(x_1,x_2,x_3)$-coordinates of this model. Hence, we can parametrize $\Sigma$ as 
 \begin{equation}\label{res0}
 \psi(u,v)=(\rho(u) \cos v, \rho (u) \sin v, h(u)),
 \end{equation}
with $4+\kappa \rho(u)^2 >0$ for all $u$. A long but direct computation in this model shows that the angle function of $\Sigma$ is 
 \begin{equation}\label{res1}
 \nu(u)= \frac{4 \rho'(u)}{\sqrt{h'(u)^2 (4+ \kappa \rho(u)^2)^2 + 16 \rho'(u)^2 (1+\tau^2 \rho(u)^2)}},
 \end{equation}
that its mean curvature $H$ is 
 \begin{equation}\label{res2}
 H=\frac{\left(4+\kappa  \rho^2\right)^2 \left(-h'^3 \kappa ^2 \rho^4+16 h' \left(h'^2-\rho \rho''+\rho'^2\right)+16 \rho^3
   \tau ^2 (h'' \rho'-h' \rho'')+16 h'' \rho \rho'\right)}{8 \rho \left(h'^2 \left(4+\kappa 
   \rho^2\right)^2+16\rho '^2 \left(\rho^2 \tau ^2+1\right)\right)^{3/2}}
 \end{equation}
and that its extrinsic curvature $K_e$ (i.e. the product of its principal curvatures) is

\begin{equation}\label{res3}
\def\arraystretch{1.6}\begin{array}{lll} K_e &=&\displaystyle\frac{h' \left(4+\kappa \rho^2\right)^2 \left(4-\rho^2 \left(\kappa -8 \tau ^2\right)\right) \left(h'' \rho' \left(4+\kappa 
   \rho^2\right)-h' \left(\rho'' \left(4+\kappa \rho^2\right)-2 \rho \rho'^2 \left(\kappa -4 \tau ^2\right)\right)\right)}{\rho \left(h'^2
   \left(4+\kappa \rho^2\right)^2+16 \rho'^2 \left(1+\rho^2 \tau ^2\right)\right)^2} \\ &  &-\displaystyle \frac{ \tau
   ^2 \left(h'^2 \left(4+\kappa \rho^2\right)^2+4 \rho^2 \rho'^2 \left(\kappa -4 \tau ^2\right)\right)^2}{ \left(h'^2
   \left(4+\kappa \rho^2\right)^2+16 \rho'^2 \left(\rho^2 \tau ^2+1\right)\right)^2}.\end{array}\end{equation}
Note that 
 \begin{equation}\label{res30}
 \nu^2 \leq \frac{1}{1+\tau^2 \rho^2}.
 \end{equation}

Consider next the parameter $s=s(u)$ for the profile curve given by the arclength parameter of the curve with respect to the Riemannian metric in the $(\rho,h)$-plane 
 \begin{equation}\label{metpv}
d\sigma^2 = (1+\tau^2 \rho^2) d\rho^2 + \frac{(4+\kappa \rho^2)^2}{16} dh^2 .
 \end{equation}


It follows then from \eqref{res1}, \eqref{res2} and \eqref{res3} that, with respect to this parameter $s$, we have

\begin{equation}\label{res5}
\nu(s)=\rho'(s),
\end{equation}
 \begin{equation}\label{resH}
 H= \frac{4+\rho'^2 (-4+(\kappa-8\tau^2)\rho^2)-\rho(\kappa \rho + \rho''(4+\kappa \rho^2)(1+\tau^2\rho^2))}{8 \rho \sqrt{1-\rho'^2 (1+\tau^2 \rho^2)}},
 \end{equation}
and 
\begin{equation}\label{res6}
K_e = -\tau^2 + \frac{\rho''}{16 \rho} (4+\kappa \rho^2)(-4+(\kappa-8\tau^2) \rho^2), \hspace{1cm} \rho=\rho(s).
\end{equation}

\section{General Weingarten surfaces in $\Ek$}\label{subsec:geo3}

\subsection{Rewriting the Weingarten equation $W(\kappa_1,\kappa_2)=0$}\label{sec:rew}
Let $\Sigma$ be an immersed oriented surface in an oriented Riemannian three-manifold $(M,\esiz,\esde)$, and assume that $\Sigma$ is an elliptic Weingarten surface. That is, its principal curvatures $\kappa_1,\kappa_2$ satisfy an equation of the type  \begin{equation}\label{weq}W(\kappa_1,\kappa_2)=0,\end{equation} where $W\in C^\8(\R^2)$ is symmetric (i.e. $W(k_1,k_2)=W(k_2,k_1)$) and satisfies the ellipticity condition 
 \begin{equation}\label{eliwe}
 \frac{\parc W}{\parc k_1} \frac{\parc W}{\parc k_2}>0 \hspace{1cm} \text{if} \ W=0. \end{equation} 
 Note that \eqref{eliwe} implies that $W^{-1}(0)\subset \R^2$ is a disjoint union of regular curves, all of which can be seen as local graphs of the form $k_1=f(k_2)$, with $f$ strictly decreasing ($f'<0$).
 
 The symmetry condition for $W$ ensures that \eqref{weq} can be written in the form $\Phi(H,K_e)=0$ for some $C^\8$ function $\Phi$, where $H,K_e$ denote respectively the mean and extrinsic curvatures of $\Sigma$. The ellipticity condition \eqref{eliwe} implies that, when \eqref{weq} is seen as a second order PDE after writing the surface $\Sigma$ as a local graph over its tangent plane with respect to some local coordinate system on $M$, this PDE is elliptic.
 
 Nonetheless, it is convenient to rewrite \eqref{weq} in a more adequate way for its study in order to avoid certain problems. For example, consider the elliptic Weingarten functional $W(k_1,k_2)=(k_1-1)(k_2-1)-1$, and the related Weingarten equation \eqref{weq} for surfaces in $\R^3$. Clearly, planes are solutions to this equation, but round spheres of radius $1/2$ also are. At first sight, this would seem to contradict the maximum principle for elliptic PDEs. However, this situation is explained by the fact that $W^{-1}(0)\subset \R^2$ has two connected components, one corresponding to the region where $k_i<1$ and the other to the region $k_i>1$, for $i=1,2$. In this sense, the geometry of the Weingarten surfaces that satisfy \eqref{weq} depends not only on the function $W$, but also on the connected component of $W^{-1}(0)$ in which the pair $(\kappa_1(p),\kappa_2(p))$ lies in, for all $p\in \Sigma$.

One way to handle this indetermination is the following: let $(x_0,y_0)\in W^{-1}(0)$, and let $\Gamma$ be the connected component of $W^{-1}(0)$ that contains $(x_0,y_0)$; note that if $W^{-1}(0)$ is empty, so is its associated geometric theory and there is nothing to study. It then follows by the symmetry and ellipticity conditions on $W$ that $\Gamma$ can be seen as a graph of the form
 \begin{equation}
 \frac{k_1+k_2}{2} = \phi\left(\frac{(k_1-k_2)^2}{4}\right),
 \end{equation}
where $\phi\in C^\8([0,\8))$ satisfies (by ellipticity) the condition
 \begin{equation}\label{elihk}
 4t (\phi'(t))^2<1 \hspace{1cm} \text{ for all } t\geq 0.
 \end{equation}
In other words, for the immersed surface $\Sigma$ in $M$, the relation $W(\kappa_1,\kappa_2)=0$ with $(\kappa_1,\kappa_2)\in \Gamma$ can be rewritten in terms of its mean and extrinsic curvature $H,K_e$ as
 \begin{equation}\label{weq2}
 H=\phi(H^2-K_e)
 \end{equation}
with $\phi$ satisfying \eqref{elihk}; note that this corresponds to \eqref{speci1}.

There is another way of handling the indetermination in equation \eqref{weq} that will also be useful for our purposes. By the conditions on the function $W$, it is clear that we can write $W(k_1,k_2)=0$ for $(k_1,k_2)\in \Gamma$ in $\R^2$ as
 \begin{equation}\label{weq3}
 k_1=f(k_2), \hspace{1cm} 
 \end{equation}
where $f$ is defined on an interval $(a,b)\subset \R$, and satisfies the following conditions:

\begin{enumerate}
\item[(i)]
$f$ is $C^\8$, and $f'<0$ (by ellipticity).
 \item[(ii)]
$f\circ f = {\rm Id} $ (by symmetry of $W$).
 \item[(iii)]
If $a\neq -\8$, then $b=+\8$ and $f(x)\to +\8$ as $x\to a$.
 \item[(iv)]
If $b\neq +\8$, then $a=-\8$ and $f(x)\to -\8$ as $x\to b$.
\end{enumerate}

Moreover, there exists $\alfa\in \R$, 	that we will call the \emph{umbilicity constant} of \eqref{weq3}, given by $\alfa=f(\alfa)$. For \eqref{weq2}, this constant is given by $\alfa=\phi (0)$, and for \eqref{weq}, by $W(\alfa,\alfa)=0$ with $\alfa \in \Gamma$. By making, if necessary, the change $f(x)\mapsto -f(-x)$ in \eqref{weq3} while reversing the orientation of the surface, we may assume without loss of generality that $\alfa\geq 0$; \emph{this will be assumed from now on}.

\begin{remark}
In the case that $W(k_1,k_2)$ is not symmetric, the fully nonlinear elliptic PDE associated to the Weingarten surfaces that satisfy $W(\kappa_1,\kappa_2)=0$ when we view the surface as a local graph in $M$ is not $C^1$ at the umbilical points of $\Sigma$. This creates an important complication in the study of such elliptic Weingarten surfaces, even when $M$ is the Euclidean three-space $\R^3$. See e.g. \cite{A,GM3} for more details on this issue.
\end{remark}

\subsection{General Weingarten surfaces in $\Ek$: definition}\label{sec:wde}

We consider now $M=\Ek$. Initially, we define a \emph{general Weingarten surface} in $\Ek$ as an immersed oriented surface in $\Ek$ whose mean curvature $H$, extrinsic curvature $K_e$ and Gauss curvature $K$ satisfy an elliptic $C^\8$ relation $\cF(H,K_e,K)=0$. It is however convenient to rewrite this expression in terms of the principal curvatures $\kappa_1, \kappa_2$ and the \emph{angle function} $\nu$ of the surface.

First, by the Gauss equation, $K,K_e$ and $\nu$ are related by the formula (see e.g. \cite{D})
 \begin{equation}\label{gausseq}
 K= K_e + \tau^2 + (\kappa-4\tau^2)\nu^2.
 \end{equation}
Thus, noting that $H=(\kappa_1+\kappa_2)/2$, $K_e=\kappa_1 \kappa_2$, the equation $\cF(H,K_e,K)=0$ can be rewritten as $W(\kappa_1,\kappa_2,\nu^2)=0$. Clearly, the resulting function $W(k_1,k_2,v)\in C^{\8}(\R^2\times [0,1])$ satisfies the symmetry condition $W(k_1,k_2,v)=W(k_2,k_1,v)$. 

Let now $\Sigma$ be an immersed oriented surface in $M=\Ek$, and $p\in \Sigma$. Choose local coordinates $(x,y,z)$ on $M$ around $p$, so that $\Sigma$ is viewed around $p$ as an upwards-oriented graph $z=u(x,y)$ in these coordinates. Then, the mean and extrinsic curvatures of $\Sigma$ are given by $$H=\mathcal{H} (x,y,u,u_x,u_y,u_{xx},u_{xy},u_{yy}), \hspace{0.5cm}  K_e = \mathcal{K}(x,y,u,u_x,u_y,u_{xx},u_{xy},u_{yy}),$$ where $\mathcal{H}, \mathcal{K}$ are smooth on $\cU \times \R^5\subset \R^8$; here $\cU\subset \R^3$ is the region where the coordinates $(x,y,z)$ vary. 

Taking this into account, it is clear that the general Weingarten equation $W(\kappa_1,\kappa_2,\nu^2)=0$ in $M$ can be written as a PDE 
$$\Psi[u]:=\Psi(x,y,u,u_x,u_y,u_{xx},u_{xy},u_{yy})=0,$$ where $\Psi=\Psi(x,y,z,p,q,r,s,t)$ is smooth on $\cU \times \R^5\subset \R^8$. 

Moreover, one can check that the ellipticity condition $4\Psi_r \Psi_t-\Psi_s^2>0$ for $\Psi$ holds at points where $\Psi(x,y,z,p,q,r,s,t)=0$ if and only if $W_{k_1} W_{k_2}>0$ holds for $W$ at points with $W(k_1,k_2,v)=0$; a way to do this is to easily check that this property is true at the origin for a canonical coordinate system $(x_1,x_2,x_3)$ of $\Ek$, and then use the geometric invariance of equation $W(\kappa_1,\kappa_2,\nu^2)=0$ to deduce that this property holds true for arbitrary coordinates $(x,y,z)$ in $\Ek$. Thus, we give the following definition:

\begin{definition}\label{prescri}
A \emph{general (elliptic) Weingarten surface} in $M=\Ek$ is an immersed oriented surface in $\Ek$ whose principal curvatures $\kappa_1,\kappa_2$ and angle function $\nu$ verify a relation
\begin{equation}\label{preq}
W(\kappa_1,\kappa_2,\nu^2)=0,
\end{equation}
where the function $W=W(k_1,k_2,v)\in C^{\8} (\R^2\times [0,1])$ satisfies:
 \begin{enumerate}
 \item
$W$ is symmetric in $k_1,k_2$, i.e. $W(k_1,k_2,v)=W(k_2,k_1,v)$.
 \item
At points $(k_1,k_2,v)$ where $W=0$ we have the following ellipticity condition:
 \begin{equation}\label{preli}
 \frac{\parc W}{\parc k_1}\frac{\parc W}{\parc k_2} >0.
 \end{equation}
 \end{enumerate}

A \emph{general Weingarten class of surfaces in $\Ek$}, denoted $\cW$, is the class of immersed oriented surfaces in $\Ek$ that satisfy a specific general Weingarten relation \eqref{preq}.

\end{definition}

By the arguments in Section \ref{sec:rew}, the general Weingarten equation \eqref{preq} subject to the conditions of Definition \ref{prescri} can be alternatively written for some smooth function $\Phi(t,v)$ on $[0,\8)\times [0,1]$ as 
 \begin{equation}\label{preq2}
 H= \Phi (H^2-K_e,\nu^2), \hspace{0.5cm} \text{with} \hspace{0.5cm} 4 t \left(\frac{\parc \Phi}{\parc t}(t,v)\right)^2<1 \hspace{0.5cm}
  \forall (t,v),
 \end{equation}
i.e. we recover \eqref{gwi}-\eqref{gwi2}. Likewise, we can also write \eqref{preq} in these conditions as
\begin{equation}\label{preq3}
\kappa_1=f(\kappa_2,\nu^2)
\end{equation}
where for each $v\in [0,1]$ fixed, the function $f=f(\cdot, v)$ is defined on a real interval $(a,b)=(a(v),b(v))$, and satisfies properties (i) to (iv) after equation \eqref{weq3} of Section \ref{sec:rew}.

We should note that, in this situation, each interval  $(a,b)=(a(v),b(v))$ depends on $v\in [0,1]$ in a non-necessarily continuous way (one can construct examples that illustrate this possibility).

The most studied case of general elliptic Weingarten surfaces in $\Ek$ spaces is, obviously, the case of constant mean curvature surfaces. An outline of the beginning of this theory can be found in \cite{DHM,FM}. Surfaces of constant positive extrinsic curvature in $\H^2\times \R$ and $\S^2\times \R$ have been studied in \cite{EGR}. The case of surfaces of constant (intrinsic) Gauss curvature in $\H^2\times \R$ and $\S^2\times \R$ was first studied in detail in \cite{AEG}. Surfaces satisfying general elliptic Weingarten equations of the type $H=\Phi(\nu^2)$ were discussed in \cite{GM}, in a more general context of surfaces in metric Lie groups whose mean curvature is given as a function of its left-invariant Gauss map. A general study of surfaces in $\R^3$ satisfying $H=\Phi(\nu^2)$ can be found in \cite{BGM}. For results about elliptic Weingarten surfaces in $\H^2\times \R$ and $\S^2\times \R$, see \cite{FP,M,MR}.

\subsection{General Weingarten surfaces and elliptic PDEs}\label{sec:wedp}

We next study the elliptic PDE that defines general Weingarten surfaces with respect to some special coordinates on $M=\Ek$. As explained in Section \ref{descr}, $\Ek$ can be covered by one (resp. two) canonical coordinate charts if $\kappa\leq 0$ (resp. $\kappa>0$). Let $\Rk$ be one of these standard coordinate model for $\Ek$, and let $\cW$ be a class of general Weingarten surfaces in $\Ek$, given by some function $W$ in the conditions of Definition \ref{prescri}. Let $(x_1,x_2,x_3)$ be the canonical coordinates in $\Rk$. Then, by the previous discussion, there exists an elliptic PDE $F[u]=0$ in these coordinates that models the class $\cW$, in the sense that an upwards-oriented graph $x_3=u(x_1,x_2)$ is an element of $\cW$ if and only if $u$ is a solution to $F[u]=0$. Also by our previous discussion, the function $F$ satisfies the following properties:

\begin{enumerate}
 \item
$F=F(x,y,z,p,q,r,s,t)\in C^{\8}(\cU)$, where $\cU=\Rk\times \R^5\subset \R^8$. 
 \item
The following ellipticity condition holds on $F^{-1}(0)\subset \cU$: \begin{equation}\label{unieli} 4 F_r F_t-F_s^2>0.\end{equation}
 \item
$F$ is invariant by vertical translations, i.e. $F$ does not depend on $z$.
 \item
$F$ is \emph{rotationally invariant} with respect to the $x_3$-axis. This means: if $u=u(x_1,
x_2)\in C^2(\Omega)$ for some domain $\Omega \subset \R^2$, and $\theta\in [0,2\pi)$, then if we define $$(x_1^{\theta},x_2^{\theta}):= (\cos \theta x_1 + \sin \theta x_2, -\sin \theta x_1 + \cos \theta x_2),$$ and $u^{\theta}(x_1^{\theta},x_2^{\theta})$ by the relation $u^{\theta}(x_1^{\theta},x_2^{\theta})=u(x_1,x_2)$ for every $(x_1,x_2)\in \Omega$, the following holds: \emph{the value of $F[u]$ at $(x_1,x_2)$ equals the value of $F[u_{\theta}]$ at $(x_1^{\theta},x_2^{\theta})$}.
 \end{enumerate}

We can also define, associated to the class $\cW$, a similar elliptic PDE $F^*[u]=0$ for the same canonical coordinates $(x_1,x_2,x_3)$, but this time with respect to the \emph{downwards} vertical direction. Since $\cW$ is closed under the transformation $(x_1,x_2,x_3)\mapsto (x_1,-x_2,-x_3)$, which is an orientation preserving isometry of all $\Ek$ spaces, we can conclude that the function $F^*$ is determined by $F$, by 
 \begin{equation}\label{ffes}
 F(x,y,p,q,r,s,t)=F^*(x,-y,-p,q,-r,s,-t).
 \end{equation}

We consider next horizontal directions. Let $\Sigma$ be an element of $\cW$ that can be seen as a downwards-oriented graph $x_1=f(x_2,x_3)$ for the canonical coordinates $(x_1,x_2,x_3)$; here, \emph{downwards-oriented} means that the unit normal of $\Sigma$ points towards the region $\{x_1<0\}$ at every $p\in \Sigma$. Then, $f$ satisfies an elliptic PDE
 \begin{equation}\label{defho}
 G[f]:=G(f,x_2,x_3,f_{x_2},f_{x_3},f_{x_2x_2}, f_{x_2 x_3},f_{x_3 x_3})=0
 \end{equation}
for some $G\in C^\8 (\Rk\times \R^5)$. Note that $G$ does not actually depend on $x_3$, since $\cW$ is closed by the vertical translations $(x_1,x_2,x_3)\mapsto (x_1,x_2,x_3+c)$, $c\in \R$. Similarly, since $\cW$ is closed by the $180º$-rotation $(x_1,x_2,x_3)\mapsto (x_1,-x_2,-x_3)$, the function $G$ satisfies the symmetry condition 
\begin{equation}\label{simh}
G(x,y,p,q,r,s,t)=G(x,-y,-p,-q,r,s,t).
\end{equation}
Moreover, since $\cW$ is closed by arbitrary rotations around the $x_3$-axis, the elliptic PDE \eqref{defho} determines the corresponding elliptic PDE associated to $\cW$ for any other horizontal direction, and in particular for upwards-oriented graphs $x_1=f(x_2,x_3)$.

In this sense, all properties of the general Weingarten class $\cW$ are basically condensed on the elliptic PDEs $F[u]=0$ and $G[u]=0$.

\begin{definition}\label{defifu}
We call $F$ and $G$ the \emph{defining functions} of the general Weingarten class $\cW$.
\end{definition}

\section{The geometry of canonical rotational examples}\label{sec:canonical}

In this section we will analyze the geometry of rotational general Weingarten surfaces in $\Ek$ that intersect their rotation axis orthogonally. First, in Section \ref{sec:app} we will show that such rotational surfaces exist, by means of a more general theorem about existence of radial solutions of fully nonlinear elliptic PDEs in dimension two. We remark that Section \ref{sec:app} can be read independently from the rest of the paper.

\subsection{Existence of radial solutions of fully nonlinear elliptic PDEs}\label{sec:app}
Consider the second order PDE in two variables for $u=u(x,y)$
 \begin{equation}\label{fulinoa}
 F(x,y,u_x,u_y,u_{xx},u_{xy},u_{yy})=0,
 \end{equation}
which for brevity will be denoted as $F[u]=0$, where $F=F(x,y,p,q,r,s,t)\in C^{3}(\cU)$, with $\cU\subset \R^7$ a convex open set. We assume that $F$ satisfies:
 \begin{enumerate}
 \item[$i)$]
$F_r>0$ and $4 F_r F_t - F_s^2>0$ on $\cU$ (ellipticity condition).
 \item[$ii)$]
There exists $\alfa\in \R$ such that $p_0:=(0,0,0,0,\alfa,0,\alfa)\in \cU$ and $F(p_0)=0$.
 \item[$iii)$]
$F$ is \emph{rotationally invariant} with respect to the $z$-axis, in the sense explained in Section \ref{sec:wedp}, i.e., for every $C^2$ function $u(x,y)$ and every $\theta\in [0,2\pi)$, if we define $(x_{\theta},y_{\theta}):= (\cos \theta x + \sin \theta y, -\sin \theta x + \cos \theta y)$ and $u^{\theta}$ given by $u^{\theta}(x_{\theta},y_{\theta})=u(x,y)$, then the value of $F[u]$ at $(x,y)$ equals the value of $F[u_{\theta}]$ at $(x_{\theta},y_{\theta})$.
 \end{enumerate}
Observe that since $F$ does not depend on the variable $z$, a solution to \eqref{fulinoa} is defined up to additive constants. In this Section \ref{sec:app} we show that there exists a radial solution $u\in C^2(D(0,\delta))$ of \eqref{fulinoa}, defined on a sufficiently small disk $D(0,\delta)\subset \R^2$. Here, by \emph{radial} we mean that $u$ depends solely on $\rho:=\sqrt{x^2+y^2}$. For any radial function $u$, denoting $x+i y=\rho e^{i\theta}$, we have
$u_{x}+iu_{y}= u'(\rho) e^{i\theta}$ and $$\def\arraystretch{1.2}\begin{array}{lll}u_{xx} &=& \cos^2 \theta \, u''(\rho) + \sin^2 \theta \, \frac{u'(\rho)}{\rho}, \\ u_{xy} &=& \cos \theta \sin \theta  \, (u''(\rho) -\frac{u'(\rho)}{\rho}),\\  u_{yy} &=& \sin^2 \theta \, u''(\rho) + \cos^2 \theta \, \frac{u'(\rho)}{\rho}.\end{array}$$ This shows that $F[u]=0$ can be reduced for radial solutions to a second order ODE. Specifically, by making $\theta=0$, we see that $u(\rho)$ satisfies
 \begin{equation}\label{edofula}
F(\rho,0,u'(\rho),0,u''(\rho),0,u'(\rho)/\rho)=0.
 \end{equation}
Since $F$ is rotationally invariant, the converse also holds: any solution $u(\rho)$ to \eqref{edofula} trivially describes a radial solution to \eqref{fulinoa}.

However, we should note that the ODE \eqref{edofula} is singular for the Cauchy data $u(0)=u'(0)=0$, so it cannot be solved directly to yield the solution we are looking for.

Also, it is obvious that the condition $ii)$ above is indispensable for the existence of a radial solution to \eqref{fulinoa} that is $C^2$ in some disk $D(0,\delta)$, since any such solution would satisfy $u_x=u_y=u_{xy}=0$ and $u_{xx}=u_{yy}$ at the origin. By the ellipticity condition $i)$ and the convexity of $\cU$, the number $\alfa\in \R$ in $ii)$ is unique.

\begin{lemma}\label{empezar}
Let $F\in C^3(\cU)$ satisfy conditions $i),ii),iii)$ above. Then, there exists a radial solution $u\in C^2(D(0,\delta))$ to \eqref{fulinoa} for $\delta>0$ small enough.

Moreover, any other radial solution $v\in C^2(D(0,\delta'))$ to \eqref{fulinoa} is given on $D(0,{\rm min}\{\delta,\delta'\})$ by $v=u+c$ for some $c\in\R$.

\end{lemma}
\begin{proof}
Uniqueness is immediate by the maximum principle and the independence of $F$ with respect to $z$. For the existence part, we use the continuity method. Let $\alfa\in \R$ be given by condition $ii)$ above, choose $\delta\in (0,1/|\alfa|)$ and let $\phi^{\alfa}(x,y):=c_{\alfa} - {\rm sign} (\alfa) \sqrt{(1/\alfa)^2-x^2-y^2}$ be the function defining a hemisphere of center $(0,0,0)$ and radius $1/|\alfa|$, translated by the constant $c_{\alfa}\in \R$ so that $\phi^{\alfa} =0$ on $\parc D(0,\delta)$ (if $\alfa=0$ we simply take $\phi^{\alfa}=0$). Define, for each $\sigma\in [0,1]$,
 \begin{equation}\label{eqsigma}
F^{\sigma}(x,y,p,q,r,s,t):=F(x,y,p,q,r,s,t)-\sigma F(x,y,\phi^{\alfa}_x,\phi^{\alfa}_y,\phi^{\alfa}_{xx},\phi^{\alfa}_{xy},\phi^{\alfa}_{yy}).
 \end{equation}
Associated to $F^{\sigma}$ we can consider the continuous PDE family $F^{\sigma}[u]=0$, which, choosing a smaller $\delta>0$ if necessary, is well defined and elliptic on the convex open set $\cV:=\cU\cap \{(x_1,\dots, x_7) : x_1^2+x_2^2<\delta^2\}$.
Note that:
 \begin{enumerate}
 \item
$F^{\sigma}\in C^3(\cV)$ and $F^{\sigma}_r=F_r$, $F^{\sigma}_s=F_s$, $F^{\sigma}_t=F_t$ on $\cV$.
 \item
$F^{\sigma}$ is rotationally invariant with respect to the $z$-axis.
 \item
 $F_0=F$.
 \item
$p_0\in \cV$ and $F^{\sigma}(p_0)= 0$ for every $\sigma\in [0,1]$.
 \item
For $\sigma=1$, the function $\phi^{\alfa}$ is a solution to the elliptic PDE $F^{\sigma}[u]=0$ with $\phi^{\alfa}=0$ on $\parc D(0,\delta)$.
\end{enumerate}
In order to prove the existence part in Lemma \ref{empezar}, it suffices to show that the Dirichlet problem $F[u]=0$ on $\Omega:=D(0,\delta)$, with $u=0$ on $\parc\Omega$, has a solution for $\delta>0$ small enough. Indeed, since $F$ does not depend on $z$, the solution to this Dirichlet problem is unique, and thus by condition $iii)$,  a radial function.

A standard application of the continuity method (see e.g. Theorem 17.8 in \cite{GT}) shows that this Dirichlet problem $F[u]=0$ on $\Omega$, $u=0$ on $\parc \Omega$, has a solution for some $\delta>0$ if we can obtain a priori $C^2$ estimates for the problems $F^{\sigma}[u]=0$ on $\Omega$, $u=0$ on $\parc \Omega$. We obtain these a priori estimates next.

Let $\ep>0$ such that $p_{\pm}:=(0,0,0,0,\alfa\pm \ep, 0,\alfa \pm \ep)\in \cV$ and $\alfa\pm \ep \neq 0$. Note that 
 \begin{equation}\label{ineqf}
 F^{\sigma} (p_-) <0< F^{\sigma} (p_+), \hspace{1cm} \forall \sigma\in [0,1],
 \end{equation}
by the monotonicity properties implied by the ellipticity of $F$ and the condition $F^{\sigma}(p_0)=0$. Define now the comparison hemispheres 
 \begin{equation}\label{defipm}
\phi_{\pm}(x,y):=c_{\pm} -{\rm sign} (\alfa \pm \ep) \sqrt{\frac{1}{(\alfa\pm \ep)^2} -x^2-y^2},\end{equation} where $c_{\pm}$ are constants to be determined later, and note that $(0,0,D \phi_{\pm} (0,0), D^2 \phi_{\pm} (0,0))= p_{\pm}$.  Thus, by \eqref{ineqf} there exists some $\delta>0$ small enough such that, for all $(x,y)\in D(0,\delta)$ and all $\sigma\in [0,1]$, the following conditions hold:

\begin{equation}\label{muchas}
\left\{ \def\arraystretch{1.4} \begin{array}{l} (x,y,D \phi_{\pm} (x,y), D^2 \phi_{\pm} (x,y)) \in \cV, \\ F^{\sigma}(x,y,D \phi_{-} (x,y), D^2 \phi_{-} (x,y))< -\frac{\gamma}{2}, \\ F^{\sigma}(x,y,D \phi_{+} (x,y), D^2 \phi_{+} (x,y))>\frac{\gamma}{2},  \end{array}\right.
\end{equation}
where $$\gamma:= {\rm min}_{\sigma\in [0,1]}\{{\rm min}\{|F^{\sigma}(p_{-})|,|F^{\sigma} (p_+)|\}\}>0.$$

Let now $u_{\sigma}(x,y)$ be the solution to the Dirichlet problem $F^{\sigma}[u]=0$ in $\Omega:=D(0,\delta)$, $u=0$ on $\parc \Omega$ for this new $\delta>0$; we note that $u_{\sigma}$ is, in case it exists, a radial function. We choose the constants $c_{\pm}$ in \eqref{defipm} so that $\phi_{\pm}=0$ on $\parc \Omega$ too. As $F$ is elliptic, we get from \eqref{muchas} and the comparison principle that 
\begin{equation}\label{comp1}
\phi_+(x,y)\leq u^{\sigma}(x,y)\leq \phi_- (x,y), \hspace{1cm} \forall (x,y)\in \overline{\Omega}, \ \forall \sigma\in [0,1].
\end{equation}
Noting that $\phi_{\pm}$ and $u_{\sigma}$ are all radial functions depending solely on $\rho=\sqrt{x^2+y^2}$ that coincide for $\rho =\delta$, it is immediate that $(\phi_-)'(\delta)\leq (u_{\sigma})'(\delta)\leq (\phi_+)'(\delta)$. As $F$ does not depend on $z$ and these inequalities for derivatives do not depend on the value chosen for $c_{\pm}$, we can derive from \eqref{comp1} the same estimate for any $\delta_0\in (0,\delta)$, obtaining finally that 
 \begin{equation}\label{comp2}
\phi_-'(\rho) \leq u_{\sigma}'(\rho)\leq \phi_+'(\rho)
, \hspace{1cm} \forall \rho\in[0,\delta), \ \forall \sigma\in [0,1].
 \end{equation}
This gives a priori $C^1$ estimates for the family of solutions $\{u_{\sigma} : \sigma\in [0,1]\}$.

In order to derive a priori $C^2$ estimates, we first observe that since $(F^{\sigma})_r =F_r >0$ for every $\sigma\in [0,1]$, by the implicit function theorem there exists some $R>0$ around $p_0$ such that, in the ball $\overline{B(p_0,R)}\subset \cV$, the equation $F^{\sigma}=0$ can be rewritten as 
 \begin{equation}\label{impli}
r= G^{\sigma} (x,y,p,q,s,t),\end{equation} for some $C^3$ function $G^{\sigma}$ defined in a convex neighborhood $\mathcal{O}$ of $(0,0,0,0,0,\alfa)\in \R^6$.

By taking $\ep>0$ small enough we can assume that $p_{\pm}\in B(p_0,R)$. Similarly, taking $\delta>0$ small enough we can assume that 
\begin{equation*} (x,y,D \phi_{\pm} (x,y), D^2 \phi_{\pm} (x,y)) \in B(p_0,R) \hspace{1cm} \forall (x,y)\in \Omega,\end{equation*} which implies in particular that \begin{equation}\label{comp5} \left(\rho,0,\phi_{\pm}'(\rho),0, \phi_{\pm}''(\rho),0, \frac{\phi_{\pm}'(\rho)}{\rho} \right)\in B(p_0,R) \hspace{1cm} \forall \rho\in (0,\delta).\end{equation}

By the radial symmetry of $u_{\sigma}$ and $F^{\sigma}$, equation $F^{\sigma}[u_{\sigma}]=0$ can be rewritten in $B(p_0,R)$ using \eqref{impli} as an ODE in normal form
 \begin{equation}\label{odesii}
 u_{\sigma}''(\rho)= G^{\sigma}\left(\rho,0, u_{\sigma}'(\rho),0,0,\frac{u_{\sigma}'(\rho)}{\rho}\right),
 \end{equation}
for every $\rho\in (0,\delta)$. In addition, by \eqref{comp2}, we have
 \begin{equation}\label{comp3}
 \frac{\phi_-'(\rho)}{\rho} \leq \frac{u_{\sigma}'(\rho)}{\rho} \leq \frac{\phi_+'(\rho)}{\rho}.
 \end{equation}
This implies that the right-hand side of \eqref{odesii} is well defined for all $\rho\in (0,\delta)$. Thus, by \eqref{impli}, $$\left(\rho,0, u_{\sigma}'(\rho),0,u_{\sigma}''(\rho),0,\frac{u_{\sigma}'(\rho)}{\rho}\right) \in B(p_0,r),$$ for all $\rho\in (0,\delta)$, and this implies that $$\{(x,y,D u_{\sigma} (x,y), D^2 u_{\sigma}(x,y)): (x,y)\in \Omega, \sigma\in [0,1]\} \subset B(p_0,R),$$ which has compact closure in $\cV$. This yields the desired a priori $C^2$ estimates and completes the proof of Lemma \ref{empezar}.
\end{proof}

\subsection{Rotational general Weingarten surfaces}\label{canexe}

Let $\cW$ denote a class of general Weingarten surfaces in $M=\Ek$, let $\Rk$ be a standard coordinate model for $\Ek$, with canonical coordinates $(x_1,x_2,x_3)$, and let $F\in C^\8(\Rk\times \R^5)$ be the \emph{defining function} of the class $\cW$ (Definition \ref{defifu}). In this section we are going to consider rotationally invariant surfaces of $\cW$ around the $x_3$-axis, given as upwards-oriented graphs $x_3=u(x_1,x_2)$ of radial solutions $u$ to $F[u]=0$. By the computations at the beginning of Section \ref{sec:app}, we know that $u(\rho)$ satisfies the second order ODE
 \begin{equation}\label{edoful}
F(\rho,0,u'(\rho),0,u''(\rho),0,u'(\rho)/\rho)=0,
 \end{equation}
and that, by rotational invariance of $F$, the converse also holds: any solution $u(\rho)$ to \eqref{edoful} trivially describes a radial solution to the PDE $F[u]=0$. 

When $\rho>0$ the ellipticity of $F[u]=0$ implies that $F_r\neq 0$ at all points of the form $$(\rho,0,u'(\rho),0,u''(\rho),0,u'(\rho)/\rho)\in F^{-1}(0).$$ Thus, \eqref{edoful} can be written locally in normal form around those points as 
 \begin{equation}\label{edonor1}
 u''(\rho)= \cF(\rho,u'(\rho)),
 \end{equation} 
for some smooth function $\cF$. In particular, \eqref{edoful} can be locally solved away from $\rho=0$.

However, \eqref{edoful} becomes singular when $\rho=0$, i.e. when the radial graph $x_3=u(\sqrt{x_1^2+x_2^2})$ touches its axis. Still, for this situation, Lemma \ref{empezar} shows that there exists a radial graph $x_3=u(\rho)$, $\rho:=\sqrt{x_1^2+x_2^2}$, that belongs to the class $\cW$ and which is defined on a disk $D(0,\delta)$. Note that this graph is a rotational surface in $\Ek$ that intersects its rotation axis orthogonally, and that $u(\rho)$ is a solution to \eqref{edoful}. Also, by the uniqueness statement of Lemma \ref{empezar} and the invariance of the class $\cW$ with respect to orientation preserving ambient isometries (including $180º$-rotations around horizontal geodesics of $\Ek$), we can deduce that this graph is, up to ambient isometries, the \emph{unique} rotational surface of the general Weingarten class $\cW$ that intersects its rotation axis orthogonally.

Here, in principle, \emph{uniqueness} is to be understood in the following way: if $S_1,S_2$ are two rotational surfaces of $\cW$ that touch their respective rotation axes orthogonally at points $p_1,p_2\in \Ek$, then there exists an orientation preserving isometry $\Psi$ of $\Ek$ with $\Psi(p_1)=p_2$ such that $S_1$ and $\Psi(S_2)$ coincide on a neighborhood of $p_1$.

But once here, noting that when $\rho>0$ the differential equation \eqref{edoful} can be written in the normal form \eqref{edonor1}, standard results from ODE theory imply that both $S_1,S_2$ can be extended to \emph{maximal} or \emph{inextendible} rotational surfaces of the class $\cW$, and in that case we have $S_1=\Psi(S_2)$.

All of this justifies the following definition.

\begin{definition}\label{def:rote}
Let $\cW$ denote a general Weingarten class of surfaces in $\Ek$. The \emph{canonical rotational example} of $\cW$ is defined as the unique (up to orientation preserving ambient isometries), inextendible surface of $\cW$ that is rotational and meets its rotation axis orthogonally.
\end{definition}

An illustrative example of this notion is given by the class of surfaces in $\H^2\times \R$ of constant mean curvature $H\in \R$. When $|H|>1/2$, the canonical rotational example is the sphere in $\H^2\times \R$ of mean curvature $H$. When $0<|H|\leq 1/2$, the canonical rotational example is a certain entire CMC graph with vanishing Abresch-Rosenberg differential. When $H=0$, the canonical rotational example is a slice $\H^2\times \{t_0\}$.

In the next Section \ref{monoton} we will describe a key property of these canonical rotational examples.

\subsection{Monotonicity of the angle function}\label{monoton}

Consider a rotational graph $x_3=u(\rho)$, $\rho:=\sqrt{x_1^2+x_2^2}$, in $\Rk$, and let $\nu(\rho)$ denote its angle function. We will always use the \emph{upwards orientation} on these graphs, so that $\nu$ is positive. By \eqref{res1}, we have 
 \begin{equation}\label{eqan1}
 \frac{16}{\nu(\rho)^2} = u'(\rho)^2 (4+\rho^2 \kappa)^2 +16 (1+\tau^2 \rho^2).
 \end{equation}

The next lemma considers a \emph{cone-type surface} with constant angle function in $M=\Ek$ that will be useful for comparison purposes (see Lemma \ref{monan}). The proof follows after elementary computations using \eqref{res2}, \eqref{res3} and \eqref{eqan1} that we omit. Note that in the limit case $\kappa=\tau=0$, i.e. when $M=\R^3$, this \emph{cone} is just a standard rotational cone in $\R^3$.

\begin{lemma}\label{conlema}
Consider on $\Rk$ the upwards-oriented rotational cone-type surface $C_{\beta}$ given by the radial graph $z=h(\sqrt{x_1^2+x_2^2})$, where 
 \begin{equation}\label{aches}
h(\rho):=\frac{4}{\beta} \int_0^{\rho} \frac{\sqrt{1-\beta^2 (1+\tau^2 t^2)}}{4+\kappa t^2} dt \ +c,\end{equation} $\beta\in (-1,1)$, $\beta\neq 0$, and $c\in \R$ is an arbitrary integration constant. The function $h$ is defined for all positive values of $\rho$ that satisfy the following additional restrictions: $\rho <2/\sqrt{-\kappa}$ if $\kappa<0$, and $\rho<\frac{\sqrt{1-\beta^2}}{|\beta \, \tau|}$ if $\tau\neq 0$. Then:
 \begin{enumerate}
 \item
The angle function of $C_{\beta}$ is constant of value $|\beta|\in (0,1)$.
 \item
The extrinsic curvature of $C_{\beta}$ is constant of value $-\tau^2$.
 \item
The mean curvature $H=H(\rho)$ of $C_{\beta}$ is strictly decreasing (resp. increasing) if $\beta>0$ (resp. $\beta<0$).
 \end{enumerate}
\end{lemma}

We will call these surfaces $C_{\beta}$ \emph{cones in $\Ek$}. Note that from the last two items, the principal curvatures $\kappa_1(\rho),\kappa_2(\rho)$ of $C_{\beta}$ are also decreasing (resp. increasing) if $\beta>0$ (resp. $\beta<0$), and if $\tau\neq 0$ they are actually \emph{strictly} decreasing (resp. increasing). If $\tau=0$, one of the principal curvatures is zero and the other one is strictly monotonic. The cones $C_{\beta}$ and $C_{-\beta}$ differ by an orientation preserving isometry of $\Ek$ and a change of orientation. 

\begin{figure}[h]
\begin{center}
\centering
\includegraphics[height=5cm]{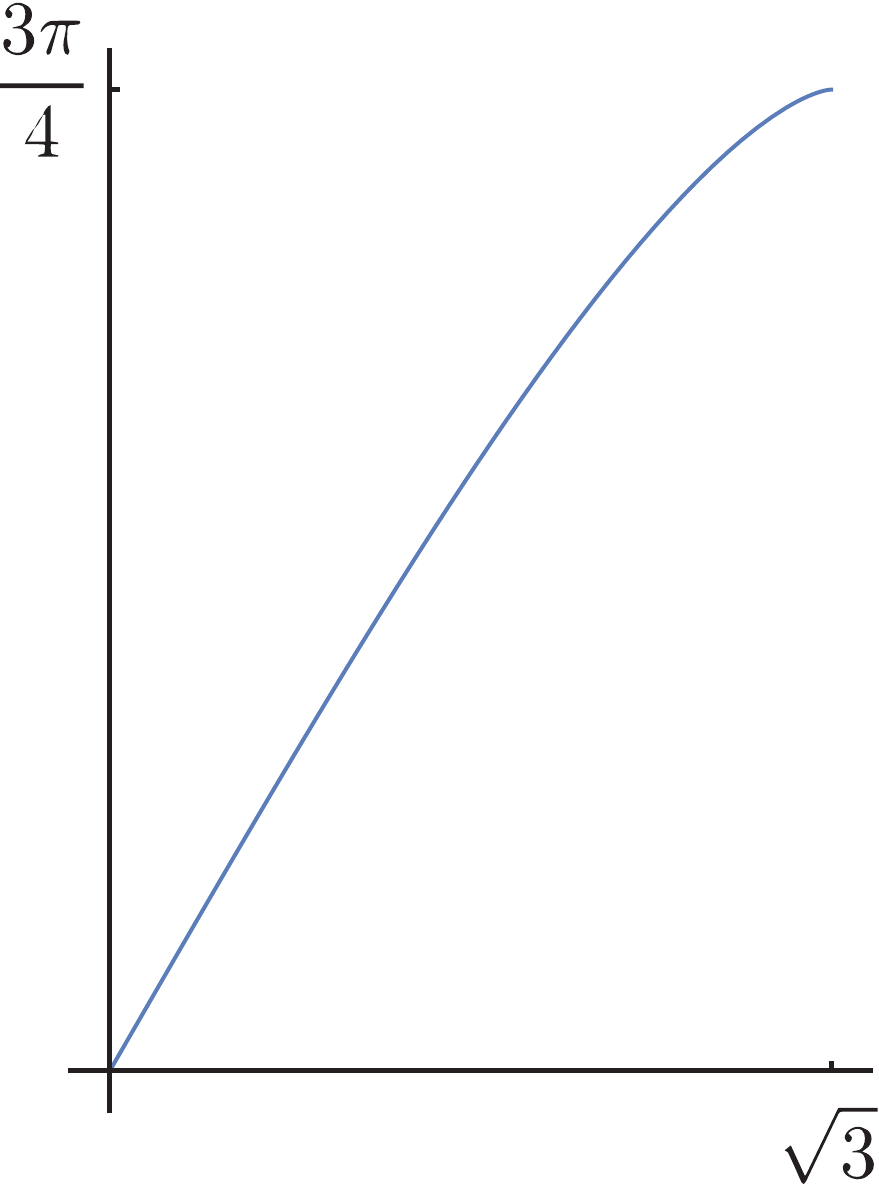} \hspace{2cm}
\includegraphics[height=5.5cm]{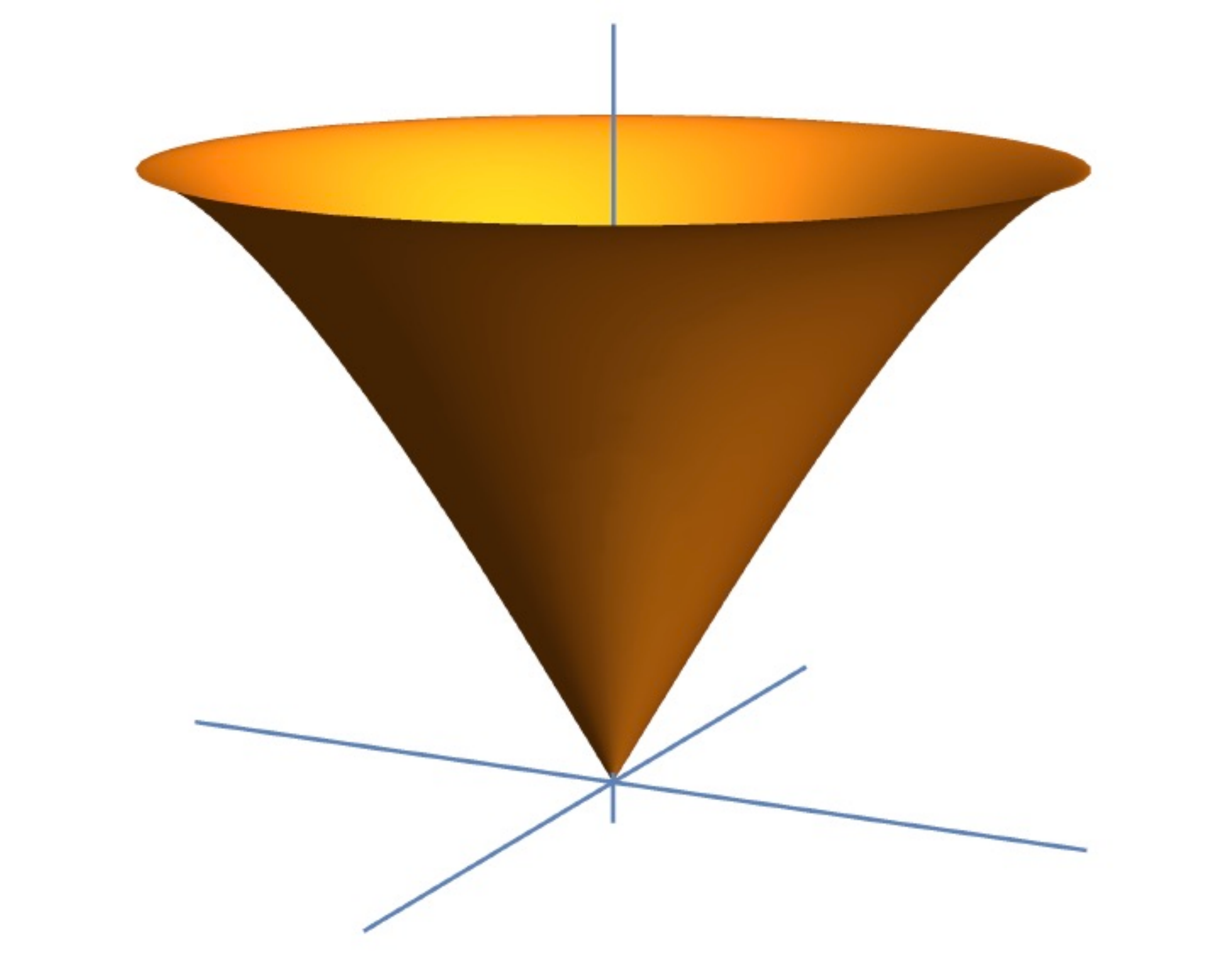}
\caption{Profile curve and picture of the cone $C_{\beta}$, $\beta=1$, in ${\rm Nil}_3=\mathbb{E}^3(0,1)$.}
\end{center}
\end{figure}

The next result is fundamental to our study, since it yields the monotonicity of the angle function of the canonical example of general Weingarten classes of surfaces in $\Ek$ spaces.

\begin{lemma}\label{monan}
Let $x_3=u(\sqrt{x_1^2+x_2^2})$ be an upwards-oriented radial graph in $\Rk$, defined on an open disk $D(0,\delta)$, that satisfies a general Weingarten equation $W(\kappa_1,\kappa_2,\nu^2)=0$ in $\Ek$. If $\tau=0$, assume that $u$ is not constant. Then:
 \begin{enumerate}
 \item
The angle function of the graph is a strictly decreasing function of $\rho:=\sqrt{x_1^2+x_2^2}$.
 \item
 If additionally $\tau =0$, then $u'(\rho)\neq 0$ for every $\rho\in (0,\delta)$.
  \end{enumerate}
\end{lemma}
\begin{proof}
Let $\nu(\rho):[0,\delta)\flecha (0,1]$ be the angle function of the (upwards oriented) rotational graph $x_3=u(\rho)$, which will be denoted by $\Sigma$. Note that $\nu(0)=1$. To prove the first assertion it suffices to show that $\nu'(\rho)\leq 0$ for every $\rho$, and that $\nu'$ cannot vanish identically on a subinterval of $[0,\delta)$.  

We will prove first of all that $\nu$ cannot be constant in an interval. Suppose, on the contrary, that $\nu=\beta\in (0,1]$ in some $[a,b]\subset (0,\rho)$. If $\beta=1$, then by \eqref{eqan1} we have $\tau=0$ and $u'(\rho)=0$ in $[a,b]$. Since $x_3=u(\rho)$ solves $W(\kappa_1,\kappa_2,\nu^2)=0$, this implies that $W(0,0,1)=0$, and so, by uniqueness in Lemma \ref{empezar}, the function $u$ is constant on $[0,\delta)$. This contradicts our hypothesis. Thus, $\beta\in (0,1)$. Then, by equation \eqref{eqan1}, it is clear that $u$ restricted to $[a,b]$ is a piece of a cone-type surface \eqref{aches}. In particular, by the monotony properties of the principal curvatures of cone-type surfaces explained after Lemma \ref{conlema}, we have that, for $\rho\in [a,b]$,


\begin{enumerate}
\item
If $\tau=0$, then a principal curvature of $\Sigma$ is zero and the other one is strictly monotonic in terms of $\rho$.
 \item
If $\tau\neq 0$, then both principal curvatures are either strictly increasing or strictly decreasing in terms of $\rho$.
\end{enumerate}
This is a contradiction with the fact that $\Sigma$ has constant angle function on $[a,b]$ and satisfies an equation of the form $W(\kappa_1,\kappa_2,\nu^2)=0$, with $W_{k_1} W_{k_2} >0$ if $W(k_1,k_2)=0$.

Thus, $\nu$ cannot be a non-zero constant on an interval in $[0,\delta)$. We next prove that $\nu'\leq 0$, what would finish the proof of item (1) of Lemma \ref{monan}.

By the non-constancy of $\nu$ on any subinterval of $(0,\rho)$, and the condition $\nu(0)=1$, it is clear that there is some $\rho_1^*>0$ arbitrarily close to zero such that $\nu'(\rho_1^*)<0$. Arguing by contradiction with $\nu'\leq 0$, assume that there exists some $\rho_2^*\in (0,\delta)$ with $\nu'(\rho_2^*)>0$. Note that we may assume $\rho_1^*<\rho_2^*$. Thus, there exists $\varrho^*$ such that $\nu$ attains its minimum value in $[\rho_1^*,\rho_2^*]$ at $\varrho^*$. As $\nu'(\rho_1^*)<0$ and $\nu'(\rho_2^*)>0$, $\varrho^*$ lies in $(\rho_1^*,\rho_2^*)$, and $\nu'(\varrho^*)=0$. Also, since $\nu$ is not constant on any subinterval, by generic transversality we deduce the existence of numbers $\rho_1<\rho_2\in [\rho_1^*,\rho_2^*]$ arbitrarily close to $\varrho^*$ where $\nu(\rho_1)=\nu(\rho_2)$, and additionally $\nu'(\rho_1)<0$ and $\nu'(\rho_2)>0$.

By \eqref{eqan1}, the condition $\nu'(\rho)=0$ (resp. $>0, < 0$) is written as 
 \begin{equation}\label{eqan2}
u'(\rho) u''(\rho) = \frac{-16 \rho \tau^2 - 2\kappa \rho u'(\rho)^2 (4+\rho^2 \kappa)}{(4+\rho^2 \kappa)^2} \hspace{1cm} \text{(resp. $< \, , \,  >$)}.
 \end{equation}
By \eqref{eqan1} and \eqref{eqan2} we can easily deduce from the fact that $\nu$ has a local minimum at $\varrho^*$ that $u'(\varrho^*)\neq 0$. We will assume that $u'(\varrho^*)>0$ (and thus $u'(\rho_1)$ and $u'(\rho_2)$ can be chosen to be positive); if $u'(\varrho^*)<0$ the argument is analogous changing $\beta$ by $-\beta$ in what follows.

Define $\beta:=\nu(\rho_1)=\nu(\rho_2)\in (0,1)$, and let $C_{\beta}$ be the cone of angle $\beta$ given by \eqref{aches}. By comparing \eqref{eqan1} and \eqref{aches}, and recalling that $u'(\rho_1)$ and $u'(\rho_2)$ are both positive, we can deduce that the cone $C_{\beta}$ is well defined at $\rho=\rho_1$ and $\rho=\rho_2$. Let us choose an adequate integration constant $c\in \R$ in \eqref{aches} so that $u(\rho_1)=h(\rho_1)$; note that we additionally have $u'(\rho_1)=h'(\rho_1)$ by \eqref{eqan1}. 

Since $\nu'(\rho_1)< 0$ and the angle function of $C_{\beta}$ is constant, we deduce from \eqref{eqan2} that $u''(\rho_1) >h''(\rho_1)$, i.e. the graph $x_3=h(\rho)$ is \emph{below} the graph $x_3=u(\rho)$ at $q_1:=(\rho_1,0,u(\rho_1))$.  Thus, denoting by $\kappa_1\leq \kappa_2$ and $\kappa_1^c\leq \kappa_2^c$ the principal curvatures of $x_3 =u(\rho)$ and $x_3= h(\rho)$, we have
 \begin{equation}\label{coc1}
 \kappa_i (\rho_1) \geq \kappa_i^c (\rho_1), \hspace{1cm} i=1,2.
 \end{equation}

The same argument at $\rho=\rho_2$ using this time that $\nu'(\rho_2)> 0$ shows that, at $q_2:=(\rho_2,0,u(\rho_2))$, the graph $x_3=h(\rho)$ of $C_{\beta}$ is \emph{above} the graph $x_3=u(\rho)$ of $\Sigma$, and hence
  \begin{equation}\label{coc2}
 \kappa_i (\rho_2) \leq \kappa_i^c (\rho_2), \hspace{1cm} i=1,2.
 \end{equation} 
 
 By \eqref{coc1}, \eqref{coc2}, and the monotonicity of the principal curvatures of cone-type surfaces explained after Lemma \ref{conlema}, we conclude then that $\kappa_i(\rho_1)\geq \kappa_i (\rho_2)$ for $i=1,2$, and that at least one of these two inequalities is strict. Since $\nu(\rho_1)=\nu(\rho_2)$ we obtain a contradiction with the fact that $W(\kappa_1,\kappa_2,\nu^2)=0$ on $\Sigma$, for $W\in C^{\8} (\R^2)$ satisfying $W_{k_1} W_{k_2} >0$ on $W^{-1}(0)\subset \R^2$.  

This contradiction shows that $\nu'(\rho)\leq 0$ for every $\rho\in [0,\delta)$, and so, finally, that $\nu(\rho)$ is strictly monotonic, what proves the first assertion in Lemma \ref{monan}.

To prove assertion (2) of Lemma \ref{monan}, we first note that from \eqref{eqan2} and the fact that $\nu(\rho)$ is strictly decreasing we have for every $\rho\in (0,\delta)$
 \begin{equation}\label{eqan5}
 \frac{d}{d\rho} (u'(\rho)^2 (4+\kappa \rho^2)^2) \geq  -32\tau^2 \rho.
 \end{equation}
 Also by \eqref{eqan2}, the height $h(\rho)$ of any cone surface \eqref{aches} satisfies, in the domain of definition of $h(\rho)$ given by Lemma \ref{monan}, that
 \begin{equation}\label{eqan6}
 \frac{d}{d\rho} (h'(\rho)^2 (4+\kappa \rho^2)^2) = -32\tau^2 \rho.
 \end{equation}
From \eqref{eqan5} and \eqref{eqan6} we have 
 \begin{equation}\label{eqan7}
 \frac{d}{d\rho} (u'(\rho)^2 (4+\kappa \rho^2)^2) \geq  \frac{d}{d\rho} (h'(\rho)^2 (4+\kappa \rho^2)^2),
 \end{equation}
on any interval $(\rho_0-\varepsilon,\rho_0+\varepsilon)$ where both $u$ and $h$ are defined. Moreover, in case $u'(\rho_0)^2\geq h'(\rho_0)^2$, integrating \eqref{eqan7} for $\rho\geq \rho_0$ we obtain 
 \begin{equation}\label{eqan77}
 u'(\rho)^2 \geq h'(\rho)^2
 \end{equation}
for every $\rho\in [\rho_0,\varepsilon)$. 
 
 Assume next that there exists $\rho_0>0$ such that $u'(\rho_0)\neq 0$. Define $\beta:= \pm \nu(\rho_0)$, with the sign being the one of $u'(\rho_0)$. It can be proved as we did above that $h(\rho)$ in \eqref{aches} is well defined on $\rho_0$ for that particular value of $\beta$. Also, 
we have $u'(\rho_0)=h'(\rho_0)$ by \eqref{eqan1}.  As $h'(\rho)\neq 0$ for every $\rho$ and $h(\rho)$ is defined for every $\rho>0$ if $\tau=0$ (see Lemma \ref{conlema}), we conclude from \eqref{eqan77} that $u'(\rho)>0$ (resp. $u'(\rho)<0$) for every $\rho\in [\rho_0,\delta)$ if $u'(\rho_0)>0$ (resp. $u'(\rho_0)<0$).
 
Recalling that $\nu$ cannot be constant on any subinterval, this clearly implies that if $\tau=0$, then $u'(\rho)\neq 0$ on $(0,\delta)$. 
\end{proof}

\section{Proof of Theorem \ref{mainth}}\label{secepru}

\subsection{Extension properties of the canonical example}\label{sub:ext}

In this section we explain how to extend the radial graph of Lemma \ref{monan} to the inextendible \emph{canonical rotational example}, and we analyze its asymptotic geometry. That this radial graph exists for $\delta>0$ small enough follows from Lemma \ref{empezar}.

Let $\cW$ denote a general Weingarten class in $\Ek$, let $x_3=u(\rho)$, $\rho=\sqrt{x_1^2+x_2^2}$, be the radial graph of Lemma \ref{monan} associated to $\cW$, and let $F$ denote the defining function of the class $\cW$. Thus, $u$ is a radial solution to $F[u]=0$, and hence, a solution to \eqref{edoful}. As explained before Definition \ref{def:rote}, standard ODE theory ensures that we can uniquely extend $u(\rho)$ to a maximal solution to \eqref{edoful}. Let us be more specific about this matter.

Let us write $u(\rho):[0,L)\flecha \R$, and note that $u(0)=u'(0)=0$. Let us denote $L_{\8}:=\8$ if $\kappa\geq 0$, and $L_{\8}:=2/\sqrt{-\kappa}$ if $\kappa<0$, and observe that $L\leq L_{\8}$. Assume for the moment that $L<L_{\8}$. Suppose that there exists a sequence $\rho_n\to L$ such that $|u(\rho_n)|+|u'(\rho_n)|+|u''(\rho_n)|$ is uniformly bounded. Up to a subsequence, we may assume that $(\rho_n,u(\rho_n),u'(\rho_n),u''(\rho_n))$ converges to some $(L,u_0,p_0,r_0)\in \R^4$. Note that we clearly have $F(L,0,u_0,p_0,0,r_0,0,p_0/L)=0$ by continuity, so we conclude from the ellipticity of $F$ on $F^{-1}(0)$ that $F_r\neq 0$ at that point. This implies that we can view \eqref{edoful} around this point in normal form, i.e. as in \eqref{edonor1}. Since $u$ is a solution to \eqref{edoful}, a standard ODE argument proves then that $u$ can be extended to $[0,L+\delta)$ for some $\delta>0$. From here and standard continuation arguments, we conclude that one of the four situations below happen:

 \begin{enumerate}
\item
$L=L_{\8}$.
 \item
$L<L_{\8}$ and there exist $\rho_n\to L$ such that $|u(\rho_n)|\to \8$.
 \item
$L<L_{\8}$ and there exist $\rho_n\to L$ such that $|u'(\rho_n)|\to \8$.
\item
$L<L_{\8}$ and there exist $\rho_n\to L$ with $|u(\rho_n)|+|u'(\rho_n)|$ uniformly bounded, such $|u''(\rho_n)|\to \8$. \end{enumerate}

%
%
%

The first situation corresponds to the case where $x_3=u(\rho)$, $\rho=\sqrt{x_1^2+x_2^2}$, is an entire rotational graph in $\Rk$. In the fourth one, the norm of the second fundamental form of the canonical rotational example blows up as $\rho \to L$.


Assume next that $u(\rho)$ is in the conditions of the situations (2) or (3) above. Observe that (2) is a particular case of situation (3), by the mean value theorem. We will assume for definiteness that $u'(\rho_n)\to \8$, and will prove first of all that there exists $\lim_{\rho\to L} u'(\rho) =\8$; an analogous argument would prove that $\lim_{\rho\to L} u'(\rho) =-\8$ if $u'(\rho_n)\to -\8$.

Take $K>0$ arbitrarily large, and let $\beta>0$ be small enough so that $L<\frac{\sqrt{1-\beta^2}}{\beta |\tau|}$; if $\tau=0$, we may choose any $\beta>0$. This implies that $h(L)=:h_{\beta}(L)$ is well defined, where $h$ is given by \eqref{aches}. Also, $h_{\beta}'(L)\to \8$ as $\beta\to 0^+$, so by choosing a smaller $\beta>0$ if necessary we can also assume that $h_{\beta}'(L)>K$.

Take now $\rho^*\in (0,L)$ such that $h_{\beta}'(\rho)>K$ for every $\rho\in [\rho^*,L]$. Since $u'(\rho_n)\to \8$, there is some $\rho^0\in [\rho^*,L]$ with $u'(\rho^0)>h_{\beta}'(\rho^0)$. Then, \eqref{eqan7} holds around $\rho^0$. By \eqref{eqan77} we get $u'(\rho) \geq h_{\beta}'(\rho)\geq K$ for every $\rho \in (\rho^0,L)$. This proves that $\lim_{\rho\to L} u'(\rho) =\8$, as claimed.

Since $L<\8$, we see then by \eqref{eqan1} that the angle function $\nu(\rho)$ tends to zero and the tangent planes to the graph $x_3=u(\rho)$ become asymptotically vertical as $\rho\to L$. 

Define $\lim_{\rho\to L} u(\rho) =:u_0 \in \R\cup \{\8\}$, which exists as a consequence of the previous discussion, and let $\psi(u,\theta)=(\rho(u) \cos \theta, \rho(u) \sin \theta, u)$ be a parametrization of the graph $x_3=u(\rho)$ in $\Rk$. Note that $\rho(u)\to L$ and $\rho'(u)\to 0$ as $u\to u_0$. For $R_{\varepsilon}:= (u_0-\varepsilon,u_0)\times (-\varepsilon,\varepsilon)$ with $\varepsilon >0$ small enough, $\psi(R_{\varepsilon})$ is a graph $x_1=f(x_2,x_3)$. Then, on $R_{\varepsilon}$,
$$f_{x_2}  =-\frac{\sin \theta}{\cos \theta}, \hspace{0.3cm} f_{x_3}= \frac{\rho'}{\cos\theta} ,\hspace{0.3cm} f_{x_2x_2} = \frac{-1}{\rho \cos^3 \theta}, \hspace{0.3cm} f_{x_2 x_3} = \frac{\rho' \sin \theta}{\rho \cos^3 \theta},\hspace{0.3cm}$$ and $$ f_{x_3x_3} =\frac{\rho''}{\cos\theta} -\frac{\rho'^2 \sin^2 \theta}{\rho \cos^3 \theta},$$ with all derivatives of $f$ evaluated at $(x_2,x_3)=(\rho(u)\sin \theta, u)$.

Since $\cW$ is a general Weingarten class of surfaces, $f(x_2,x_3)$ satisfies the elliptic PDE 
 \begin{equation}\label{newt1}
 G(f,x_2,f_{x_2},f_{x_3},f_{x_2x_2},f_{x_2x_3},f_{x_3x_3})=0, \end{equation} 
where $G$ is the defining function of $\cW$. By making $\theta=0$ and using the previous formulas, this PDE turns into the ODE for $\rho=\rho(u)$ 
  \begin{equation}\label{newt2}
 G(\rho,0,0,\rho',\frac{-1}{\rho},0,\rho'')=0,
 \end{equation} 
where we have used that $f(u,0)=\rho(u)$ and we are writing $G=G(x,y,p,q,r,s,t)$. 

If $\rho''(u)$ is unbounded as $u\to u_0$, the norm of the second fundamental form of the graph $x_3=u(\rho)$ blows up as $\rho\to L$.

Assume next that $\rho''(u)$ is bounded as $u\to u_0$. Thus, there is a sequence $\hat{u}_n\to u_0$ such that $\rho''(\hat{u}_n)$ converges to some $t_0\in \R$. Since $G_t\neq 0$ on $G^{-1}(0)$, we may write equation \eqref{newt2} around the point $(L,0,0,0,-1/L,0,t_0)$ as an ODE in normal form
 \begin{equation}\label{edof}
\rho''=\cG(\rho,\rho')
 \end{equation}
for some smooth function $\cG$ defined on an open neighborhood of $(L,0)$ in $\R^2$. Let us point out that since $G$ satisfies the symmetric condition \eqref{simh}, the function $\cG$ satisfies $\cG(x,y)=\cG(x,-y)$. We distinguish now two cases:

\vspace{0.1cm}

{\bf Case 1:} $u_0\in \R$ (i.e. $u_0\neq \8$). In that case, $\rho(u)\in C^\8 ([0,u_0))$ extends $C^1$ to the value $u=u_0$, with $\rho(u_0)=L$ and $\rho'(u_0)=0$. By the uniqueness of the solution to the Cauchy problem for \eqref{edof} and the previously mentioned symmetry of $\cG(x,y)$, $\rho(u)$ extends smoothly across $u_0$ so that it is defined in $[0,2 u_0]$, following the symmetric condition $$\rho(u_0+u)=\rho(u_0-u)$$ for every $u\in [-u_0,0]$. This proves that the radial graph $x_3=u(\sqrt{x_1^2+x_2^2})$ we started with extends in this situation to an immersed rotational sphere in $\Ek$, which is an element of $\cW$. In other words, the canonical rotational example of $\cW$ is, in this case, a sphere $S$.

Moreover, $S$ is a rotational symmetric \emph{bi-graph}, in the sense that it can be decomposed as $S=S_1\cup S_2$ with $\parc S_1 =\parc S_2$, so that:

\begin{enumerate}
\item
Both $S_1,S_2$ are compact rotational graphs in $\Ek$ diffeomorphic to a closed disk, with the same rotation axis and the same boundary curve (the orbit of a point under the rotational group around the axis). In particular, both $S_1$ and $S_2$ are embedded, but the interiors of $S_1$ and $S_2$ might intersect if $\Ek$ is diffeomorphic to $\S^3$.
 \item
$S_1$ and $S_2$ are congruent; specifically, the $180º$-rotation around any horizontal geodesic in $\Ek$ orthogonal to the rotation axis and that passes through their common rotational boundary takes $S_1$ into $S_2$ and vice versa.
\end{enumerate}
Moreover, by Lemma \ref{monan}, the angle function $\nu(\rho)$ of $S_1$ is a strictly decreasing function with respect to $\rho$, and takes all values in $[0,1]$. In the same way, the angle function $\nu(\rho)$ for $S_2$ is strictly increasing with respect to $\rho$, and takes all values in $[-1,0]$. This shows that, after parametrizing the profile curve of $S$ in a regular way as a map $\alfa(t):[a,b]\flecha \Ek$, with $\alfa(a),\alfa(b)$ being the \emph{north} and \emph{south} poles of the sphere $S$, the angle function of $S$ can be viewed as a bijective map between $[a,b]$ and $[-1,1]$. 


\vspace{0.1cm}

{\bf Case 2:} $u_0=\8$. In this case, $\rho:[0,\8)\flecha \R$ satisfies that $\rho(u)\to L>0$ and $\rho'(u)\to 0$ as $u\to \8$. Thus, there exists some $K>0$ such that, for every $u>K$, $(\rho(u),\rho'(u))$ lies  in the domain of definition of the function $\cG(x,y)$ appearing in \eqref{edof}. In particular, $\rho(u)$ is a solution to \eqref{edof} for $u>K$.

Denote now
 $\rho_{\landa} (u):= \rho(u+\landa)$ for $\landa\in (K,\8)$. Clearly, each $\rho_{\landa}$ is also a solution to \eqref{edof} in some interval of the form $[-\ep_{\landa},\8)$, $\ep_{\landa}>0$, and $$\lim_{\landa\to \8} (\rho_{\landa} (0),\rho'_{\landa} (0))=(L,0).$$ By regularity of ODEs with respect to initial conditions, this shows that there exists $\lim_{\landa\to \8} \rho_{\landa}''(0)=\cG(L,0)$. This implies that $\rho''(u)\to 0$ as $u\to \8$ and that $\cG(L,0)=0$. Hence, the constant function $\rho\equiv L$ is a solution to \eqref{edof}, and $\rho_{\landa}(u)$ converges smoothly on compact sets to this constant $L$ as $\landa \to \8$, again by regularity of ODEs.
 
Geometrically, this means that the radial graph $x_3=u(\sqrt{x_1^2+x_2^2})$ we started with defines in this situation to a complete (non-entire) rotational graph $S$ that converges asymptotically in the $C^\8$ topology to a cylinder $x_1^2+x_2^2=R^2$ in $\Rk$. This cylinder corresponds to the lift $C=\pi^{-1}(\gamma)$ in $\Ek$ of a circle $\gamma$ in $\mathbb{M}^2(\kappa)$; we note that $C$ has the topology of a cylinder (resp. of a torus) if $\Ek$ is non-compact (resp. compact). In this situation, both $S$ and $C$ are elements of the general Weingarten class $\cW$, and $S$ is actually the canonical rotational example of $\cG$.

We summarize all the previous discussion in the following proposition:

\begin{proposition}\label{posrot}
Let $\cW$ be a general Weingarten class in $M=\Ek$, and let $S$ denote the canonical rotational example of $\cW$. Let $\Rk$ denote a canonical coordinate model for $M$, so that the rotation axis of $S$ is the $x_3$-axis in these $(x_1,x_2,x_3)$-coordinates. 

Then, one of the following four situations holds for $S$ in $\Rk$:
 \begin{enumerate}
 \item
$S$ is an entire graph in $\Rk$.
 \item
 $S$ is a rotational sphere contained in $\Rk$.
  \item
 $S$ is a proper rotational graph $x_3=u(x_1,x_2)$ over some bounded open disk $D_R=\{x_1^2+x_2^2 < R^2\}$, and it is smoothly asymptotic to the cylinder $x_1^2 + x_2^2 = R^2$ in $\Rk$.
   \item
$S$ is a rotational graph over a bounded open disk $D_R$ in $\Rk$, and its second fundamental form is unbounded.
 \end{enumerate}
\end{proposition}

\subsection{Proof of Theorem \ref{mainth} when $M=\Ek$, with $\kappa\leq 0$}\label{sec:r3}

Note that in this case, $M$ is diffeomorphic to $\R^3$, and in particular we can identify $\Ek=\Rk$. Let $\cW$ be a general Weingarten class of surfaces in $M$, and denote by $S$ its canonical rotational example. By hypothesis, the second fundamental form of $S$ is bounded. Thus, by Proposition \ref{posrot}, $S$ is either: $i)$ an entire rotational graph, $ii)$ a rotational sphere, or $iii)$ a complete, non-entire rotational graph $C^\8$-asymptotic to a vertical cylinder.

If $S$ is an entire rotational graph, it is immediate by the maximum principle and the invariance of the class $\cW$ by vertical translations of $M$ that there are no compact surfaces in the class $\cW$; in particular, there are no immersed spheres.

Assume now that $S$ is a rotational sphere. By our study in Subsection \ref{sub:ext}, we know that $S$ is an embedded symmetric bi-graph, and that its angle function, seen as a map defined in terms of a regular parameter of the profile curve of $S$, is bijective into $[-1,1]$. By Lemma \ref{gaussmap}, this means that the family $$\cS:=\{\Psi (S) : \Psi \in {\rm Iso}^0(M)\}$$ is a transitive family of surfaces in $M$. Recall that ${\rm Iso}^0(M)$ stands for the orientation preserving isometries of $M$ that also preserve the unit Killing field $\xi$.

Once here we can use the authors' previous work \cite{GM3}. Note that all elements of $\cS$ belong to the general Weingarten class $\cW$\, since $\cW$ is closed by orientation preserving ambient isometries. As we explained in Section \ref{sec:wedp}, any general Weingarten class of surfaces $\cW$ in $M$ is locally modeled by an (absolutely) elliptic PDE around each point in $M$ and each direction in the tangent bundle, once we fix coordinates in the space. In particular, $\cW$ is a \emph{class of surfaces modeled by an elliptic PDE} as introduced in Definition 2.3 of \cite{GM3}. 

In these conditions, we can use Theorem 2.4 in \cite{GM3} to deduce that if $\Sigma$ is an immersed sphere in $M$ that belongs to the general Weingarten class $\cW$, then $\Sigma$ is an element of $\cS$, i.e. $\Sigma$ differs by an ambient isometry $\Psi\in {\rm Iso}^0(M)$ of the canonical rotational sphere $S$. 

To finish, let us consider now the third possibility for the canonical rotational example $S$, i.e. the case that $S$ is a complete rotational graph that is $C^\8$-asymptotic to a rotational vertical cylinder. Let $C$ denote this rotational vertical cylinder in $\Ek$; note that the angle function of $C$ is identically zero. Let $S'$ denote the $180º$-rotation of $S$ around a horizontal geodesic of $\Ek$ passing through the origin; note that $S'$ is a downwards-oriented graph in $M$, asymptotic to $C$. Finally, let define the family $\cS$ as $$\cS=\{\Psi(S), \Psi(S'), \Psi(C) : \Psi\in {\rm Iso}^0(M)\}.$$ As in the previous case, every element of $\cS$ belongs to the general Weingarten class $\cW$, which is, as explained above, a class of surfaces modeled by an elliptic PDE. So, if we prove that $\cS$ is a transitive family of surfaces in $\Ek$, we can use again Theorem 2.4 in \cite{GM3} to deduce that any immersed sphere of the class $\cW$ is an element of the family $\cS$. As this time the family $\cS$ does not contain immersed spheres, this means that in the present case \emph{there are no immersed spheres in the class $\cW$}.

To prove that $\cS$ is a transitive family, we consider for every $\Sigma'\in \cS$ its Legendrian lift $\cL_{\Sigma'}$ that sends each $q\in \Sigma'$ to the pair $(q,N(q))\in TU(M)$, where $N(q)$ is the unit normal of $\Sigma'$ at $q$. Define the family of lifts $\cF:=\{\cL_{\Sigma'}: \Sigma'\in \cS\}$, all of which are regular surfaces in $TU(M)$. Note the following properties:

 \begin{enumerate}
 \item
\emph{For every $(p,w)\in TU(M)$ there exists some $\cL_{\Sigma'}\in \cF$ such that $(p,w)\in \cL_{\Sigma'}$.} Indeed, let $\nu_0\in [-1,1]$ denote the inclination of the tangent plane associated to $(p,w)$. Note that the angle function of $S$ takes all values in $(0,1]$, the angle function of $S'$ takes all values in $[-1,0)$, and the angle function of $C$ is zero. Thus, there is an element $\Sigma_0\in \{S,S',C\}$ whose angle function at some point $q$ is $\nu_0$. Lemma \ref{incli} implies then that there exists an isometry $\Psi\in {\rm Iso}^0(M)$ that takes $(q,N_{\Sigma_0}(q))$ to $(p,w)$. The claim follows then immediately.
 \item
\emph{If $\cL(\Sigma_1)=\cL(\Sigma_2)$ at some point $(p,v)\in TU(M)$ for some $\Sigma_1,\Sigma_2\in \cS$, then $\cL(\Sigma_1)=\cL(\Sigma_2)$ at every point.} Indeed, in these conditions there is a unique element $\Sigma_0\in \{S,S',C\}$ such that both $\Sigma_1,\Sigma_2$ are congruent to $\Sigma_0$ by an element of ${\rm Iso}^0(M)$ (since two different elements of $\{S,S',C\}$ never have the same angle function). If $\Sigma_0\neq C$, then we can use the monotonicity of its angle function given by Lemma \ref{monan} to prove by the same arguments used in Lemma \ref{gaussmap} that if the equality $\cL(\Sigma_1)=\cL(\Sigma_2)$ holds at one point, it must hold globally, as wished. Finally, if $\Sigma_0=C$, a similar argument shows that the condition $\cL(\Sigma_1)=\cL(\Sigma_2)$ at one point implies that $\Sigma_1$ and $\Sigma_2$ differ by a \emph{vertical} translation. But since $C$ is invariant by vertical translations, we obtain again that $\cL(\Sigma_1)=\cL(\Sigma_2)$ everywhere. This proves the claim.
 \end{enumerate}

Finally, note that as the surface $S$ converges asymptotically to $C$ with $C^\8$ regularity (and thus $S'$ also converges $C^\8$-smoothly to $C$), the family $\cS$ is a smooth family of surfaces in $M$. These properties together are enough to ensure that $\cS$ is, as desired, a transitive family of surfaces in $M$ according to Definition \ref{defitransi}.

By putting all this discussion together, we come then to the following conclusions if $M=\Ek$, with $\kappa\leq 0$:

\begin{enumerate}
\item
If the canonical rotational example $S$ of a general Weingarten class of surfaces $\cW$ in $\Ek$ is not compact, then there are no immersed spheres in the class $\cW$.
 \item
If the canonical rotational example $S$ is compact, then $S$ is up to ambient isometry the only immersed sphere in the class $\cW$.
 \item
By our previous study in Section \ref{sec:canonical} we know that if the canonical rotational example $S$ is compact, then $S$ is rotational and embedded, it is a bi-graph that is symmetric with respect to the $180º$-rotation around some horizontal geodesic, and its angle function is monotonic and surjective onto $[-1,1]$ with respect to any regular parametrization of the profile curve of $S$.
\end{enumerate}
These facts together prove Theorem \ref{mainth} in the case that $M=\Ek$ with $\kappa\leq 0$.

\subsection{Proof of Theorem \ref{mainth} when $M=\Ek$, $\kappa> 0$, $\tau=0$}\label{sec:s2xr}

In this case, $M=\S^2(\kappa)\times \R$. Let $\cW$ denote a general Weingarten class of surfaces in $M$, and let $S$ be its canonical rotational example, which by hypothesis has bounded second fundamental form. Assume that the rotation axis $L$ of $S$ is the $x_3$-axis for some canonical coordinate model $\Rk$ for $M$, and recall that this model recovers $(\S^2(\kappa)\setminus \{p\}) \times \R$, where $L^*\equiv \{p\}\times \R$ is the antipodal fiber in $M$ of the rotation axis $L$ of $S$.

Assume that $S$ is not an entire graph in $\Rk$, i.e. assume that $S$ remains at a positive distance in $M$ from $L^*$. Then, arguing as in Section \ref{sec:r3}, it follows that:
 \begin{enumerate}\item If $S$ is a rotational sphere, then $S$ is an embedded bi-graph in $M$ and any other immersed sphere of the general Weingarten class $\cW$ is congruent to $S$.
  \item If $S$ is not a sphere, then there are no immersed spheres within the class $\cW$. 
  \end{enumerate}
Thus, Theorem \ref{mainth} holds in that case.
 
Next, suppose that $S$ is an entire graph in $\Rk$. If the angle function $\nu$ of $S$ is constant, then $S$ is a slice $\S^2\times \{t_0\}$, and it is immediate by the maximum principle that $S$ is (up to vertical translation) the only compact immersed surface in $M$ that belongs to the general Weingarten class $\cW$. Thus, all conclusions of Theorem \ref{mainth} hold in this case.

Finally, assume that the angle function $\nu$  of the entire graph $S$ is not constant, and let $s$ denote the parameter for the profile curve of $S$ defined before \eqref{metpv}.
By Lemma \ref{monan} and \eqref{res5}, it follows that $\rho'(0)=\nu(0)=1$ and that $\rho'(s)$ is strictly decreasing, with $\rho'(s)>0$ if $s>0$. Since $S$ is an entire graph, we see that $\rho(s)\to \8$. From these conditions it is easy to see from \eqref{resH} and \eqref{res6} that the norm of the second fundamental form of $S$, given for $\rho=\rho(s)$ by
$$|\sigma|^2=4H^2-2K_e = \frac{(4-\kappa\rho^2)^2(1-\rho'(s)^2)^2+\rho''(s)^2 \rho^2 (4+\kappa \rho^2)^2 }{16 \rho^2 (1-\rho'(s)^2)},$$  
blows up as $\rho \to \8$. This contradicts our hypothesis, what concludes the proof of Theorem \ref{mainth} when $M=\S^2(\kappa)\times \R$.

\subsection{Proof of Theorem \ref{mainth} when $M=\Ek$, $\kappa> 0$, $\tau\neq 0$}\label{sec:s3}

In this case, $M$ is diffeomorphic to $\S^3$. Let $\cW$ be a general Weingarten class in $M$, and let $\pi:M\flecha \mathbb{S}^2(\kappa)$ denote the canonical fibration. Let $S$ be the canonical rotational example in $M$ of the class $\cW$, and let $L:=\pi^{-1}(p)$ and $L^*:=\pi^{-1}(p^*)$ denote, respectively, the rotation axis of $S$ and its antipodal fiber.

Consider next a canonical coordinate system $\Rk$ in $M$ so that, in these $(x_1,x_2,x_3)$-coordinates, the $x_3$-axis corresponds to the universal covering of $L$ and $\pi(x_1,x_2,x_3)=(x_1,x_2)$. Recall that this model is not global; specifically, $\Rk$ can be identified with the Riemannian universal covering of $M\setminus L^*$. We refer to Appendix 2 for the details on this and other particularities of the geometry of $M$.

By our analysis in Subsection \ref{sub:ext}, the rotational surface $S$ in $M$ corresponds in this $\Rk$ model to a rotational surface $\hat{S}\subset \Rk$ with rotation axis the $x_3$-axis, and for which one of the following three possibilities hold (recall that, by hypothesis, the second fundamental form of $S$ is bounded):

\begin{enumerate}
\item
$\hat{S}$ is an entire radial graph in $\Rk$.
\item
$\hat{S}$ is an embedded rotational sphere in $\Rk$.
 \item
$\hat{S}$ is a radial graph in $\Rk$ defined on a disk $x_1^2+x_2^2<R^2$, and that converges $C^\8$-asymptotically to the circular cylinder $x_1^2+x_2^2 =R^2$ in $\Rk$.
\end{enumerate}

In cases (2) and (3), the arguments in Section \ref{sec:r3} still work. Specifically, these arguments together with Lemma \ref{gaussmap} prove that:
 \begin{enumerate}
 \item[i)]
In case (2) above, $S$ is a rotational sphere with strictly monotonic angle function; in particular, the family $\{\Psi(S): \Psi\in {\rm Iso}^0(M)\}$ is a transitive family of surfaces in $M$.
 \item[ii)]
In case (3) above, if we denote by $\cS^*$ the family composed by $S$, by its $180º$-rotation $S'$ with respect to some horizontal geodesic of $M$, and by the rotational surface $\pi^{-1}(\gamma)$ tangent to the vertical Killing field $\xi$ to which $S$ converges asymptotically, then $\cS:=\{\Psi(\cS^*) : \Psi\in {\rm Iso}^0(M)\}$ is again a transitive family of surfaces in $M$. Note that, in this case, $\pi^{-1}(\gamma)$ is a torus (since the fibers of $\pi$ are diffeomorphic to $\S^1$), and in the $\Rk$ model for $M$, it corresponds to some cylinder $x_1^2+x_2^2=R^2$.
  \end{enumerate}
In particular, we can deduce as in Section \ref{sec:r3} that in case (2) any immersed sphere of the general Weingarten class $\cW$ is congruent to $S$, and that in case (3) there are no immersed spheres in $\cW$.

Finally, assume that we are in case (1) above. Hence, $\hat{S}$ is an entire rotational graph in $\Rk$, which can be parametrized as $\hat{\psi}(\rho,\theta)=(\rho \cos \theta, \rho \sin \theta, h(\rho))$, for all values of $\rho>0$, and with $h(0)=h'(0)=0$. In particular, note that $S$ approaches the antipodal fiber $L^*$ of its rotation axis $L$.

\emph{For simplicity in the computations, we will assume from now on without loss of generality that the constant $\kappa>0$ is actually $\kappa=4$; this corresponds to a rescaling of the metric on $M$}.

In order to understand the behavior of $S$ as it approaches $L^*$, we consider a different coordinate system $(y_1,y_2,y_3)$ on $M$, via stereographic projection $\pi_0 : M\equiv \S^3\setminus\{p_N\} \flecha \R^3$ from a \emph{north pole} point $p_N\in L$. See equation \eqref{newc} in Appendix 2 for an explicit description of these coordinates.

Let $S_0$ denote the open piece of $S$ that lifts into $\hat{S}$. That is, $S_0=\Psi(\hat{S})$, where $$\Psi:\Rk\flecha M\setminus L^*\equiv (\S^3\setminus L^*,g)$$ is given by formula \eqref{io} in Appendix 2. Then, in these $(y_1,y_2,y_3)$-coordinates, we have by \eqref{af2} the following parametrization for $S_0$:
\begin{equation}\label{ber1}
\varphi(\rho,\theta)= \frac{1}{\sqrt{1+\rho^2}-\sin(h(\rho)/\tau)} \left(\rho \cos(\theta+h(\rho)/\tau), \rho \sin (\theta + h(\rho)/\tau ), \cos(h(\rho)/\tau)\right).
\end{equation}
Note that $\varphi([0,\8)\times [0,2\pi))$ is a rotational surface $\Sigma_0$ in $\R^3$, with profile curve $\gamma:[0,\8)\flecha \R^2$ given by \eqref{af3}, i.e.
 \begin{equation}\label{ber2}
 \gamma(\rho):=(\alfa(\rho),\beta(\rho)) = \frac{1}{\sqrt{1+\rho^2}-\sin(h(\rho)/\tau)} \left(\rho, \cos(h(\rho)/\tau)\right). \end{equation} 
Note that $\gamma(0)=(0,1)$, since $h(0)=0$. Clearly, $\gamma(\rho)\to (1,0)$ as $\rho\to \8$. Since by hypothesis $S$ has bounded second fundamental form in $M$, we have that $\gamma$ has bounded (Euclidean) curvature as a planar curve. We briefly comment now two elementary facts about regular planar curves of bounded curvature:

\vspace{0.2cm}

{\bf Fact 1:} \emph{A complete regular arc $\gamma:[0,\8)\flecha \R^2$ with bounded curvature $|\kappa_{\gamma}|\leq C<\8$ cannot converge to a point $q\in \R^2$.}

\vspace{0.1cm} \emph{Sketch of proof:} It follows from the also well-known fact that if a complete planar curve $\gamma$ has bounded curvature, then there exists a fixed $\ep_0>0$ (that only depends on the upper bound for $|\kappa_{\gamma}|$) such that, for each $p\in \gamma$, the curve $\gamma$ can be seen locally around $p$ as a graph of a function with gradient bounded in absolute value by $1$, over an interval of length $\ep_0$ of its tangent line at $p$.

\vspace{0.2cm}

{\bf Fact 2:} \emph{Let $\gamma: [0,L)\flecha \R^2$ be a planar curve of bounded curvature, parametrized by arc length, with $\lim \gamma(s) = p\in \R^2$ as $s\to L$. Then $\gamma(s)$ extends $C^1$-smoothly to $s=L$.}

\vspace{0.1cm} \emph{Sketch of proof:} Let $\vartheta(s)$ denote the angle between $\gamma'(s)$ and a fixed unit vector in $\R^2$. If $\gamma(s)$ does not extend $C^1$ to $s=L$, there are sequences $s_n,s_n^* \to L$ such that $\vartheta (s_n) \to \vartheta_1$ and $\vartheta(s_n^*)\to \vartheta_2$, with $\vartheta_2-\vartheta_1 \neq 0$. This implies by the mean value theorem that $\vartheta'(\bar{s_n})\to \8$ for a sequence $\bar{s_n}\to L$, which contradicts that $\gamma$ has bounded curvature, since $\kappa_{\gamma}(s)=\vartheta'(s)$.

\vspace{0.2cm}

We go back next to our situation regarding the regular planar curve $\gamma=\gamma(\rho)$ in \eqref{ber2}. The previous two facts ensure that $\gamma$ has finite length, and a well defined limit tangent direction as $\rho\to \8$. Let us compute next this tangent direction.

First, we note that the function $h(\rho)$ is bounded (otherwise, by \eqref{ber2}, the curve $\gamma$ winds infinitely around $(1,0)$ as $\rho\to \8$, with contradicts existence of a limit tangent direction). Thus, there exists a sequence $\rho_n\to \8$ such that $h(\rho_n)\to a\in \R$ and $h'(\rho_n)\to 0$, for some $a\in \R$. 
In these conditions, a computation from \eqref{ber2} shows that the desired limit tangent direction is given by
\begin{equation}\label{limm}
\lim_n \frac{\gamma'(\rho_n)}{|\gamma'(\rho_n)|} = \left(-\sin(a/\tau),-\cos (a/\tau) \right)=:v_a.
 \end{equation}

Consider next the rotational surface $S^*$ in $\Rk$ given by 
 \begin{equation}\label{chista}
\psi^*(\rho,\theta)=(\rho \cos \theta, \rho \sin \theta, h^*(\rho)), \hspace{1cm} h^*(\rho):= -h(\rho) +\pi \tau + 2 a,
 \end{equation}
i.e. $S^*$ is the $180º$-rotation of $\hat{S}$ around the $x_1$-axis, composed with a specific vertical translation. In particular, if $\Psi:\Rk\flecha M\setminus L^*$ is the local isometry given by \eqref{io}, it follows that $S_0:=\Psi(\hat{S})$ and $S_1:=\Psi(S^*)$ differ by an ambient isometry in $M$. Also, note that both $S_0$ and $S_1$ are rotational graphical disks in $M$ with common boundary equal to the axis $L^*$. We prove next:

\vspace{0.2cm}

{\bf Claim:} \emph{$S_0\cup L^*\cup S_1$ is a smooth rotational sphere in $M$. Thus, it is equal to the canonical rotational surface $S$.}

\vspace{0.1cm} \emph{Proof of the claim:} To prove the claim, it suffices to check that $S_0$ and $S_1$ are glued together smoothly along $L^*$ in $M$. 

First, note that if we apply the same arguments above but changing $h(\rho)$ by the function $h^*(\rho)$ in \eqref{chista}, we conclude that the surface $S_1$ can be written in the $(y_1,y_2,y_3)$-coordinates as a rotational surface $\Sigma_1$ with a profile curve $\gamma^*$, given by \eqref{ber1}, \eqref{ber2} changing $h(\rho)$ by $h^*(\rho)$, respectively. Note that $\gamma^*(\rho)\to (1,0)$ as $\rho\to \8$. A calculation similar to the one in \eqref{limm} show that the unit limit tangent vector at $\gamma^*(\8)=(1,0)$ is given by $v_a^*:=-v_a$. Therefore, the curves $\gamma$ and $\gamma^*$ can be joined $C^1$-smoothly at the point $(1,0)$. This means that their associated rotational surfaces $\Sigma_0,\Sigma_1$ in the $(y_1,y_2,y_3)$-coordinates can also be joined $C^1$-smoothly around the circle $\{(y_1,y_2,y_3): y_1^2+y_2^2=1, y_3=0\}$.

Moreover, note that both $\Sigma_0,\Sigma_1$ are solutions to the same general Weingarten equation. Thus, using computations similar to those carried out after \eqref{newt1}, their profile curves $\gamma,\gamma^*$ are solutions to the same second order ODE, when parametrized as graphs over their common tangent line. This implies that the curve $\gamma\cup \gamma^*\cup \{(1,0)\}$ is smooth at the point $(1,0)$. From here, the statement of the Claim trivially follows.

\vspace{0.2cm}

Therefore, the canonical rotational example $S$ in $\Ek$ is a smooth immersed rotational sphere. which is also a symmetric bi-graph. Moreover, by construction, the angle function of $S$ is strictly monotonic as a function of the parameter of the profile curve of $S$, and takes all values in $[-1,1]$. Thus, by
Lemma \ref{gaussmap}, the family $\{\Psi(S): \Psi\in {\rm Iso}^0(M)\}$ is a transitive family of surfaces in $M$, and we conclude as in previous sections that any immersed sphere in the class $\cW$ is equal to $S$, up to ambient isometry. This concludes the proof of Theorem \ref{mainth}.

\section{General Weingarten spheres with bounded mean curvature}\label{sec6}

The aim of this section is to use Theorem \ref{mainth} in order to prove the theorem below.

\begin{theorem}\label{bomi}

Let $\Sigma$ be an immersed sphere in $M=\Ek$ that satisfies a general Weingarten equation 
 \begin{equation}\label{weie}
 H=\Phi(H^2-K_e,\nu^2),
 \end{equation}
where $\Phi\in C^{\8}([0,\8)\times [0,1])$ verifies $a<\Phi<b$ for positive constants $a,b$.

Then, $\Sigma$ is a rotational sphere in $M$.
\end{theorem}
\begin{proof}
Let $S$ denote the canonical rotational example in $M$ associated to \eqref{weie}. If $S$ is a rotational sphere, the result follows from Theorem \ref{mainth}. So, we will assume from now on that $S$ is not a rotational sphere. By Proposition \ref{posrot}, we can write $S$ with respect to a canonical coordinate system $\Rk$ for $M$ as a radial graph $x_3=u(\rho)$, $\rho:=\sqrt{x_1^2+x_2^2}$, where the function $u$ is smooth in some interval $[0,L)$, with $L\leq 2/\sqrt{-\kappa}$ in case $\kappa<0$ (by definition of the model $\Rk$).

We start by proving:

\vspace{0.2cm}

\noindent {\bf Claim:} \emph{If $\kappa>0$, then $L<\8$}.

\vspace{0.1cm}

{\it Proof of the Claim:} Let $a>0$ be as in the statement, and let $S_a$ denote the unique (up to vertical translations) rotational sphere in $M$ with rotation axis corresponding to the $x_3$-axis in $\Rk$, and with constant mean curvature $H=a$. It is well known that $S_a$ is contained in $\Rk$, since $a>0$, and is an embedded bi-graph in this $\Rk$ model (although it might not be embedded in $M$ if $\tau \neq 0$).

Let $x_3=h_{a}(\rho)$, $\rho:=\sqrt{x_1^2+x_2^2}$, be the graph in $\Rk$ that defines the \emph{lower hemisphere} of $S_a$, defined on a disk $D(0,\rho_0)$ of radius $\rho_0>0$, and recall that $S$ is given by the radial graph $x_3=h(\rho)$. Note that $h_a'(\rho)\to \8$ as $\rho \to \rho_0$.

Since $a<\Phi$, we have 
 \begin{equation}\label{condi1}
F_a > F,
\end{equation} 
where $F_a=F_a(x,y,z,p,q,r,s,t)$ (resp. $F=F(x,y,z,p,q,r,s,t)$) is the defining function of the class of surfaces given by $H=a$ (resp. of the class of surfaces defined by \eqref{weie}); see Definition \ref{defifu}.

Assume that $h$ is defined on some larger disk $\D(0,\rho_0+\delta)$. Note that for any 
$c\in \R$, we have $F[h +c]=0=F_a[h_a]$, so by \eqref{condi1} we see that $F_a[h + c]> F_a[h_a]$ on $\D(0,\rho_0)$. Let us choose $c$ such that $h+c=h_a$ in $\parc \D(0,\rho_0)$. Then, by the comparison principle (see e.g. \cite[pg. 443]{GT}), we have $h+c\leq h_a$ on $\D(0,\rho_0)$. This implies that $h'(\rho)\to \8$ as $\rho\to \rho_0$, what contradicts the assumption that $h$ is defined on $\D(0,\rho_0+\delta)$. Thus, $h$ is defined at most on $\D(0,\rho_0)$, i.e. $L\leq \rho_0<\8$. This completes the proof of the Claim.

\vspace{0.2cm}

We now complete the proof of Theorem \ref{bomi}. If $L=2/\sqrt{-\kappa}$ when $\kappa<0$ (resp. $L=\8$ when $\kappa=0$), then $S$ is an entire graph in $M=\Ek$, which by the condition $\kappa\leq 0$ is diffeomorphic to $\R^3$. In this situation, the result is trivial, since by the maximum principle there are no compact surfaces immersed in $M$ that satisfy \eqref{weie}; see e.g. the second paragraph in Section \ref{sec:r3}.

So, from now on \emph{we will assume that $L<2/\sqrt{-\kappa}$ if $\kappa<0$, and that $L<\8$ if $\kappa=0$}. By the Claim above, we then have $L<\8$ in any $M=\Ek$.

Let $(\rho(s),0,h(s))$ denote the profile curve of $S$, parametrized with respect to the parameter $s$ defined before \eqref{metpv}. Thus, equations \eqref{res5}, \eqref{resH} and \eqref{res6} hold for $\rho(s)$, and we have the initial conditions $\rho(0)=h(0)=0$, $h'(0)=0$, $\rho'(0)=1$. Since $L<2/\sqrt{-\kappa}$ if $\kappa<0$, we see that $4+\kappa \rho(s)^2$ is greater than some positive constant, and since $L<\8$, we see that $\rho(s)$ is bounded.

As $S$ satisfies \eqref{weie}, and $a<\Phi<b$, it follows that the mean curvature of $S$ is bounded. Thus, using \eqref{resH} together with the previous boundedness properties of $\rho(s)$ and $4+\kappa \rho(s)^2$, we deduce that $\rho''(s)$ is bounded. But this implies by \eqref{res6} that the extrinsic curvature $K_e$ of $S$ is also bounded (note that $K_e$ is trivially bounded around $\rho(0)=0$, since $S$ meets its rotation axis smoothly). As a result, $S$ has bounded second fundamental form. The result follows then from Theorem \ref{mainth}, what completes the proof of Theorem \ref{bomi}.
\end{proof}

\begin{corollary}\label{prh}
Let $\Phi\in C^{\8}([0,1])$, $\Phi>0$. Then any immersed sphere in $M=\Ek$ whose mean curvature $H$ and angle function $\nu$ satisfy $$H=\Phi(\nu^2)$$ is a sphere of revolution in $M$.
\end{corollary}

Corollary \ref{prh} clearly contains (for $\Phi$ constant) the Abresch-Rosenberg classification of constant mean curvature spheres in $\Ek$ spaces (Theorem \ref{arth}), except for the case $H=0$ when $\kappa>0$. However, this particular case can also be recovered as a particular corollary of Theorem \ref{mainth}. Indeed, the canonical rotational example $S$ for $H=0$ in $\S^2(\kappa)\times \R$ is a totally geodesic slice $\S^2(\kappa)\times \{t_0\}$, while for $H=0$ in a Berger sphere $\Ek =(\S^3,g)$, $S$ is an \emph{equatorial sphere}. In any of these two cases, $S$ is compact, and the result follows from Theorem \ref{mainth}.

\section{Solution to Minkowski-type problems in $\Ek$}\label{sec7}

The classical Minkowski problem asks, given $\mathcal{K}\in C^{\8}(\S^n)$, $\mathcal{K}>0$, to determine existence and uniqueness of a compact hypersurface $S\subset \R^{n+1}$ such that its extrinsic curvature $K_e$ (i.e. the product of its principal curvatures) is given  by $K_e = \mathcal{K} \circ \eta$, where $\eta:S\flecha \S^n$ is the Gauss map of $S$. In the case that $\mathcal{K}$ is an even function that is rotationally symmetric with respect to the $x_{n+1}$-axis, this equation is written as
 \begin{equation}\label{mi1}
 K_e = \Phi(\nu^2)>0,
 \end{equation}  
where $\nu=\esiz \eta,e_{n+1}\esde$ is the angle function of $S$ in the vertical direction, and $\Phi\in C^{\8}([0,1])$.

Equation \eqref{mi1} also makes sense in any homogeneous $\Ek$ space, what leads to the consideration of a natural Minkowski-type problem in this context. Note that \eqref{mi1} can be seen as a general elliptic Weingarten equation, and so Theorem \ref{mainth} can be applied to it. Our objective in this section will be to prove existence, uniqueness and rotational symmetry of solutions to \eqref{mi1} in any $\Ek$ space.

\subsection{Classification of immersed spheres of constant positive extrinsic curvature}\label{kec}

The classical Liebmann theorem states that any immersed sphere of constant positive extrinsic curvature in $\R^3$, $\H^3$ or $\S^3$ is a round sphere. Liebmann's theorem was extended by Espinar, Gálvez and Rosenberg \cite{EGR} to $\H^2\times \R$ and $\S^2\times \R$. In the theorem below we extend this result to any rotationally symmetric homogeneous three-manifold $\Ek$.

\begin{theorem}\label{posex}
For every $c>0$ there exists a rotational sphere $S$ in $M=\Ek$ with constant extrinsic curvature $K_e=c$ (embedded if $M$ is not compact), and any other immersed sphere with $K_e=c$ in $M$ is equal to $S$, up to ambient isometry.

In particular, any immersed sphere of constant positive extrinsic curvature in $M$ is a rotational sphere.
\end{theorem}
\begin{proof}
Let $S$ be the canonical rotational example in $M=\Ek$ associated to the elliptic Weingarten equation $K_e =c>0$. We will prove that $S$ is a sphere. This provides the existence in the statement of Theorem \ref{posex}, while the uniqueness up to ambient isometry and embeddedness properties follow directly from Theorem \ref{mainth}.

Assume that the rotation axis of $S$ in a standard coordinate model $\Rk$ corresponds to the $x_3$-axis. Let $(\rho(s),0,h(s))$ denote the profile curve of $S$, parametrized with respect to the parameter $s$ defined before  \eqref{metpv}. Thus, equations \eqref{res5} and \eqref{res6} hold for $\rho(s)$. Moreover, we have the initial conditions $\rho(0)=h(0)=0$, $h'(0)=0$, $\rho'(0)=1$. Also, note that $4+\kappa \rho(s)^2>0$ for all $s$, because of the definition of $\mathcal{R}^3 (\kappa,\tau)$. It follows from \eqref{res6} that
 \begin{equation}\label{ress4}
 \rho''(s)=\frac{16 \rho(s) (c+\tau^2)}{(4+\kappa \rho(s)^2)^2 (-4+ (\kappa-8\tau^2)\rho(s)^2)}.
 \end{equation}
Thus, if we write $x(s):=\rho(s)$, $y(s):=\rho'(s)=\nu(s)$, we can see from \eqref{ress4} that $(x(s),y(s))$ is an orbit of the autonomous ODE system
\begin{equation}\label{ress5}
\left. \def\arraystretch{1.5}\begin{array}{lll} x' & = & y \\ y' & = & G_c(x) \end{array} \right\}, \hspace{0.5cm} G_c(x):=\frac{16 x (c+\tau^2)}{(4+\kappa x^2)^2(-4+(\kappa-8\tau^2) x^2)}.\end{equation}

Note that $(x(0),y(0))=(0,1)$. Let $\theta(s)$ denote the restriction of $(x(s),y(s))$ to a maximal interval $I\subset [0,\8)$, with $0\in I$, and such that $y(s)>0$ for every $s\in I$. It then follows from Lemma \ref{monan} and \eqref{res5} (or directly from \eqref{ress5}) that  $y(s)$ is strictly decreasing in $I$. Also, by \eqref{ress5}, we have $x'(s)>0$ for all $s\in I$. In particular, $\{\theta(s): s\in I\}$ is a smooth graph $y=y(x)$, where $x$ varies in an interval $J\subset [0,\8)$ containing the value $x=0$, and $y'(x)<0$ for all $x\in J$, $x>0$.

We compute next the explicit expression of $y(x)$. By \eqref{ress5} we have

 
 $$ y'(x)= \frac{G_c(x)}{y}.
$$
Solving this ODE by separating variables, with the initial condition that $y=1$ at $x=0$, we get
 \begin{equation}\label{ress7}
y(x)=\sqrt{1+\frac{2(c+\tau^2)}{\kappa-4\tau^2} \, {\rm log} \left(\frac{4-(\kappa-8\tau^2)x^2}{4+\kappa x^2}\right) }.
\end{equation}

\vspace{0.2cm}

\noindent {\bf Claim:} \emph{There exists $x_0>0$ such that} 
 \begin{equation}\label{eqde}
 \delta(x)= 1+\frac{2(c+\tau^2)}{\kappa-4\tau^2} \, {\rm log} \left(\frac{4-(\kappa-8\tau^2)x^2}{4+\kappa x^2}\right) \end{equation} \emph{is a strictly decreasing, well defined function in $[0,x_0]$, with $\delta(0)=1$ and $\delta(x_0)=0$.}

\vspace{0.2cm}

\noindent \emph{Proof of the Claim :}
It is clear from \eqref{eqde} that $\delta(x)$ is well defined in $[0,x^*)$, where $$x^*= 
\left\{ 
\def\arraystretch{1.3}\begin{array}{lll} \frac{2}{\sqrt{-\kappa}} & \text{ if } & \kappa <0,\\ \frac{2}{\sqrt{\kappa-8\tau^2}} & \text{ if } & \kappa >8\tau^2,\\
\8 & \text{ if } & 0\leq \kappa \leq 8 \tau^2, \end{array}
\right.$$ and it satisfies $\delta'(x)<0$ for all $x\in [0,x^*)$. In the first two cases for $x^*$, and also for $\kappa=0$ and $\kappa=8\tau^2$, it is easy to check that $\delta(x)\to -\8$ as $x\to x^*$, what proves the Claim in those cases.

Finally, if $0<\kappa<8\tau^2$ (with $\kappa\neq 4\tau^2$), we have 
 \begin{equation}\label{lir}
\lim_{x\to x^*} \delta(x) = 1+ \frac{2(c+\tau^2)}{\kappa-4\tau^2} {\rm log} \left(\frac{8\tau^2-\kappa}{\kappa}\right).
 \end{equation} 
It is easy to prove that the function $$\varphi(t):= 1 + \frac{2(\bar{c}+t)}{1-4t} {\rm log} \left(8t-1\right)$$ is negative for all $t>1/8$, what proves that the right-hand side of \eqref{lir} is negative, choosing $t=\tau^2/\kappa$ and $\bar{c}= c/\kappa$. Once here, a similar argument to the one used in the other two cases proves the Claim for the remaining situation $0<\kappa<8\tau^2$.

\vspace{0.2cm}

We now complete the proof of Theorem \ref{posex}. By Proposition \ref{posrot}, to prove that $S$ is a sphere we only need to show that the norm of the second fundamental form of $S$ is bounded for values of the angle function $\nu\in [0,1]$, that $S$ cannot converge $C^\8$-asymptotically to a rotational vertical cylinder $x_1^2+x_2^2=R^2$ in $\Rk$, 
and that $S$ is not an entire graph in $\Rk$. The fact that $S$ cannot converge $C^\8$-asymptotically to a cylinder is clear, because of the condition $K_e=c>0$ (the cylinders $x_1^2+x_2^2=R^2$ have constant extrinsic curvature $K_e=-\tau^2$). Also, it follows from \eqref{ress7}, the Claim above and $x(s)=\rho(s)$ that $S$ is not an entire graph.

Finally, in order to prove that the second fundamental form of $S$ is bounded, it suffices to prove that the right-hand side of \eqref{resH} is bounded, which by \eqref{ress4} fails to holds only if one of the next situations happens for $s$ approaching some value $s_*>0$ with $y(s)=\nu(s)>0$ for all $s\in [0,s_*)$.

\begin{enumerate}
\item
$\rho'(s)^2=\nu(s)^2$ converges to $1/(1+\tau^2 \rho(s)^2)$. 
\item
$4+\kappa \rho(s)^2 $ converges to $0$.
 \item
$4-(\kappa-8\tau^2)\rho(s)^2$ converges to $0$.
\end{enumerate}
The last two conditions are impossible, by \eqref{ress7}, the above Claim, and the fact that, in these conditions $\delta(\rho(s))\to -\8$ as $s\to s_*$. As regards the first condition, it is also impossible since the function $$\varphi(x) :=\delta(x) - \frac{1}{1+\tau^2 x^2},$$ where $\delta(x)$ is given by \eqref{eqde}, satisfies that $\varphi(0)=0$ and $\varphi'(x)<0$ if $x>0$. Thus, we deduce that $S$ is a rotational sphere of constant extrinsic curvature $c>0$. This proves Theorem \ref{posex}.
\end{proof}

Let us remark that, because of \eqref{ress7} and \eqref{eqan1}, it is possible to give an explicit (but complicated) expression for the sphere of constant extrinsic curvature $c>0$ in $\Ek$ in terms of an integral. We omit the specific formula.

\begin{figure}
\begin{center}
\includegraphics[height=7cm]{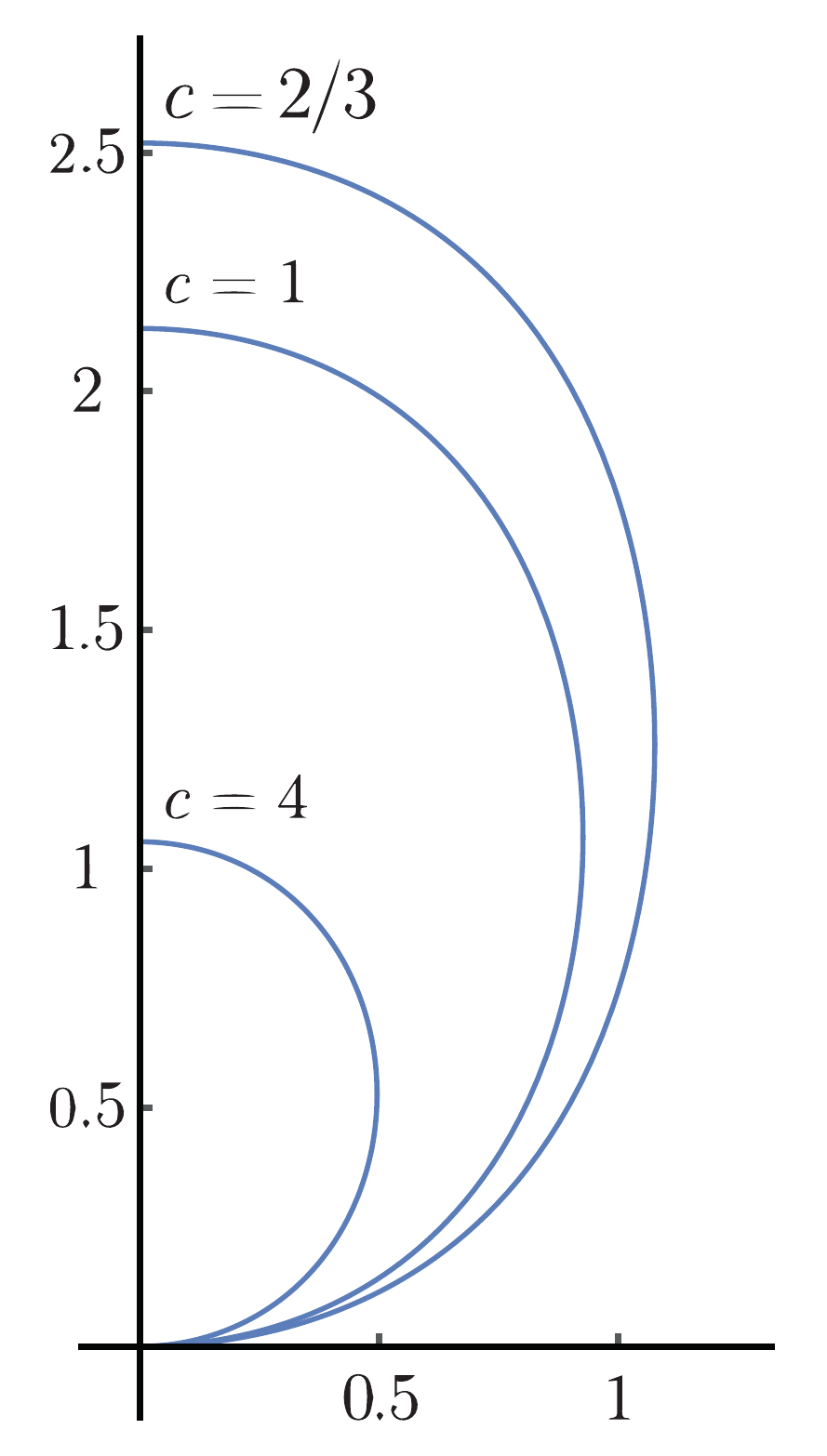}  \hspace{2cm}
\includegraphics[height=7cm]{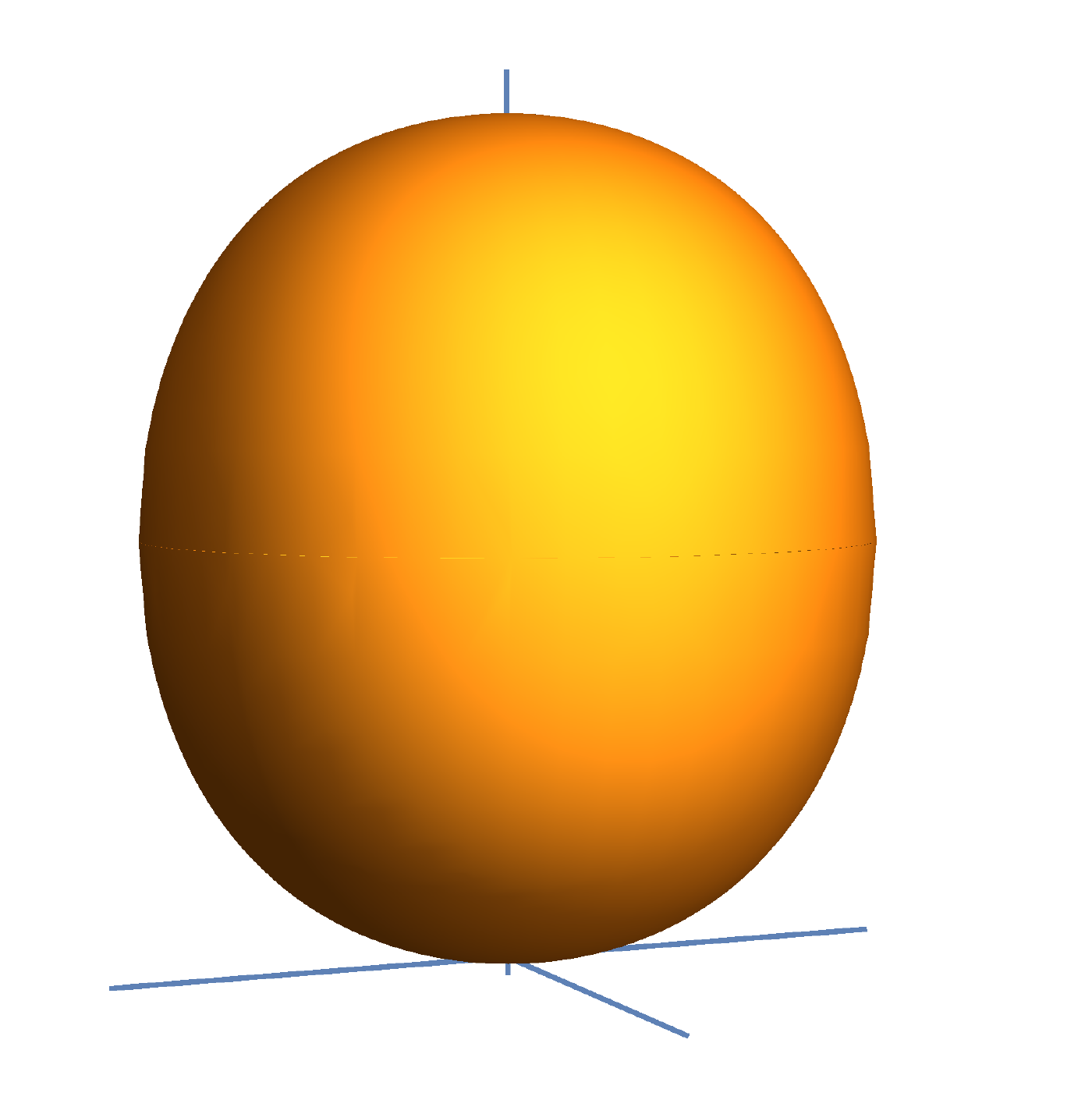}
\caption{Left: profile curves of rotational spheres with $K_e=c$ in ${\rm Nil}_3 = \mathbb{E}^3(0,1)$ for some values of $c>0$. Right: Rotational sphere with $K_e=1$ in ${\rm Nil}_3$.}
\end{center}
\end{figure}

\begin{remark}
The proof of Theorem \ref{posex} shows that if $S$ is a rotational sphere in $M=\Ek$ and we view it in a canonical coordinate model $\Rk$, then $S$ is an embedded rotational sphere in $\Rk$. However, when $M$ is compact, i.e. a Berger sphere, this does not imply that the rotational sphere $S$ is actually embedded in $M$.

More specifically, let $\pi:M=\Ek\flecha \S^2(\kappa)$ denote the canonical fibration of $M$ onto $\S^2(\kappa)$, and let $\Rk$ denote an associated coordinate model for $M$, with coordinates $(x_1,x_2,x_3)$. As explained in Appendix 2, two points $(x_1,x_2,x_3)$ and $(x_1,x_2,x_3+ 8\pi \tau/\kappa)$ in $\Rk$ correspond to the same point in $M$. In particular, if the rotational sphere $S$ in $\Rk$ starts in $\Rk$ at height zero, and its maximum height is greater than $8\pi \tau/\kappa$, then $S$ is not embedded when viewed in $M$. There are values of $\kappa,\tau$ and $c$ for which this situation happens; see Figure \ref{noen} and Figure \ref{figesf}.
\end{remark}

\begin{figure}[h]
\begin{center}
\includegraphics[height=6cm]{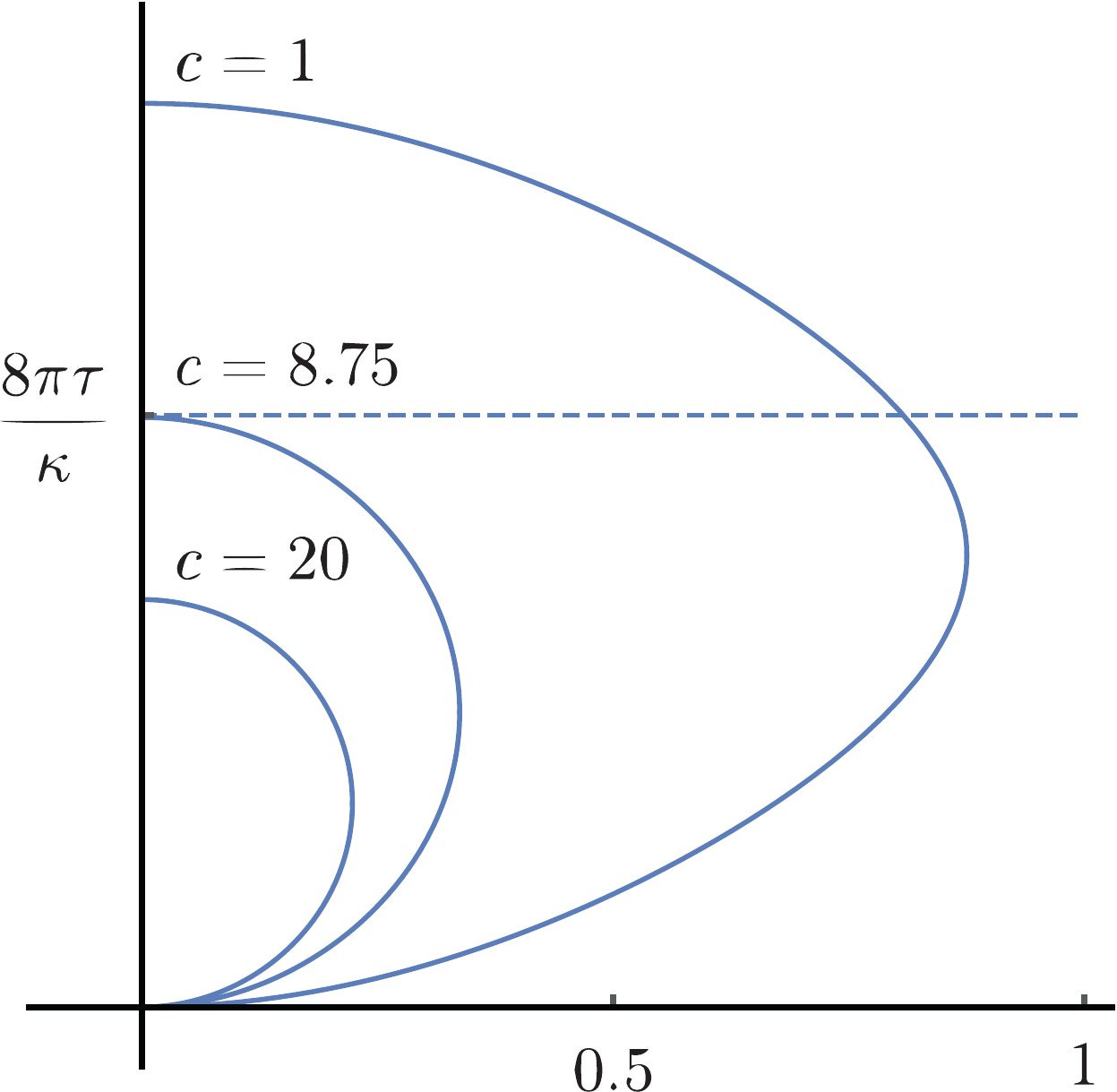}  \hspace{1cm}
\caption{Profile curves of rotational spheres with $K_e=c$ in the Berger sphere $M=\Ek$ for $\kappa=4$, $\tau=1/10$, in a canonical coordinate model $\Rk$. The ones whose profile curves reach the height $8\pi \tau/\kappa$ are not embedded in $M$.}\label{noen}
\end{center}
\end{figure}

\subsection{A Minkowski-type problem in $\Ek$}

In this section we prove, using the arguments and results from the previous section:

\begin{theorem}\label{th:min}
Let $M=\Ek$ and let $\Phi\in C^\8 ([0,1])$, $\Phi>0$ Then:
\begin{enumerate}
\item
There exists a rotational sphere $S$ in $M$ that satisfies 
 \begin{equation}\label{enumin}
 K_e =\Phi(\nu^2),
\end{equation}
where $K_e$ and $\nu$ are the extrinsic curvature and the angle function of $S$, respectively.
 \item
Any other immersed sphere in $M$ whose extrinsic curvature and angle function satisfy \eqref{enumin} is, up to ambient isometry, the rotational sphere $S$.
\end{enumerate}
\end{theorem}
\begin{proof}
Equation \eqref{enumin} is a general elliptic Weingarten equation in $\Ek$, so in particular the canonical rotational example $S$ that satisfies \eqref{enumin} exists. As in the proof of Theorem \ref{posex}, we want to prove that $S$ is a sphere, what will prove item (1). Item (2) is a direct consequence of item (1) and Theorem \ref{mainth}.

In order to prove that $S$ is a sphere, we start arguing as in the proof of Theorem \ref{posex}. We assume that the rotation axis of $S$ corresponds to  the $x_3$-axis in a standard coordinate model $\mathcal{R}^3(\kappa,\tau)$ of $\Ek$, and we let $(\rho(s),0,h(s))$ denote the profile curve of $S$, parametrized with respect to the parameter $s$ defined before  \eqref{metpv}, and so that $\rho(0)=h(0)=h'(0)=0$, $\rho'(0)=1$. Writing $x(s):=\rho(s)$, $y(s):=\rho'(s)=\nu(s)$, we deduce from \eqref{enumin}, just as we did in \eqref{ress5} for the case $\Phi={\rm constant}$, that $\theta(s):=(x(s),y(s))$ defines an orbit of the autonomous ODE system
 \begin{equation}\label{rest5}
\left. \def\arraystretch{1.5}\begin{array}{lll} x' & = & y \\ y' & = & G(x,y) \end{array} \right\}, \hspace{0.5cm} G(x,y):=\frac{16 x (\Phi(y^2)+\tau^2)}{(4+\kappa x^2)^2(-4+(\kappa-8\tau^2) x^2)}.\end{equation}

Note that $\theta(0)=(0,1)$. By the monotonocity properties of orbits of \eqref{rest5}, the piece of $\theta(s)$ that contains $(0,1)$ and lies in the region $y>0$ is a graph $y=y(x)$ on an interval $J\subset [0,\8)$ containing the value $x=0$, with $y'(x)<0$ if $x>0$. By \eqref{rest5} and the initial conditions, the geodesic curvature of $\theta(s)$ at $s=0$ has the value $-(\Phi(1)+\tau^2)$. 

Choose now $c\in (0,m)$, with $m:={\rm min}\{\Phi(v): v\in [0,1]\}$, and let $\theta_c(s)$ denote the orbit of \eqref{ress5} corresponding to the canonical rotational sphere $S_c$ in $\Ek$ with $K_e=c>0$. Again, $\theta_c(0)=(0,1)$. Moreover, $\theta_c(s)$ restricted to the region $y\geq 0$ is the graph $y=y_c(x)$ where $y_c(x)\in C^0([0,x_0])$ is given by the right-hand side of \eqref{ress7}. By the previous computation, the geodesic curvature of $\theta_c(s)$ at $s=0$ is $-(c+\tau^2)$, and so $y(x)<y_c(x)$ for all $x>0$ sufficiently small.

\begin{figure}[h]\label{fig73}
\begin{center}
\includegraphics[height=6cm]{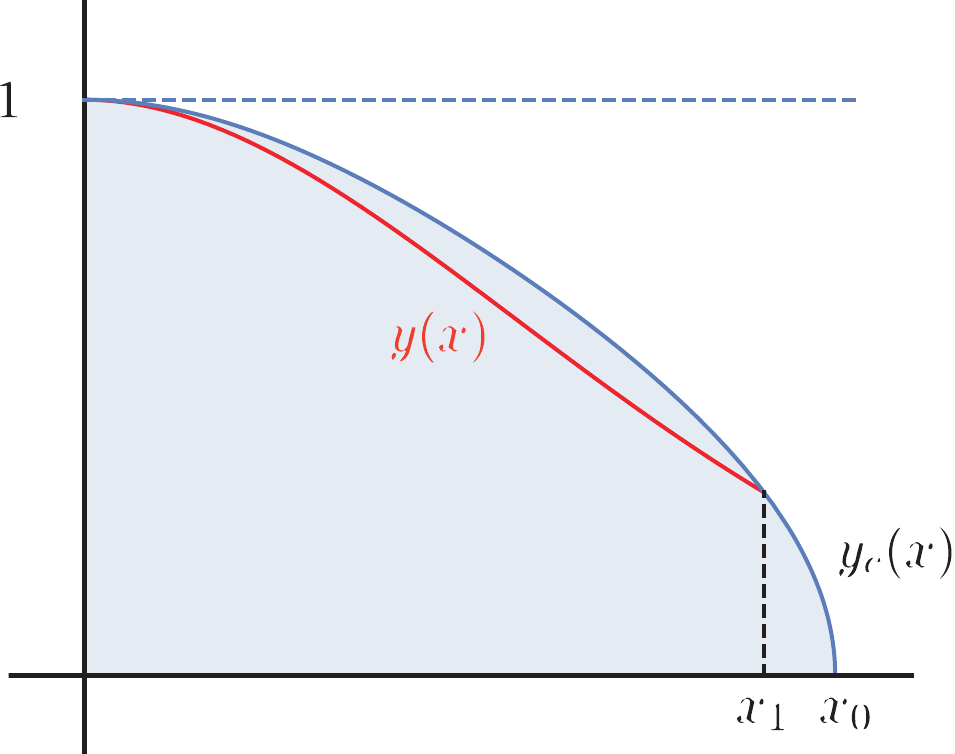}  \hspace{1cm}
\caption{One cannot have $y(x_1)=y_c(x_1)$ for $x_1\in (0,x_0)$.}
\end{center}
\end{figure}

Assume now that there exists $x_1\in (0,x_0)$ with $y(x_1)=y_c(x_1)$ and $y(x)<y_c(x)$ for all $x\in (0,x_1)$. Then we would have $y'(x_1)^2 \leq y_c'(x_1)^2$, what is impossible from the definition of $c$, by comparing \eqref{ress5} with \eqref{rest5}; see Figure. Thus, the orbit $\theta(s)$ cannot intersect $\theta_c(s)$ in the region $y>0$. In particular, by our previous study of $y_c(x)$, we have that the denominator of $G(x(s),y(s))$, for $G$ given as in \eqref{rest5}, is bounded from above by a negative constant. This implies that the graph $y=y(x)$ given by the orbit $\theta(s)$ is defined on an interval $[0,\bar{x}_0]$ with $\bar{x}_0<x_0$, with $y(\bar{x}_0)=0$, and with $y(x)<y_c(x)$ for all $x\in (0,\bar{x}_0]$. This also implies that $$y(x)^2 <y_c(x)^2 < \frac{1}{1+\tau^2 x^2}$$ for all $x\in (0,\bar{x}_0)$, where the last inequality comes from the last part of the proof of Theorem \ref{posex}. Taking all of this into account, we can argue as in the proof of Theorem \ref{posex} and conclude that $S$ is a rotational sphere. This completes the proof of Theorem \ref{th:min}.
\end{proof}

\section{Weingarten spheres in $\H^2\times \R$ and $\S^2\times \R$}\label{sec8}

\subsection{Classification of Weingarten spheres in $\H^2\times \R$}\label{esth2}

In the present Section \ref{esth2} we will let $\cW$ denote the class of immersed oriented surfaces in $\H^2\times \R$ that satisfy an arbitrary elliptic Weingarten equation \eqref{speci1}. As explained in Section \ref{sec:rew}, we can view \eqref{speci1} in the form \eqref{weq3}, i.e. $\kappa_1=f(\kappa_2)$, where $f\in C^\8 (a,b)$ satisfies conditions (i) to (iv) stated after equation \eqref{weq3}.

Let $S$ be the canonical rotational example in $\H^2\times \R$ of the class $\cW$; see Definition \ref{def:rote}. We prove next:

\begin{proposition}\label{boundh2}
The norm of the second fundamental form of $S$ is bounded.
\end{proposition}
\begin{proof}
Let us view $\H^2\times \R$ in the standard way as a subset of the Minkowski $4$-space $\mathbb{L}^4$ with signature $(- + + \, +)$. Let $\gamma(s)$ denote the profile curve of $S$ in this model, parametrized by arc-length, and assume by contradiction that $S$ does not have bounded second fundamental form. Then we can write $$\gamma(s)=(\sinh r(s),0,\cosh r(s),h(s)), \hspace{1cm} s\in [0,s_*),$$
with $r'^2+h'^2=1$, $r(0)=h(0)=0$, $h'(0)=0$ and $r(s)>0$ if $s>0$; moreover, there exists a sequence $(s_n)_n \in (0,s_*)$ converging to $s_*$ such that $|\sigma |(s_n)\to \8$, where $|\sigma | $ denotes the norm of the second fundamental form of $S$.

The angle function of $S$ is given by $\nu(s)=r'(s)$ (thus, $\nu(0)=1$), and the principal curvatures of $S$ are 
 \begin{equation}\label{pcs}
 \kappa_1 = r'h'' - r'' h', \hspace{1cm} \kappa_2 = h' \coth r.
 \end{equation}

Assume that there exists $\bar{s}\in (0,s_*)$ such that $\nu(\bar{s})=0$. Then, by our analysis in Section \ref{sub:ext} the canonical rotational example $S$ is a rotational sphere, what contradicts that $|\sigma | $ is unbounded on $S$. Therefore, $\nu(s)>0$ in $[0,s_*)$, i.e. $r(s)$ is strictly increasing in $[0,s^*)$.

Let $\alfa\geq 0$ denote the umbilicity constant of the class $\cW$, given by the condition $f(\alfa)=\alfa$. If $\alfa=0$, $S$ is the totally geodesic slice $\H^2\times \{0\}$, and the result is immediate. We assume from now on that $\alfa>0$. By \eqref{pcs}, and since $S$ has an umbilic point at $s=0$, we see that $\alfa=h''(0)$.

By Lemma \ref{monan}, the angle function $\nu(s)=r'(s)$ is strictly decreasing, and $h'(s)>0$ in $(0,s_*)$. In particular, by $r'^2+h'^2=1$ we conclude that $h''\geq 0$. This implies by \eqref{pcs} that $\kappa_1(s)\geq 0$ and $0<\kappa_2(s)\leq C$ for some $C<\8$, for every $s\in [0,s_*)$.

We can deduce directly from here that the second fundamental form of $S$ would be bounded whenever the domain of definition $(a,b)$ of the function $f$ satisfies that $a<0$ or $a=-\8$. Indeed, in that case, by the properties (i)-(iv) of $f$, the arc of the curve $k_1=f(k_2)$ in the $(k_1,k_2)$-plane that passes through $(\alfa,\alfa)$ and lies in the region $\{k_1\geq 0, k_2 \geq 0\}$, is bounded. As $(\kappa_1(s),\kappa_2(s))$ must lie in this arc for all $s\in [0,s_*)$, the second fundamental form of $S$ is uniformly bounded in that case.

Consider, thus, the remaining case, i.e. $a\geq 0$. Since $0<\kappa_2(s)\leq C$, we conclude that $\kappa_1(s_n)\to \8$ as $s_n\to s_*$, where $(s_n)_n$ is the sequence defined at the beginning of the proof.

By \eqref{pcs}, we can regard (using that $h'>0$) the Weingarten equation $\kappa_1=f(\kappa_2)$ as the following autonomous ODE system, where $x(s):=r(s)>0$ and $y(s):= r'(s)\in (-1,1)$:
\begin{equation}\label{odesi}
\left. \def\arraystretch{1.5}\begin{array}{lll} x' & = & y \\ y' & = & -\sqrt{1-y^2} \, f\left(\sqrt{1-y^2} \coth x \right) \end{array} \right\}.
\end{equation}

Let $\theta(s):=(x(s),y(s))$, $s\in (0,s_*)$, denote the orbit of \eqref{odesi} that corresponds to the canonical rotational example $S$. Note that $y(s)>0$ and $y'(s)<0$ for every $s\in (0,s_*)$, and that $\theta(s)\to (0,1)$ as $s\to 0$. Also, since $\kappa_1(s_n)\to \8$, we see from \eqref{odesi} that the points $\theta(s_n)$ converge to the curve $\Gamma \subset [0,\8)\times [0,1]$ given by $a=\sqrt{1-y^2} \coth x$. In particular, we must have $a>0$, and so $\Gamma$ is given by $$x= \Gamma(y):= \tanh^{-1} \left(\frac{1}{a} \sqrt{1-y^2}\right).$$ Consider next $\alfa_0 \in (a,\alfa)$, and the curve $\Gamma_0$ in $[0,\8)\times [0,1]$ given by 
 \begin{equation}\label{eqg0}
 x= \Gamma_0(y):= \tanh^{-1} \left(\frac{1}{\alfa_0} \sqrt{1-y^2}\right).
  \end{equation}
  Note that $(0,1)\in \Gamma \cap \Gamma_0$, and that $\Gamma_0$ is \emph{in the left side} of $\Gamma$, in the following sense: if $(x,y)\in \Gamma$ for $y\in [0,1)$, then there exists $x_0\in (0,x)$ such that $(x_0,y)\in \Gamma_0$.
  
A direct computation from \eqref{eqg0} shows that the absolute value of the geodesic curvature of $\Gamma_0$ at the point $(0,1)$ is $\alfa_0^2$. Similarly, a computation using \eqref{odesi} shows that the absolute value of the geodesic curvature of the orbit $\theta(s)$ at $s=0$, i.e. also at the point $(0,1)$, is given by $\alfa^2$. Since $\alfa_0^2<\alfa^2$, the orbit $\theta(s)$ is, near $(1,0)$, in the left side of $\Gamma_0$ (in the above sense).
We should also observe that $\theta(s)$ is a graph $x=g(y)$ for some smooth function $g$, since $y'(s)<0$ for all $s\in (0,s_*)$.

\begin{figure}[h]
\begin{center}
\includegraphics[height=5.5cm]{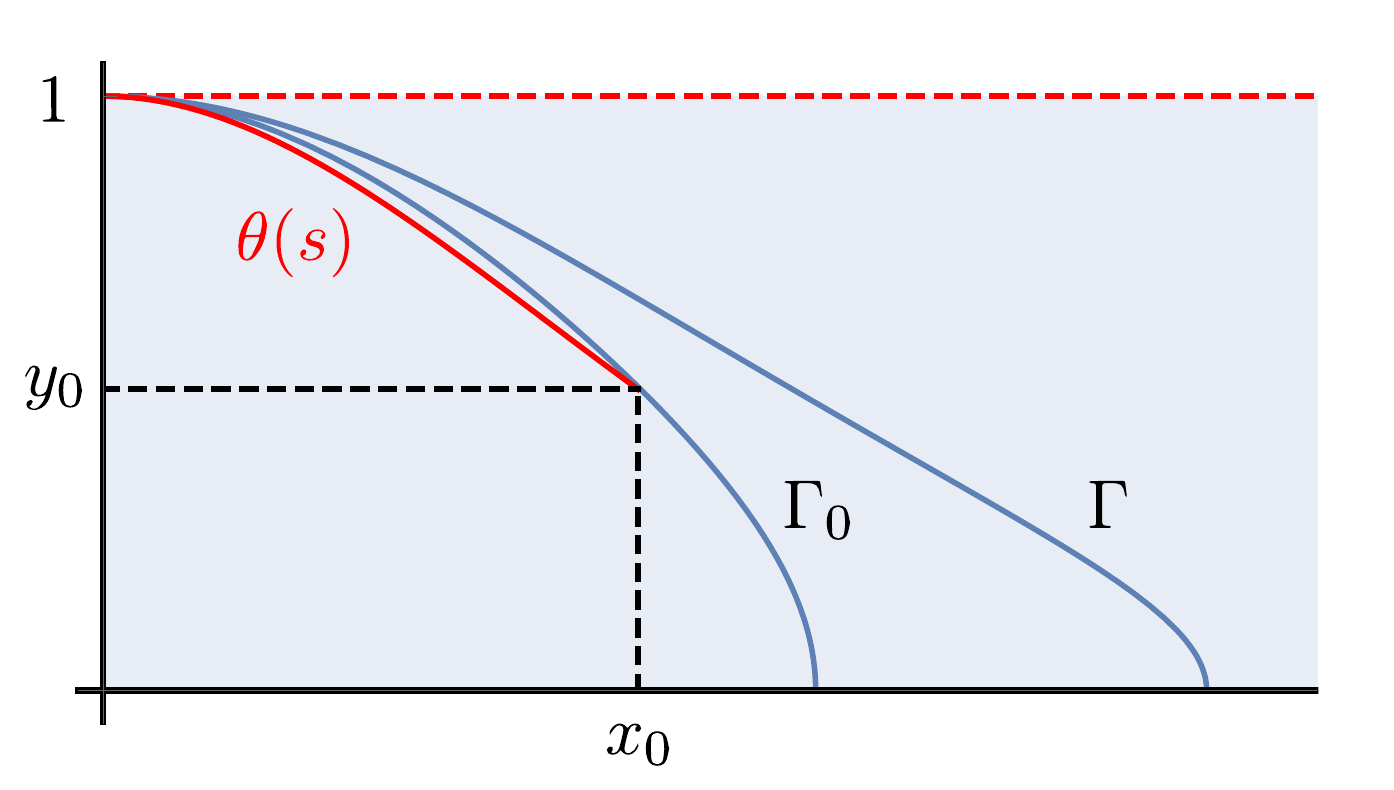}  \hspace{1cm}
\caption{The orbit $\theta(s)$ cannot intersect $\Gamma_0$, and hence it remains away from $\Gamma$ for positive values of $s$.}
\end{center}
\end{figure}

We next claim that $\theta(s)$ never intersects the curve $\Gamma_0$. Indeed, assume that $\theta(s_0)=(x_0,y_0)\in \Gamma_0$ for some $s_0\in (0,s_*)$, and that $\theta(s)\not\in \Gamma_0$ for all $s\in (0,s_0)$. It follows then from \eqref{odesi} that the slope of the tangent line to $\theta(s)$ at $(x_0,y_0)$ is $$m=-\frac{\sqrt{1-y_0^2}}{y_0} f(\sqrt{1-y_0^2} \coth x_0) = -\frac{\sqrt{1-y_0^2}}{y_0} f(\alfa_0)<0,$$ where the second equality comes from $(x_0,y_0)\in \Gamma_0$. On the other hand, a direct calculus shows that the slope of the tangent line to $\Gamma_0$ at $(x_0,y_0)$ is $$m_0=  -\frac{\sqrt{1-y_0^2}}{y_0} \frac{\alfa_0}{\cosh^2 x_0}<0.$$ Since $\alfa_0<\alfa$, $f(\alfa)=\alfa$ and $f'<0$, it follows that $f(\alfa_0)>\alfa$, and so $m^2 >m_0^2$. This is a contradiction, since the graph $x=g(y)$ that defines $\theta(s)$ lies, for $s\in (0,s_0)$, in the left side of $\Gamma_0$, i.e. it satisfies $g(y)<\Gamma_0(y)$ for all $y\in (y_0,1)$.

Consequently, $\theta(s)$ remains globally in the left side of $\Gamma_0$. This is a contradiction with the fact that the sequence $\theta(s_n)$ approaches the curve $\Gamma$, which lies in the right side of $\Gamma_0$. This contradiction proves that $S$ has bounded second fundamental form, and completes the proof of Proposition \ref{boundh2}.
\end{proof}

As a direct consequence of Proposition \ref{boundh2} and Theorem \ref{mainth}, we have the main conclusion of this section:

\begin{theorem}\label{clash2}
Any elliptic Weingarten sphere immersed in $\H^2\times \R$ is an embedded, rotational sphere.
\end{theorem}

\subsection{On general elliptic Weingarten spheres in $\H^2\times \R$}\label{sec:gews}

Our approach in the previous section does not work (without additional restrictions) for the case of general elliptic Weingarten surfaces in $\H^2\times\R$, given by a relation of the form \eqref{preq}. More specifically, we show below examples of equations of the form \eqref{preq} for which the canonical rotational example of the class does not have bounded second fundamental form (thus, Proposition \ref{boundh2} does not hold in this more general context).

Let us consider a general elliptic Weingarten relation 
 \begin{equation}\label{preqc}
W(\kappa_1,\kappa_2,\nu^2)=0
 \end{equation} 
in $\H^2\times \R$, where the function $W=W(k_1,k_2,v)\in C^{\8} (\R^2\times [0,1])$ is defined by
 \begin{equation}\label{contraw}
 W(k_1,k_2,v)= (k_1-\phi_0(v))(k_2-\phi_0(v))-1,
 \end{equation}
 for $\phi_0\in C^{\8} ([0,1])$, $\phi_0>0$, that satisfies
  \begin{equation}\label{contraf}
  \phi_0(v)>\sqrt{1-v} , \hspace{1cm} \forall v\in [0,1].
  \end{equation} 
Note that for each fixed value of $v$ there are two connected components of $W^{-1}(0)$; we will work in the one given by $k_i>\phi_0(v)$, $i=1,2$.
	
Let $S_0$ denote the canonical rotational example in $\H^2\times \R$ associated to \eqref{preqc}. Following the notations in the proof of Proposition \ref{boundh2}, we let $\gamma(s)$ be its profile curve. Then, just as we did to obtain equation \eqref{odesi}, the relation \eqref{preqc}-\eqref{contraw} can be written as an autonomous system of ODEs 

\begin{equation}\label{odesi2}
\left. \def\arraystretch{1.5}\begin{array}{lll} x' & = & y \\ y' & = & -\sqrt{1-y^2} \, f\left(\sqrt{1-y^2} \coth x,y^2 \right) \end{array} \right\},
\end{equation}
where $f(u_1,u_2)$ is defined whenever $u_1>\phi_0(u_2)$, and given by
 \begin{equation}\label{odesi3}
 f(u_1,u_2):= \phi_0(u_2) + \frac{1}{u_1-\phi_0(u_2)}.
\end{equation}
Observe that in \eqref{odesi2} we have the restrictions $x>0$, $y\in (-1,1)$ and $\sqrt{1-y^2}\coth x > \phi_0(y^2)$.

Let $\theta_0(s):=(x_0(s),y_0(s))$ denote the orbit of \eqref{odesi2} that corresponds to the canonical rotational example $S_0$; that is, $x_0(s)=r(s)$ and $y_0(s)=r'(s)$ for the function $r(s)$ associated to the profile curve $\gamma(s)$ of $S_0$. In particular, $x_0(s)$ and $y_0(s)$ satisfy the restrictions above. Also note that $\theta_0(s)\to (0,1)$ as $s\to 0^+$.

Choose now $\ep>0$ small enough so that $\theta_0(s)$ is defined for all $s\in (0,\ep]$, and take $(x^{\ep},y^{\ep}):=(x_0(\ep),y_0(\ep))$, with $y^{\ep}>0$. We define next a new function $\phi\in C^{\8} ([0,1])$ so that:

\begin{enumerate}
\item
$\phi=\phi_0$ in $[(y^{\ep})^2,1]$, and $\phi>\phi_0$ in $[0,(y^{\ep})^2)$.
 \item
$\phi ((y^{\ep})^2/4)> {\rm tanh} (x^{\ep}) / \sqrt{1-(y^{\ep})^2/4}$.
\end{enumerate}

\begin{proposition}\label{popu0}
Let $S$ be the canonical rotational example of the class of elliptic Weingarten surfaces $\cW_{\phi}$ in $\H^2\times \R$ given by \eqref{preqc}, where $W=W(k_1,k_2,v)\in C^\8(\R^2\times [0,1])$ is given by  \begin{equation}\label{contraw2}
 W(k_1,k_2,v)= (k_1-\phi(v))(k_2-\phi(v))-1,
 \end{equation}
with respect to the positive function $\phi\in C^{\8} ([0,1])$ defined above. Then, $S$ is a non-complete rotational graph with unbounded second fundamental form.
\end{proposition}
\begin{proof}
First, note that in the definition of $\cW_{\phi}$ we are substituting in our previous discussion \eqref{contraw} by \eqref{contraw2}, i.e. we are substituting $\phi_0$ by $\phi$. So, accordingly, the autonomous system \eqref{odesi2}-\eqref{odesi3} also transforms by replacing $\phi_0$ with $\phi$. Let $S$ denote the canonical rotational example of $\cW_{\phi}$, and let $\theta(s)=(x(s),y(s))$ denote its associated orbit. Note that $\theta_0(s)=\theta(s)$ for all $s\in [0,\ep]$, by uniqueness of canonical solutions and by condition (1) above. Also, by the monotonicity properties of solutions to \eqref{odesi2}-\eqref{odesi3}, we have $x'(s)>0$ and $y'(s)>0$ as long as the orbit $\theta(s)$ satisfies $0<y(s)<1$. If we use now the inequality for $\phi((y^{\ep})^2/4)$ in condition (2) above, this implies that $\theta(s)$ approaches the curve $\Gamma\subset \R^2$ given by $\phi(y^2)= \sqrt{1-y^2} \, {\rm coth}\, x$ before it reaches the value $y(s)=y^{\ep}/2$; i.e. there exists $s_*>\ep$ such that $\theta(s)\to (x_*,y_*)$ as $s\to s_*$, with $x_*>x^{\ep}$, $y_*\in (y^{\ep}/2,y^{\ep})$ and $\phi(y_*^2)=\sqrt{1-y_*^2} \, {\rm coth} \, x_*$. This indicates that $S$ is a rotational, non-complete graph in $\H^2\times \R$ with unbounded second fundamental form, since its principal curvature $\kappa_1$ blows up as $s\to s_*$.
\end{proof}


However, despite the possibility of constructing such examples, the arguments in Section \ref{esth2} still work for \emph{some} classes of general Weingarten surfaces in $\H^2\times \R$. Recall that, as explained in \eqref{preq}, the curvature relation that defines a class of general elliptic Weingarten surfaces can be written as
\begin{equation}\label{preq33}
\kappa_1=f(\kappa_2,\nu^2)
\end{equation}
where for each $v\in [0,1]$ fixed, the function $f=f(\cdot, v)$ is defined on a real interval $(a,b)=(a(v),b(v))$, and satisfies conditions (i) to (iv) after equation \eqref{weq3} of Section \ref{sec:rew}.
 
 In these conditions, we have:


\begin{proposition}\label{popu}
Let $\cW$ be a class of general elliptic Weingarten surfaces in $\H^2\times \R$, given by a relation \eqref{preq33}. Assume additionally that for each $v\in [0,1]$ we have $a(v)<0$ or $a(v)=-\8$. 

Then the canonical rotational example $S$ of $\cW$ has bounded second fundamental form, and any immersed sphere of the class $\cW$ is a sphere of revolution.
\end{proposition}
\begin{proof}
For the first statement, we simply need to observe that all the arguments in the first part of the proof of Proposition \ref{boundh2} (i.e. those corresponding to $a<0$ or $a=-\8$) also hold in this more general context under the assumptions that $a(v)<0$ or $a(v)=-\8$. The second statement is then a direct consequence of Theorem \ref{mainth}.
\end{proof}

As an immediate corollary, we have the following result, that extends the Abresch-Rosenberg theorem for CMC spheres in $\H^2\times \R$ to the case of (non-constant) prescribed mean curvature (compare it also with Corollary \ref{prh}, valid only for $\Phi>0$).

\begin{corollary}
Let $\Sigma$ be an immersed sphere in $\H^2\times \R$ whose mean curvature $H$ and angle function $\nu$ satisfy $$H=\Phi(\nu^2)$$ for some $\Phi\in C^{\8} ([0,1])$. Then, $\Sigma$ is a rotational sphere.
\end{corollary}

\subsection{On Weingarten spheres in $\S^2\times \R$}\label{ests2}
The proof of Proposition \ref{boundh2} does not work in general if the ambient space $\H^2\times \R$ is substituted by $\S^2\times \R$. An important reason for that is the behavior of the principal curvature $\kappa_2(s)$ in this setting. Specifically, if we view $\S^2\times \R\subset \R^4$ and let $$\gamma(s)=(\sin r(s),0,\cos r(s),h(s))$$
be the profile curve of a rotational surface, with $r'^2+h'^2=1$, $r(0)=h(0)=0$, $h'(0)=0$ and $r(s)>0$ if $s>0$, then the angle function of the surface is $\nu(s)=r'(s)$ and the principal curvatures of $S$ are 
 \begin{equation}\label{pcs2}
 \kappa_1 = r'h'' - r'' h', \hspace{1cm} \kappa_2 = h' \cot r.
 \end{equation}
Note that, in contrast with \eqref{pcs}, this time $\kappa_2$ can be negative even if $h'>0$ and $r>0$.

Moreover, we can show that Proposition \ref{boundh2} does not hold in $\S^2\times \R$. Specifically, the next example proves that for some classes of elliptic Weingarten surfaces in $\S^2\times \R$, the canonical rotational example reaches its \emph{antipodal axis} forming a singularity around it, and so that the principal curvature $\kappa_2$ blows up at that point.

\begin{example}
\emph{Let us view $\S^2\times \R\subset \R^4$. Let $S$ denote a rotational surface in $\S^2\times \R$ with rotation axis $L\equiv (1,0,0)\times \R$, parametrized as} $$\psi(\rho,\theta)=(\sin \rho \cos \theta, \sin \rho \sin \theta, \cos \rho, h(\rho)),$$ \emph{where $\rho\in (0,\pi)$, $\theta\in [0,2\pi)$ and so that $h(\rho)\in C^{\8}([0,\pi])$ satisfies:}

\begin{enumerate}
\item
\emph{$h(0)=h'(0)=0$, $h''(0)>0$.}
 \item
\emph{$h''(\rho)>0$ for every $\rho\in (0,\pi)$.}
 \item
\emph{$\kappa_1'(\rho) \kappa_2'(\rho)<0$ for every $\rho\in (0,\pi)$, where }
 \begin{equation}\label{kij}
 \kappa_1(\rho)= \frac{h''(\rho)}{(1+h'(\rho)^2)^{3/2}}, \hspace{1cm} \kappa_2(\rho)= \frac{h'(\rho) \cot (\rho)}{\sqrt{1+h'(\rho)^2}}.
 \end{equation}
  \item
There exists $\lim_{\rho\to 0} \kappa_1'(\rho)/\kappa_2'(\rho) = -1$.
\end{enumerate}
\emph{One can construct functions $h(\rho)$ in these conditions, for instance, by considering an even polynomial of the form $$h(\rho)= \frac{R}{2} \rho^2 + \frac{(R+ \delta)^3}{8} \rho^4,$$ with adequate constants $R,\delta>0$; for example, the following choices make all of the previous conditions hold:} $$R=1/20, \hspace{1cm} \delta =-R + \left(\frac{R}{6}+R^3\right)^{1/3}.$$

\emph{The first condition above indicates that the surface $S$ intersects its rotation axis $L$ orthogonally at the point $(1,0,0,0)$ (this corresponds to the value $\rho=0$). The second condition ensures that the angle function of $S$, given by $$\nu(\rho)=\frac{1}{\sqrt{1+h'(\rho)^2}},$$ is strictly decreasing. In this way, $S$ is a graph over $\Omega = \S^2\setminus\{(-1,0,0)\}$, which presents a conical singularity when $\rho=\pi$; that is, the rotational surface $S$ intersects its antipodal fiber $L^*$, but does so in a singular way.}

\emph{Finally, the third and fourth conditions prove that $S$ is an elliptic Weingarten surface in $\S^2\times \R$, as we explain next. The functions $\kappa_i(\rho)$ appearing in \eqref{kij} are the principal curvatures of $S$. Thus, the condition $\kappa_1'(\rho)\kappa_2'(\rho)<0$ indicates that $\kappa_1(\rho)$ is strictly increasing and $\kappa_2(\rho)$ is strictly decreasing, or vice versa. In any case, we can deduce that $S$ satisfies an elliptic Weingarten equation of the form $\kappa_1=f(\kappa_2)$, where $f(t)$ is smooth, with $f'<0$. The fourth condition implies that the slope of the curve $k_1=f(k_2)$ in the $(k_1,k_2)$-plane is $-1$ at the point where it meets the diagonal $k_1=k_2$, and thus the Weingarten relation can be rewritten as $W(\kappa_1,\kappa_2)=0$ with $W(k_1,k_2)$ symmetric in $(k_1,k_2)$.}

\end{example}

It is also possible to modify the previous example, and create situations where the canonical example of a class of elliptic Weingarten surfaces in $\S^2\times \R$ has unbounded second fundamental form but remains at a positive distance from its antipodal axis.

Some of the ideas developed in Sections \ref{esth2} and \ref{sec:gews} for $\H^2\times \R$ also work for some special classes of Weingarten surfaces in $\S^2\times\R$, but we will not follow that line of inquiry here.

\section*{Appendix 1: Statement of the results in $\R^3$ and $\S^3$}


As explained in Remark \ref{ekte}, the Euclidean space $\R^3$ can be seen as a degenerate case of the spaces $\Ek$, obtained by choosing $\kappa=\tau=0$. A similar situation happens in the round sphere $\S^3(c)$, which can be recovered as the degenerate case of the spaces $\Ek$ for $\kappa=4\tau^2=4c>0$. Note that $\S^3(c)$ can be viewed as $(\S^3,g)$ with the metric $g$ in \eqref{metrige} for $\kappa=4\tau^2$.

It then turns out that the results that we have obtained here for $M=\Ek$ also work for $M=\R^3$ and $\S^3(c)$, in general with simpler arguments and computations. Next, we state some of these results explicitly.

Let $M$ be $\R^3$ or $\S^3(c)$, and let $\xi$ denote a unit Killing field on $M$. 
Then, for any immersed oriented surface $\Sigma$ in $M$ we can define its \emph{angle function} in the direction $\xi$ as $\nu:=\esiz \eta,\xi\esde \in C^{\8}(\Sigma)$, where $\eta$ is the unit normal of $\Sigma$ in $M$.

In these conditions, the definition of a \emph{general (elliptic) Weingarten surface} given in Definition \ref{prescri} (and its equivalent definitions \eqref{preq2} and \eqref{preq3}) also makes sense when $M=\R^3,\S^3(c)$. We remark that, in $\R^3$, this general Weingarten equation $W(\kappa_1,\kappa_2,\nu^2)=0$ falls into the more general class of surfaces governed by prescribed curvature equations $W(\kappa_1,\kappa_2,\eta)=0$, where $\eta$ is the unit normal of the surface. For results dealing with the uniqueness of immersed spheres in $\R^3$ satisfying this type of prescribed curvature equations, see e.g. \cite{A,GM3,HW} and references therein.

As proved for $\Ek$-spaces, given a general Weingarten class $\cW$ in $M$, there is an inextendible rotational surface $S$ of the class $\cW$ with rotation axis $L$ tangent to $\xi$, and such that $S$ touches $L$ orthogonally at some point. Then, the following result holds, which corresponds to Theorem \ref{mainth} for our situation.

\begin{theorem}\label{mtc}
Let $M=\R^3$ or $\S^3(c)$, and let $\cW$ denote a general Weingarten class of surfaces in $M$. Assume that the canonical rotational example $S$ of $\cW$ has bounded second fundamental form.

Then, any immersed sphere of the class $\cW$ is a rotational sphere. If $M=\R^3$, this sphere is actually strictly convex, i.e. an ovaloid.
\end{theorem}
The statement of Theorem \ref{bomi} also holds if $M=\R^3$ or $\S^3(c)$. Thus, its Corollary \ref{prh} also holds in these cases. We must note, however, that Corollary \ref{prh} was already known if $M=\R^3$, as a consequence of previous results by the authors; see \cite{BGM,GM,GM3}.

Similarly, the Minkowski-type result in Theorem \ref{th:min} also holds for $M=\R^3$ or $\S^3(c)$, but for $\R^3$ it is an immediate consequence of the classical solution to the Minkowski problem. For $M=\S^3(c)$, Theorem \ref{th:min} seems new, and is motivated by the following natural extension of the classical Minkowski problem (which Theorem \ref{th:min} solves in the case that $\mathcal{K}$ is rotationally symmetric in some direction).

\vspace{0.2cm}

{\bf The left-invariant Minkowski problem in $\S^3$:} \emph{given $\mathcal{K}\in C^\8(\S^2)$, $\mathcal{K}>0$, prove existence and uniqueness of a strictly convex sphere $\Sigma\subset \S^3$ whose extrinsic curvature $K_e$ satisfies $K_e=\mathcal{K}\circ g$, where here $g:\Sigma\flecha \S^2\subset T_e\S^3$ is the left invariant Gauss map of $\Sigma$, obtained by viewing $\S^3$ as the Lie group ${\rm SU}(2)$ with a bi-invariant metric, and left-translating the unit normal $\eta$ of $\Sigma$ to the identity element $e$ of ${\rm SU}(2)$.}

\vspace{0.2cm}

Proposition \ref{popu} also holds when the ambient space is $\R^3$ instead of $\H^2\times \R$. In $\R^3$ one can also create, similarly to Proposition \ref{popu0}, examples of general Weingarten classes $\cW$ for which the canonical rotational example $S$ is a non-complete graph of unbounded second fundamental form.

\section*{Appendix 2: On the geometry of Berger spheres}

In this Appendix we describe in more detail the geometry of Berger spheres, i.e. of the $\Ek$ spaces with $\kappa>0$ and $\tau\neq 0$, all of which are diffeomorphic to $\S^3$. 

Consider first of all $\S^3=\{(z,w)\in \C^2 : |z|^2+|w|^2=1\}$. A basis of the tangent bundle of $\S^3$ is given by the vector fields 
 \begin{equation}\label{ecuchi}
e_1 = (-\overline{w},\overline{z}), \hspace{0.5cm} e_2 = (-i\overline{w},i \overline{z}), \hspace{0.5cm} \hat{\xi} = (iz, iw).
 \end{equation}
 The Berger sphere $M=\Ek$ can be seen then as $(\S^3,g)$, where $g$ is the Riemannian metric on $\S^3$ given for any $X,Y\in T\S^3$ by
 
 \begin{equation}\label{metrige}
  g(X,Y)= \frac{4}{\kappa}\left(\esiz X,Y\esde + \left(\frac{4\tau^2}{\kappa}-1\right) \esiz X, \hat{\xi}\esde \esiz Y, \hat{\xi}\esde\right),
   \end{equation}
 where $\esiz,\esde$ denotes the Euclidean metric in $\R^4\equiv \C^2$. 
 
Let $\S^2(\kappa)$ denote the two-dimensional sphere of constant curvature $\kappa>0$, which we will view as $\S^2(\kappa)=\{x\in \R^3 : \esiz x,x\esde = 1/\kappa\}$. Then, the Hopf fibration $\pi :(\S^3,g)\flecha \S^2(\kappa)$, given by $$\pi(z,w)=\frac{2}{\sqrt{\kappa}} \left(z\overline{w}, \frac{1}{2}(|z|^2-|w|^2)\right)$$ is a Riemannian submersion, with kernel generated by the vector field $\hat{\xi}$, which is a Killing field of $(\S^3,g)$ of constant length. Thus, $\pi$ corresponds to the canonical submersion of $\Ek$, and $\xi=\hat{\xi}/|\hat{\xi}|$ is the vertical unit Killing field of $\Ek$.

In order to construct a canonical coordinate model $\Rk$ associated to the space, we first parametrize $\S^2(\kappa)\setminus \{(0,0,-1/\sqrt{\kappa})\}$ by inverse stereographic projection:

$$\varphi (x_1,x_2)=\left(\landa x_1, \landa x_2, \frac{1}{\sqrt{\kappa}} (1-2\landa)\right) : \R^2\flecha \S^2(\kappa)\setminus \{(0,0,-1/\sqrt{\kappa})\},$$ where 
 \begin{equation}\label{elan}
\landa = \frac{1}{1+ \frac{\kappa}{4} (x_1^2+x_2^2)}.
 \end{equation}
This provides the $(x_1,x_2)$-coordinates in $\Rk$. We let the $x_3$-coordinate of $\Rk$ be the unit speed parametrization of the fiber $\pi^{-1}(\varphi(x_1,x_2))$. In this way, we obtain that the coordinates $(x_1,x_2,x_3)\in \R^3$ cover $\S^3$ minus the fiber $\pi^{-1} ((0,0,-1/\sqrt{\kappa}))$, which equals $L^*:=\{(e^{i\theta},0): \theta \in \R\}$. More specifically, this $\Rk$ model corresponds then to the universal cover of $\S^3\setminus L^*$, and two points $(x_1,x_2,x_3)$ and $(x_1,x_2,x_3 + 8\tau \pi/\kappa)$ correspond to the same point of $\S^3$.

Explicitly, we have the isometric immersion $\Psi: (\Rk,ds^2)\flecha (\S^3\setminus L^*,g) $,
\begin{equation}\label{io}
\Psi(x_1,x_2,x_3)=\sqrt{\landa} \left(\frac{\sqrt{\kappa}}{2}(x_1 +i x_2) e^{i\sigma x_3}, e^{i\sigma x_3}\right), \hspace{0.5cm} \sigma:=\frac{\kappa}{4\tau}.
\end{equation}
where $\landa$ is given by \eqref{elan} and $ds^2$ is the metric \eqref{eq:metric}.

Note that the $x_3$-axis in $\Rk$ corresponds to the universal cover of the fiber $L:=\{(0,e^{i\theta}) : \theta \in \R\}$ of $\S^3$, and that the fiber $L^*$ does not appear in this $\Rk$ model. In order to have a model for $\S^3$ where both $L,L^*$ appear, we consider the stereographic projection of $\S^3 \setminus \{p_N\}\subset \R^4$ into $\R^3$, $p_N:=(0,0,0,1)$, given by 
 \begin{equation}\label{newc}
(x_1,x_2,x_3,x_4)\mapsto (y_1,y_2,y_3):=\frac{1}{1-x_4} (x_1,x_2,x_3).
 \end{equation}
In these $(y_1,y_2,y_3)$-coordinates, $L\setminus \{p_N\}$ 
corresponds to the $y_3$-axis, while the fiber $L^*$ corresponds diffeomorphically to the circle $\{y_1^2+y_2^2=1, y_3=0\}$.

With these different models in hand, we turn now our attention to rotational surfaces $S$ in $(\S^3,g)$. Assume that the rotation axis of  $S$ is the fiber $L$, and that $S$ remains away from its \emph{antipodal fiber} $L^*$. Then, in the $\Rk$ model, $S$ defines a rotational surface $\hat{S}$ around the $x_3$-axis, that can be parametrized as $$\psi(u,v)=(\rho(u) \cos v, \rho (u) \sin v, h(u)).$$ In the $(y_1,y_2,y_3)$-coordinates for $\S^3\setminus\{p_N\}$, $S$ is given by 
 \begin{equation}\label{af2}
\varphi(u,v)= \frac{1}{\sqrt{1+\frac{\kappa}{4} \rho^2} - \sin(\sigma h)}\left(\frac{\sqrt{\kappa}}{2} \rho \cos (\theta + \sigma h), \frac{\sqrt{\kappa}}{2} \rho \sin (\theta + \sigma h), \cos(\sigma h)\right),
 \end{equation} where $\sigma:=\frac{\kappa}{4\tau}$ and we are denoting $\rho=\rho(u)$, $h=h(u)$. We note that the surface in $\R^3$ given by \eqref{af2} is a rotational surface around the $y_3$-axis, with profile curve in the $(y_1,y_3)$-plane given by
 \begin{equation}\label{af3}
 \gamma(u)= \frac{1}{\sqrt{1+\frac{\kappa}{4} \rho(u)^2} - \sin(\sigma h(u))}\left(\frac{\sqrt{\kappa}}{2} \rho(u), \cos(\sigma h(u))\right), \hspace{0.5cm} \sigma:= \frac{\kappa}{4\tau}. \end{equation} 

When the rotational surface $S$ approaches its opposite axis $L^*$ in $\S^3$, the radius $\rho$ of its associated surface $\hat{S}$ in the $\Rk$ model blows up, while in the $(y_1,y_2,y_3)$-coordinates, the profile curve \eqref{af3} converges to the point $(1,0)$.

\def\refname{References}

\vskip 0.2cm

\noindent José A. Gálvez

\noindent Departamento de Geometría y Topología,\\ Universidad de Granada (Spain).

\noindent  e-mail: {\tt jagalvez@ugr.es}

\vskip 0.2cm

\noindent Pablo Mira

\noindent Departamento de Matemática Aplicada y Estadística,\\ Universidad Politécnica de Cartagena (Spain).

\noindent  e-mail: {\tt pablo.mira@upct.es}

\vskip 0.4cm

\noindent Research partially supported by MINECO/FEDER Grant no. MTM2016-80313-P and Programa de Apoyo a la Investigacion,
Fundación Séneca-Agencia de Ciencia y Tecnologia
Region de Murcia, reference 19461/PI/14.

\end{document}